\documentclass[12pt]{amsart}
\usepackage{lineno,hyperref}
\usepackage[margin=2cm]{geometry}
\usepackage{amsmath}
\usepackage{amsthm}
\usepackage{tikz}
\usetikzlibrary{positioning,decorations.pathreplacing,shapes}
\usepackage{algorithm}
\usepackage{algorithmic}

\newtheorem{prop}{Proposition}
\newtheorem{thm}{Theorem}
\newtheorem{rmk}{Remark}
\newtheorem{alg}{Algorithm}

\newcommand{\xx}{{\bf x}}
\newcommand{\vv}{{\bf v}}
\newcommand{\A}{{\bf A}}
\newcommand{\B}{{\bf B}}
\newcommand{\E}{{\bf E}}
\newcommand{\floor}[1]{\lfloor #1 \rfloor}

\begin{document}
	\title[]{ Efficient ensemble stochastic algorithms for agent-based models with spatial predator-prey dynamics }
	
\author[]{Giacomo Albi, Roberto Chignola \and Federica Ferrarese}
	
		\thanks{
			\\
			Giacomo Albi:
			Department of Computer Science, University of Verona, e-mail: giacomo.albi@univr.it.
			\\
			Roberto Chignola:
			Department of Biotechnology, University of Verona, e-mail: roberto.chignola@univr.it.
			\\
			Federica Ferrarese:
			Department of Mathematics, University of Trento, e-mail: federica.ferrarese@unitn.it.
			}

	\begin{abstract}
		Experiments in predator-prey systems show the emergence of long-term cycles. Deterministic model typically fails in capturing these behaviors, which emerge from the microscopic interplay of individual based dynamics and stochastic effects. However, simulating stochastic individual based models can be extremely demanding, especially when the sample size is large. Hence, we propose an alternative simulation approach, whose computation cost is lower than the one of the classic stochastic algorithms. First, we describe the agent-based model with predator-prey dynamics, and its mean-field approximation.Then, we provide a consistency result for the novel stochastic algorithm at the microscopic and mesoscopic scale.  Finally, we perform different numerical experiments in order to test the efficiency of the proposed algorithm, focusing also on the analysis of the different nature of oscillations between mean-field and stochastic simulations.
	\end{abstract}
	
		\maketitle
		
	\section{Introduction}\label{intro}
	The description of how biological populations interact and evolve in both space and time is central to theoretical ecology, \cite{friedman2017ecological, pielou1969introduction, maynard1974models, renshaw1993modelling, murray2007mathematical,black2012}.
	To this aim, several mathematical models have been proposed to approximate the evolution of large ensembles of interacting species, starting from microscopic stochastic agent-based dynamics, passing from mesoscopic equations, up to macroscopic systems, such as reaction-diffusion systems. This hierarchy of descriptions, at different scales, has been used extensively to model interactions among species and their spatial spread beyond theoretical ecology, and it has a profound impact in several fields such as epidemiology, \cite{boscheri2021modeling,webb1981reaction,allen2008asymptotic}, biology, physiology and medicine, \cite{carrillo2018zoology,shigesada1980spatial,roose2007mathematical}, 
	chemical reactions \cite{hellander2007hybrid,duncan2016hybrid}, and in socio-economic systems opinions, markets, veichular traffic and crowds \cite{pareschi2013interacting,cristiani2014multiscale,albi2019vehicular}.
	In the present manuscript, we will focus on the efficient simulation of predator-prey dynamics in population biology, such as \cite{murray2007mathematical,chen2016non,maynard1974models}.
	One striking feature of population biology is the emergence of temporal cycles of species densities, \cite{murray2007mathematical, nisbet2003modelling, berryman2002population}. Various mechanisms can cause populations to oscillate, the most investigated being the cyclic dynamics that arise from the interactions between predators and preys (see ref. \cite{Blasius2020} and references cited therein). 
	Traditionally, these interactions have been studied in the context of deterministic Lotka-Volterra-type models, \cite{murray2007mathematical, nisbet2003modelling, berryman2002population}. In their simplest form, these are coupled deterministic equations that include specific birth, death, competition and predation processes, and they can show limit cycles behaviors for appropriate choice of parameters values. Long-term persistence of predator-prey cycles have been recently observed in a series of well controlled microcosm experiments with freshwater organisms, \cite{Blasius2020}. Planktonic rotifers cultured together with their prey unicellular green algae showed oscillatory behaviors of both prey and predator densities that persisted for up to approximately 50 cycles or $\sim$300 generations, \cite{Blasius2020}. Interestingly, dominant dynamics characterized by coherent oscillations were interrupted by shorter episodes of irregular non-coherent oscillations, \cite{Blasius2020}. The experiments clearly demonstrate that sustained oscillations in population dynamics do arise even in simple well-controlled ecosystems. They strongly indicate that stochastic forces are at work to drive the reversible shift from coherent to non-coherent oscillatory behaviours of observed populations, \cite{Blasius2020}, but their role in driving population dynamics can not be investigated within the context of deterministic predator-prey models.
	
	This novel experimental evidence, therefore, calls for new theoretical explanations, and we note that these might be found within the theoretical framework developed by McKane and Newman (see \cite{mckane2004stochastic, mckane2005predator}). They studied individual based predator-prey models described by simple stochastic processes of mortality, reproduction and predation, and showed that large cycles of species densities persisted unless the number of individuals was taken to be strictly infinite, \cite{mckane2004stochastic, mckane2005predator}. Biological cycles have been found to arise from a novel resonance effect that include the statistical fluctuations of a given finite-size population, an effect that fade away when different realization of the model are averaged to recover the mean-field behavior described by the classic deterministic predator-prey equations, \cite{mckane2004stochastic, mckane2005predator}. 
	
	%	The approach by McKane and Newman is promising and might help understanding the emergence of long-lasting cycles as observed in actual experiments. 

	%	\textcolor{blue}{
	%		Traditional deterministic macroscopic models for interacting species in biology are driven typically by reaction-diffusion systems of the following type
	%		\begin{equation}\label{genmod}
	%		\begin{aligned}
	%		%\begin{cases}
	%		{\partial_t {\bf u}} &= \mathcal R({\bf u})+\mathcal D({\bf u}) \Delta {\bf u},
	%		\end{aligned}
	%		\end{equation}
	%		where for two-interacting species dynamics we have ${\bf u}(x,t)= [f(x,t),g(x,t)]^\top$ representing the vector of densities of agents located at $x$ at time $t$ respectively for each populations; $\mathcal R(\cdot)$ models the interactions between species taking in account competitive or cooperative behaviors, as well as predator-prey dynamics or other specific effects; diffusion term $\mathcal D({\bf u}) \Delta {\bf u}$ describes species migration through the landscape.	
	%		Model \eqref{genmod} describes a wide variety of systems, such as the interactions between host and pathogens or immune and cancer cell dynamics (see e.g. refs. \cite{murray2001mathematical, Siekmann2009, Fenton2010, AdamJohnA1997ASoM}, and references cited therein).
	%	}
	Finding the solution of stochastic individual-based models, however, can easily become computationally challenging, above all when the final goal is to investigate the dynamics of large, albeit finite, population sizes over hundreds of species generations. Moreover, depending on the context of application, ad-hoc methods are of paramount importance to capture the essential features of phenomena at various scale, see for example \cite{flegg2015convergence,hellander2017mesoscopic,engblom2009simulation,spill2015hybrid}.
	
	In the present work, to cope with the increasing complexity of the model we rely on Monte Carlo methods, introducing a novel stochastic algorithm for the simulation of the population dynamics with individual based models. We investigate the efficiency of the new method on spatial predator-prey models and we compare the computational costs, also by exploring in part the space of parameters, to that of different stochastic algorithms such as the classic Monte Carlo, the direct method or the $\tau-$leaping approach considered as benchmarks \cite{marchetti2017simulation,gillespie1976general,caflisch1998monte}.

	The fundamental idea behind our work is that with classic approaches the whole population sample must be reconstructed any time interaction between individuals occurs with the updated number of individuals.  However, we show that there exists a maximum number of interactions that can take place before updating the sample, and this allows us to reduce the total number of time steps required for the whole simulation. As the consequence, the new algorithm has a computational cost that is up to 2 orders of magnitude less than those of the classic stochastic algorithms considered for a wide range of parameter values. The simulations obtained with the new algorithm converge to the mean-field solutions for increasing sample size $N$ with error $\propto 1/\sqrt N$, as also observed for classic algorithms, while individual realizations for finite population sizes oscillates as observed by McKane and Newman.
	
	In Section \ref{section2} we present an individual-based stochastic model where predators compete with preys and migration of individuals is allowed. We derive the master equation and the mean-field equations which are obtained when the number of individuals is taken to be strictly infinite.  In Section \ref{section3} we present a novel ensemble stochastic algorithm and we prove its consistency with the classical formulation. In Section \ref{section4} we show different numerical experiments aimed at testing the new algorithm and at comparing it with the classical approaches. We also partly investigate the stochastic persistency of the oscillatory behavior shown by both predator and prey dynamics. Overall our approach can be exploited to efficiently simulate stochastic agent-based models and thus to explore the emergence of long-lasting persistent resonant effects in population dynamics. In  \ref{app:homo} we show how to reduce the model to the homogeneous case, assuming that individuals are not allowed to migrate. In  \ref{sec:algH} we briefly report some of the classical exact and approximated stochastic algorithms and we describe them in the homogeneous case.

	\section{Agent-based models with predator-prey interactions }\label{section2}
	We consider an agent based stochastic model of predation between two species with migration. We will show that for considerably large number of individuals the evolution of this model reduces to  a system of partial differential equations.
	
	\subsection{Spatially heterogeneous predator-prey model}\label{sec:NH}
We consider a stochastic process describing the evolution of individuals distributed in the spatial domain $\Omega = [0,L]$ with $L>0$, divided in $M_c$ spatial cells.  We assume that every cell $C_{\ell}$, with $\ell=1,\ldots,M_c$, has $N_c$ components, where each component is realized in one of the following states: predator ($s_{A,\ell}$), prey ($s_{B,\ell}$), or  empty $(s_{E,\ell})$.
		To describe the stochastic dynamics, we consider cell $C_\ell$ and its nearest cells $C_{\ell\pm1}$, for a fixed time $t>0$ we define different interaction events. 
		\begin{itemize}	
			\item {\em Competition \& Birth events.} We sample a component in cell $C_{\ell}$ and with probability $q_1\in[0,1]$ we let it interacts with another component, randomly chosen among the other $N_c-1$ components in cell $C_{\ell}$, according to the following birth and competition rules with rates $b^r,p^r_1,p^r_2$
			\begin{equation} \label{eq:5a} \begin{split}
					&s_{B,\ell}s_{E,\ell} \xrightarrow{b^r} s_{B,\ell}s_{B,\ell},\quad s_{A,\ell}s_{B,\ell} \xrightarrow{p^r_{1}} s_{A,\ell}s_{A,\ell}, \quad  s_{A,\ell}s_{B,\ell} \xrightarrow{p^r_{2}} s_{A,\ell}s_{E\ell},
			\end{split} \end{equation}	
			where $b^r,p^r_1,p^r_2 >0$ are constant parameters.
			We assume that if interactions occur among two empty components then no change in the populations sizes is accounted. 
			\item {\em Migration event.} We sample one component in cell $C_\ell$ and we assume that  with probability $q_2\in[0,1]$ such that $q_1+q_2 \leq 1$ it interacts with another component, randomly chosen among the $N_c$ positions in one of the nearest cells $C_{\ell \pm 1}$. Changes in the state happen according to migration rates $m^r_1,m^r_2$
			\begin{equation} \label{eq:5b} \begin{split}
					&s_{A,\ell}s_{E,\ell\pm 1} \xrightarrow{m^r_1 } s_{E,\ell} s_{A,\ell\pm 1},\quad s_{B,\ell}s_{E,\ell\pm 1}, \xrightarrow{m^r_2} s_{E,\ell} s_{B,\ell\pm 1},\\
					& s_{A,\ell\pm 1}s_{E,\ell} \xrightarrow{m^r_1} s_{E,\ell\pm 1} s_{A,\ell},\quad s_{B,\ell \pm 1}s_{E,\ell} \xrightarrow{m^r_2} s_{E,\ell \pm 1 } s_{B,\ell},
			\end{split} \end{equation}
			where $m^r_1,m^r_2>0$.
			If interactions occur among predator/prey and empty components or among two empty components then no changes in the populations sizes happen. 
			\item {\em Death event.} We sample a component in cell $C_\ell$, and we assume that with probability $(1-q_1-q_2)$ it changes according to death rates $d^r_1, d^r_2$
			\begin{equation} \label{eq:5c} \begin{split}
					s_{A,\ell}\xrightarrow{d^r_1}s_{E,\ell}, \quad s_{B,\ell}\xrightarrow{d^r_2} s_{E,\ell},
			\end{split} \end{equation}
			where $d_1^r, d_2^r >0$. 
			We suppose that if the selected component is empty then no change is accounted.
		\end{itemize}
		We assume that at most one birth/competition, one migration and one death event can occur at each time step.  
		In order to introduce a master equation for such process, we consider the vector state $\xx=(\A,\B,\E)\in \mathbb{R}^{3M_c}$, where $\A=(A_1,A_2,\ldots,A_{M_c})$, $\B=(B_1,B_2,\ldots,B_{M_c})$, $\E=(E_1,E_2,\ldots,E_{M_c})$ account respectively the number of predators, prey and empty spaces in each cell.
		We first define the stoichiometry matrices $\hat V$, associated to single competition/birth and death events, and  $\hat{V}_M$, associated to migration events
		\begin{equation*}
			\hat{V}=\begin{bmatrix}
				0 &  1 & -1 \\
				1 & -1 & 0 \\
				0 & -1 & 1 \\
				-1 & 0 & 1 \\
				0 & -1 & 1
			\end{bmatrix}, \quad
			\hat{V}_M=\begin{bmatrix}
				1 &  0 & -1 \\
				-1 & 0 & 1 \\
				0 & 1 & -1 \\
				0 & -1 & 1 \\
			\end{bmatrix}.
		\end{equation*}
		Hence, the stoichiometry matrix of the full process over the lattice of $M_c$ cells is defined as follows,
		\begin{equation}\label{eq:stoicNH}
			\tilde{V}= \begin{bmatrix}
				\hat{V} \otimes I \\
				\hat{V}_M \otimes M_- \\
				\hat{V}_M \otimes M_+
			\end{bmatrix},	
		\end{equation}
		where $\otimes$ denotes the Kronecker product, and $I$ is the identity matrix of size $M_c\times M_c$, and $M_-,M_+$ are squared	are square matrices of size $M_c\times M_c$ defined as
		\begin{equation*}
			M_-= \begin{bmatrix}
				0 & 0 & \ldots & 0 & 0 \\
				1 & -1 & 0 &  & \vdots \\
				0 & 1 & \ddots & \ddots & 0 \\
				\vdots & \ddots & \ddots & -1 & 0 \\
				0 & \ldots & 0 & 1 & -1
			\end{bmatrix},\quad
			M_+= \begin{bmatrix}
				-1 & 1 & 0 & \ldots & 0 \\
				0 & -1 & 1 & \ddots & \vdots \\
				0 & 0 & \ddots & \ddots & 0 \\
				\vdots & \ddots & \ddots & -1 & 0 \\
				0 & \ldots & 0 & 0 & 0
			\end{bmatrix}.
		\end{equation*}
		Here, each row of the stoichiometry matrix \eqref{eq:stoicNH} represents the changes in the populations sizes due to the occurrence of the events described in \eqref{eq:5a}-\eqref{eq:5b}-\eqref{eq:5c} and each column represents the predators, preys and empty components in each cell.	
			Then, for any $\ell = 1,\ldots, M_c$, we write the associated transition rates as follows		
			\begin{equation}\label{eq:transition}
				\begin{split}
					& \pi_{\vv_{\ell_{0}}}(\xx)= 2 b^r q_1 \frac{B_\ell}{N_c} \frac{E_\ell}{N_c-1}, \qquad \pi_{\vv_{\ell_{1}}}(\xx) = 2 p_1^r q_1 \frac{A_\ell}{N_c} \frac{B_\ell}{N_c-1},\cr &\pi_{\vv_{\ell_{2}}}(\xx)= 2 p_2^r q_1 \frac{A_\ell}{N_c} \frac{B_\ell}{N_c-1}, \qquad \pi_{\vv_{\ell_{3}}}(\xx) =  d_1^r (1-q_1-q_2) \frac{A_\ell}{N_c},\cr 
					&\pi_{\vv_{\ell_{4}}}(\xx)=  d_2^r (1-q_1-q_2) \frac{B_\ell}{N_c},\quad \pi_{\vv_{\ell_{5}}}(\xx) = m_1^r q_2 \frac{A_{\ell}}{N_c} \frac{E_{\ell-1}}{N_c},\cr
					&\pi_{\vv_{\ell_{6}}}(\xx) = m_1^r q_2 \frac{A_{\ell-1}}{N_c} \frac{E_{\ell}}{N_c},\quad \qquad \pi_{\vv_{\ell_{7}}}(\xx)= m_2^r q_2 \frac{B_{\ell}}{N_c} \frac{E_{\ell-1}}{N_c},\cr &\pi_{\vv_{\ell_{8}}}(\xx)= m_2^r q_2 \frac{B_{\ell-1}}{N_c} \frac{E_{\ell}}{N_c},\quad \qquad \pi_{\vv_{\ell_{9}}}(\xx) = m_1^r q_2 \frac{A_{\ell}}{N_c} \frac{E_{\ell+1}}{N_c},\cr
					& \pi_{\vv_{\ell_{10}}}(\xx) = m_1^r q_2 \frac{A_{\ell+1}}{N_c} \frac{E_{\ell}}{N_c},\quad \qquad\pi_{\vv_{\ell_{11}}}(\xx) = m_2^r q_2 \frac{B_{\ell}}{N_c} \frac{E_{\ell+1}}{N_c},\cr &\pi_{\vv_{\ell_{12}}}(\xx) = m_2^r q_2 \frac{B_{\ell+1}}{N_c} \frac{E_{\ell}}{N_c},
				\end{split}
			\end{equation}
			where  we define the operators $\pi_{\vv_{\ell_{j}}}(\xx)=\pi(\xx+\vv_{\ell_j}|\xx)$ for any $j= 0,\ldots, M$. Here, $M+1=13$ is the total number of events described in \eqref{eq:5a}-\eqref{eq:5b}-\eqref{eq:5c} and $\vv_{\ell_j}$ is the $\ell_j$-th row of the stoichiometry matrix $\tilde{V}$ defined  in \eqref{eq:stoicNH}, $\ell_j=\ell+jM_c$ for $j=0,\ldots,M$.
		We assume that migrations events are not allowed at the boundaries. Hence we introduce $A_{0},E_0,B_0,A_{M_c+1},E_{M_c+1},B_{M_c+1}$ to be equal to zero, in order to properly define the transition rates associated to the boundary cells. 
	The density $P(\xx,t)$ describing the probability of the state $\xx$ evolves according to the master equation as follows
	\begin{equation} \label{eq:3b}
		\frac{dP(\xx,t)}{dt}= \sum_{\ell_j\in \mathcal{J}}\Biggl[ \pi(\xx|{\bf x}-{\bf v}_{\ell_j}) P({\bf x}-{\bf v}_{\ell_j},t)-   \pi({\bf x}+{\bf v}_{\ell_j}|\xx)P({\bf x},t)\Biggr],
	\end{equation}  
	where \begin{equation}\label{eq: setJ}
		\mathcal{J}=\{\ell+jM_c|j=0,\ldots,M \text{ and  }\ell=1,\ldots,M_c\}.
	\end{equation}

		\begin{rmk}~ \begin{itemize}
				\item We assume that each spatial cell contains at most $N_c$ components that can be either preys or predators or empty spaces. Hence, even if predators and preys are restricted to move only to empty neighbor sites, the model can include regions occupied both by predators and preys. 
				In the case of a lattice consisting of a single spatial cell, we refer to  spatial homogeneous model, where migration events are neglected and maximal density per cell is $N_c = N$.  We report in \ref{app:homo} the corresponding agent-based dynamics, and the associated master equation \eqref{eq:3b}.
				\item 	We observe that the vector state $\xx$ can be dimensionaly reduced if we assume to remove the empty components, considering the relation $\E=N_c-\A-\B$. However, we prefer to give a more general description, to allow for further generalization, such as models that may include empty sites with different nature.				
			\end{itemize}
		\end{rmk}
	
	\subsection{Mean-field approximation}
In order to recover the mean-field behavior of the stochastic process,
		we introduce the empirical densities $f_\ell^{N_c}(t),g_\ell^{N_c}(t)$, for $\ell=1,\ldots, M_c$, as the averaged quantities
	\begin{equation}\label{eq:empirical}
		\begin{aligned}
			f_\ell^{N_c}(t)=\frac{\left  \langle A_\ell \right \rangle}{N_c} = \frac{1}{N_c} \sum_{A_\ell=0}^{N_c}\sum_{B_\ell=0}^{N_c}\sum_{E_\ell=0}^{N_c} A_\ell P(\xx,t),\cr
			g_\ell^{N_c}(t)=\frac{\left  \langle B_\ell \right \rangle}{N_c} = \frac{1}{N_c} \sum_{A_\ell=0}^{N_c}\sum_{B_\ell=0}^{N_c}\sum_{E_\ell=0}^{N_c} B_\ell P(\xx,t),
		\end{aligned}
	\end{equation}
	where $\langle\cdot\rangle$ denotes the expected value,
and the empirical density for the empty component can be recovered in each cell computing $1-f_\ell^{N_c}(t)-g_\ell^{N_c}(t)$.
	
	Hence, multiplying the master equation by $A_\ell$ and by $B_\ell$, and summing over all the values of $A_\ell$, $B_\ell$ and $E_\ell$ for any $\ell=1\,\ldots,M_c$ leads to the following Proposition. 
	\begin{prop}\label{prop:1b} By standard assumptions of the mean-field limit we assume that  for any $\ell=1\,\ldots,M_c$, $\left\langle A_\ell B_\ell \right\rangle = \left\langle A_\ell\right\rangle \left\langle B_\ell\right\rangle $, $\left\langle A_\ell^2\right\rangle = \left\langle A_\ell\right\rangle^2$, $\left\langle B_\ell^2\right\rangle = \left\langle B_\ell \right\rangle^2$, $\left\langle A_{\ell \pm 1} B_\ell\right\rangle = \left\langle A_{\ell \pm 1} \right\rangle \left\langle B_\ell \right\rangle$ and $\left\langle  A_\ell B_{\ell \pm 1}\right\rangle = \left\langle A_\ell \right\rangle \left\langle B_{\ell \pm 1} \right\rangle$ then the time evolution of the empirical population densities $f_\ell^{N_c}$, $g_\ell^{N_c}$ is given by
		\begin{equation}\label{eq:A3b}
			\begin{split}
				\frac{d f_\ell^{N_c}}{d \tau} =& 2\tilde{p^r_1}  \frac{ \langle A_\ell \rangle }{N_c} \frac{\langle B_\ell \rangle }{N_c-1} - \tilde{d^r_1}  \frac{\langle A_\ell \rangle }{N_c}\cr 
				&\qquad\qquad\qquad+ \tilde{m}^r_1\left( \frac{	\Delta_\epsilon\langle A_\ell \rangle }{N_c} +  \frac{\langle A_\ell \rangle }{N_c}  \frac{	\Delta_\epsilon\langle B_\ell \rangle }{N_c}  -\frac{\langle B_\ell \rangle }{N_c}  \frac{	\Delta_\epsilon\langle A_\ell \rangle }{N_c} \right),
				\\
				\frac{d g_\ell^{N_c}}{d \tau}=& r \frac{\langle B_\ell \rangle }{N_c}   \left( 1-\frac{\langle B_\ell \rangle }{q N_c}\right)-\alpha \frac{\langle A_\ell \rangle }{N_c} \frac{\langle B_\ell \rangle }{N_c}    \cr
				&	\qquad\qquad\qquad+ \tilde{m}^r_2\left( \frac{	\Delta_\epsilon\langle B_\ell \rangle }{N_c} +  \frac{\langle B_\ell \rangle }{N_c}  \frac{	\Delta_\epsilon\langle A_\ell \rangle }{N_c}  -\frac{\langle A_\ell \rangle }{N_c}  \frac{	\Delta_\epsilon\langle B_\ell \rangle }{N_c}\right),
			\end{split}
		\end{equation}
		where\[
		\Delta_\epsilon h_\ell = \sum_{s\in \{ \ell-1,\ell+ 1\}} \frac{h_s - h_\ell}{\epsilon^2},
		\]
		for any function $h$ is the discrete Laplace operator, $\epsilon$ is the lattice spacing, $\tau = t/N_c$, the parameters are defined according to the following scaling
		\begin{equation}
			\begin{split}
				&\tilde{b}^r=b^r q_1 \quad \tilde{p_1}^r= p^r_1 q_1 \quad \tilde{p_2}^r= p^r_2 q_1 \quad \tilde{d^r_1}=(1-q_1-q_2) d^r_1 \\
				& \tilde{d_2}^r=(1-q_1-q_2) d^r_2 \quad \tilde{m}^r_1=q_2 \epsilon^{-2} m^r_1 \quad \tilde{m}^r_2=q_2 \epsilon^{-2} m^r_2,
			\end{split}
		\end{equation}
		and 	\[
		r=2\tilde{b}^r-\tilde{d_2}^r,\qquad q=1-\frac{\tilde{d_2}^r}{2\tilde{b}^r},\qquad \alpha = 2(\tilde{p}^r_1+\tilde{p}^r_2+\tilde{b}^r).
		\]
		
	\end{prop}
	
	Finally  the mean-field behavior is recovered for $N_c\gg1$ as the result of the following Theorem.
	\begin{thm}
		Consider the discrete mean-field model \eqref{eq:A3b} for $N_c$ individuals. Then taking the limit for $N_c\to \infty$, $\epsilon \to 0$,  the mean-field equations for the densities $f(x,t)$ and $g(x,t)$,
		\begin{equation}\label{eq:6}
			\begin{aligned}
				&{\partial_\tau f}=2 \tilde{p}_1 f g -\tilde{d}_1 f + \tilde{m}_1 (f \Delta g+ (1-g) \Delta f),\\
				&{\partial_\tau g}=r g \left( 1-\frac{g}{q} \right) -\alpha f g +\tilde{m}_2 (g \Delta f +(1-f) \Delta g),
			\end{aligned}
		\end{equation}
		where each entry of the vectors $f({\bf x },t)$ and $g({\bf x},t)$  represents the predators and preys densities in each cell $C_{\ell}$, $\ell=1,\ldots, M_c$ and \[
		\Delta h(x,t) = \lim_{\epsilon \to 0} {\Delta_\epsilon}h
		\] for any function h.
	\end{thm}	
	We refer to \cite{mckane2004stochastic, mckane2005predator, van1992stochastic}, for a detailed proof of the mean-field limit.

		\begin{rmk}~
			Note that under condition $(1-q_1-q_2) d^r_2<(2q_1 b^r)$ model \eqref{eq:6} corresponds to the well known Lotka-Volterra equations with logistic growth term and diffusion, \cite{murray2007mathematical}. 
	\end{rmk}
	\section{Efficient ensemble stochastic algorithms}\label{section3}
	The dynamics described by model \eqref{eq:6} can be properly simulated with stochastic algorithms. Exact stochastic algorithms predict which is the next firing event and at which time it will fire, as for example in \cite{gillespie1976general}.  However, a direct implementation is often prohibitively expensive since the final goal is the simulation of a large stochastic process. On the other hand, classic approximated algorithms, such as the Monte Carlo algorithm and the $\tau$-leaping method, can speed up the simulations but their efficiency and accuracy are strongly related to the choice of the time step that can vary in time or be inversely proportional to the sample size. The main idea of the procedure is to fix a constant time step, and to allow multiple events to happen at the same time. Hence, we describe the associated stochastic process for the spatial heterogeneous predator-prey dynamics and show its consistency with the previous description. Second we describe the ensemble stochastic algorithm in the spatial homogeneous and heterogeneous setting.
	\subsection{Predator-prey model with ensemble interactions}\label{section3.2}
	At each time step $t\geq0$ consider two fractions $q_1,q_2\in[0,1]$ such that $q_1+q_2 \leq 1$. Assume that the following events can occur.
	\begin{itemize}
		\item {\em Competition \& Birth}. Sample $\floor{q_1N_c}$ components in each cell $C_\ell$, $\ell=1,\ldots,M_c$ assuming that each one interacts with another component, randomly chosen among the $N_c$ positions in cell $C_\ell$ without repetition, according to the rules described in \eqref{eq:5a}. 
		\item {\em Migration}. Sample other $\floor{q_2 N_c}$ components in any cell $C_\ell$, different from the ones previously selected, and let each one to interact with another state, randomly chosen in cell $C_{\ell \pm 1}$ without repetition, according to the migration rules defined in \eqref{eq:5b}. 
		\item {\em Death}. The remaining $N_c -\floor{q_1 N_c} -\floor{q_2 N_c}$ components in any cell $C_\ell$ change according to death rules defined in equation \eqref{eq:5c}. 
	\end{itemize}
	Hence, in any infinitesimal time interval $[t,t+d\tau_{N_c}]$, where $d\tau_{N_c} := dt N_c$, a transition can occur from the state with $\xx=(\A,\B,\E)$ individuals to the states with $\xx+k\vv_{\ell_j}$  individuals for any $k=1,\ldots,N_c$, where $\vv_{\ell_j}$ is the $\ell_j$-th row of the stoichiometry matrix $\tilde{V}$ defined in \eqref{eq:stoicNH} for any $\ell=1,\ldots, M_c$, $j=0,\ldots,M$. For $k = 1,\ldots N$, the transition rates write as follows
		\begin{equation}\label{eq:A6}
			\pi(\xx+k\vv_{\ell_{j}}|\xx) = \prod_{i=1}^{k} \pi^i_{\vv_{\ell_{j}}}(\xx)
		\end{equation} where
		\begin{equation}\label{eq:A0a}
			\begin{split}
				&\pi_{\vv_{\ell_{0}}}^i(\xx) = 2 
				b^r q_1 \frac{ \tilde{B}_\ell^i}{ \tilde{N_c}^i}\frac{ \tilde{E}_\ell^i}{\tilde{N}_c^i},~~\quad \qquad \pi_{\vv_{\ell_{1}}}^i(\xx) = 2 
				p_1^r q_1 \frac{ \tilde{A}_\ell^i}{ \tilde{N}_c^i}\frac{ \tilde{B}_\ell^i}{\tilde{N}_c^i},\\
				&\pi_{\vv_{\ell_{2}}}^i(\xx) = 2 
				p_2^r q_1  \frac{\tilde{A}_\ell^i}{\tilde{N}_c^i}\frac{ \tilde{B}_\ell^i}{\tilde{N}_c^i},~~\quad \qquad
				\pi_{\vv_{\ell_{3}}}^i(\xx) =  
				d_1^r (1-q_1-q_2) \frac{\tilde{A}_\ell^i}{\tilde{N}_c^i},\\
				&\pi_{\vv_{\ell_{4}}}^i(\xx) =  
				d_2^r (1-q_1-q_2) \frac{ \tilde{B}_\ell^i}{\tilde{N}_c^i}, \quad
				\pi_{\vv_{\ell_{5}}}^i(\xx) =  
				m_1^r q_2 \frac{ \tilde{A}_{\ell}^i}{\tilde{N}_c^i}\frac{ \tilde{E}_{\ell-1}^i}{\tilde{N}_c^i},\\
				&\pi_{\vv_{\ell_{6}}}^i(\xx)=  
				m_1^r q_2 \frac{ \tilde{A}_{\ell-1}^i}{\tilde{N}_c^i}\frac{ \tilde{E}_{\ell}^i}{\tilde{N}_c^i},\qquad \quad
				\pi_{\vv_{\ell_{7}}}^i(\xx) =  
				m_2^r q_2 \frac{ \tilde{B}_{\ell}^i}{\tilde{N}_c^i}\frac{ \tilde{E}_{\ell-1}^i}{\tilde{N}_c^i},\\
				&\pi_{\vv_{\ell_{8}}}^i(\xx) =  
				m_2^r q_2 \frac{ \tilde{B}_{\ell-1}^i}{\tilde{N}_c^i}\frac{ \tilde{E}_{\ell}^i}{\tilde{N}_c^i},\qquad \quad
				\pi_{\vv_{\ell_{9}}}^i(\xx) =  
				m_1^r q_2 \frac{ \tilde{A}_{\ell}^i}{\tilde{N}_c^i}\frac{ \tilde{E}_{\ell+1}^i}{\tilde{N}_c^i},\\
				&\pi_{\vv_{\ell_{10}}}^i(\xx) =  
				m_1^r q_2 \frac{ \tilde{A}_{\ell+1}^i}{\tilde{N}_c^i}\frac{ \tilde{E}_{\ell}^i}{\tilde{N}_c^i},\qquad \quad
				\pi_{\vv_{\ell_{11}}}^i(\xx) =  
				m_2^r q_2 \frac{ \tilde{B}_{\ell}^i}{\tilde{N}_c^i}\frac{ \tilde{E}_{\ell+1}^i}{\tilde{N}_c^i},\\
				&\pi_{\vv_{\ell_{12}}}^i(\xx) =  
				m_2^r q_2 \frac{ \tilde{B}_{\ell+1}^i}{\tilde{N}_c^i}\frac{ \tilde{E}_{\ell}^i}{\tilde{N}_c^i}.
			\end{split}
		\end{equation}
		Here we define the operators  $\pi_{\vv_{\ell_{j}}}^i(\xx) = \pi(\xx+i\vv_{\ell_{j}}|\xx+(i-1)\vv_{\ell_{j}})$ for any $i=1,\ldots,N_c$, we use the following notation $\tilde{Y}^i=Y-i+1$ for any state $Y$  in the set $\{A_\ell, B_\ell,E_\ell,A_{\ell\pm 1}, B_{\ell\pm 1},E_{\ell\pm 1},N_c\}$ and since migration outside the boundaries is not allowed, we introduce $A_{0},E_0,B_0,A_{M_c+1},E_{M_c+1},B_{M_c+1}$ to be equal to zero.	
		The master equation associated to this process is derived in the following Proposition.
		\begin{prop}[Consistency]\label{prop:2} 
			The master equation can be written as 
			\begin{equation} \label{eq:master_NH}
				\frac{dP(\xx,s)}{ds}= \sum_{\ell_j\in \mathcal{J}}\Biggl[ \pi(\xx|{\bf x}-{\bf v}_{\ell_j}) P({\bf x}-{\bf v}_{\ell_j},s)-   \pi({\bf x}+{\bf v}_{\ell_j}|\xx)P({\bf x},s)\Biggr],
			\end{equation}
			where $s = t N_c$, the set $\mathcal{J}$ is defined as in \eqref{eq: setJ} and $|\mathcal{J}|=M_c(M+1)$ is the number of events that can occur.
		\end{prop} 
		\begin{proof}
			The probability to be in the state with $\xx$ individuals at time $t+\tau_{N_c}$ is given by two contributions:
			\begin{itemize}
				\item	the probability of staying in the state $\xx$ at time $t+\tau_{N_c}$ that is
				\begin{equation}\label{eq:proof_1}
					1- \sum_{\ell_j\in\mathcal{J}} \sum_{k=1}^{N_c} \pi(\xx+k\vv_{\ell_j}|\xx)d\tau_{N_c}^k;
				\end{equation}
				\item  the probability of moving from the state with $\xx$ individuals to the state with $\xx+k\vv_{\ell_j}$ individuals in the time interval $[t,t+\tau_{N_c}]$, for any $k=1,\ldots,N_c$ and $\ell_j\in\mathcal J$ that is  
				\begin{equation}\label{eq:proof_2}
					\sum_{\ell_j\in\mathcal{J}} \sum_{k=1}^{N_c} \pi(\xx|\xx-k\vv_{\ell_j})d\tau_{N_c}^k.
				\end{equation}
			\end{itemize}
			Indeed, recall that for any $i=1,\ldots,N_c$ a transition from the state with $\xx+i\vv_{\ell_j}$ individuals to the state with $\xx+(i-1)\vv_{\ell_j}$ individuals in the infinitesimal time interval $[t,t+\tau_{N_c}]$ occurs with probability $\pi_{\ell_j}^i(\xx) d\tau_{N_c}$
			and hence the probability of moving from the state with $\xx$ to the state with $\xx+k\vv_{\ell_j}$ individuals, for any $k=1,\ldots,N_c$, is 
			\[
			\prod_{i=1}^{k} \left\lbrace \pi_{\vv_{\ell_j}}^i(\xx) d\tau_{N_c}\right\rbrace = \pi(\xx+k\vv_{\ell_j}|\xx) d\tau_{N_c}^k,
			\]
			that is exactly the expression that appears in equation \eqref{eq:proof_1}. Similarly we can derive equation \eqref{eq:proof_2}. Hence,
			\begin{equation}\label{eq:proof_3}
				\begin{split}
					P(\xx,t+d\tau_{N_c}) =& \sum_{\ell_j\in \mathcal{J}}\sum_{k=1}^{N_c} \Big[ P(\xx-k\vv_{\ell_j},t) \pi(\xx|\xx-k\vv_{\ell_j}) d\tau_{N_c}^k \Big] \\ &+P(\xx,t) \Biggl ( 1-\sum_{\ell_j\in \mathcal{J}}\sum_{k=1}^{N_c} \Big[  \pi(\xx+k\vv_j|\xx) d\tau_{N_c}^k \Big]\Biggr).
			\end{split}\end{equation}
			Rewrite explicitly the term for $k=1$ in equation \eqref{eq:proof_3} to have
			\begin{equation}\label{eq:proof_4}
				\begin{split}
					&P(\xx,t+d\tau_{N_c}) - P(\xx,t) =  d\tau_{N_c}\sum_{\ell_j\in \mathcal{J}} \Big[ P(\xx-\vv_j,t) \pi(\xx|\xx-\vv_j) - P(\xx,t) \pi(\xx+\vv_j|\xx)  \Big]\\
					&+ d\tau_{N_c}^k\sum_{\ell_j\in \mathcal{J}}\sum_{k=2}^{N_c} \Big[ P(\xx-k\vv_j,t) \pi(\xx|\xx-k\vv_j)  - P(\xx,t) \pi(\xx+k\vv_j|\xx)  \Big].
				\end{split}
			\end{equation}
			Dividing both sides of equation \eqref{eq:proof_4} by $d\tau_{N_c}$ and letting $d\tau_{N_c}\to 0$ we obtain the consistency with the master equation \eqref{eq:3b}, where the time variable is scaled by a factor $N_c$, i.e. $s= t N_c$.
		\end{proof}
		\begin{rmk}
			We observe that thanks to the consistency result of Proposition \ref{prop:2}, the corresponding mean-field approximation is equivalent to \eqref{eq:6}. According to \eqref{eq:6} the time scale of the mean-field dynamics is such that $$\tau = s /{N_c} = N_c t/N_c = t,$$ 
			then the time scales of the ensemble agent-based dynamics and of the mean-field dynamics are equivalent.
		\end{rmk}
		
	\subsection{Efficient Monte-Carlo methods}
 The idea of the efficient Monte Carlo algorithm is to allow multiple events to occur at the same time step that is supposed to be fixed and constant for any choice of $N_c$. Therefore, as we will see in the numerical experiments, its accuracy is comparable with the one of classic approximated algorithms but its efficiency is improved. In the following we will focus first on the spatial homogeneous case, where a single cell is accounted ($M_c=1$) and $N_c = N$, second we will consider the full dynamics with spatial interaction. We refer to \ref{app:homo} for further details on the homogeneous case.
	
	\paragraph{Spatially homogeneous case} 
Divide a priori the time interval considering a constant time step $\tau$ that is independent of the sample size. Introduce a parameter $\mu \in [0,1]$ and at each time $t$ assume that simultaneously $\floor{\mu N}$ individuals interact two by two. If the two individuals are
	\begin{itemize}
		\item a prey and an empty space: with probability $b^r \tau$ a new prey born and occupies the empty space;
		\item a predator and a prey: with probability $p_1^r \tau$ then a new predator born and with probability $p_2^r \tau$ the prey dies and a new empty space is added to the system.
	\end{itemize}
	Assume that the remaining $N-\floor{\mu N}$ individuals are subjected to death events that happen with probability $d_1^r\tau$ in the predators population and with probability $d_2^r\tau$ in the preys one.
	Algorithm \ref{alg_MC_new} defines the details of the efficient Monte Carlo algorithm. 
	\begin{alg}[Efficient ensemble stochastic algorithm - homogeneous]~ \label{alg_MC_new}
		\begin{enumerate}
			\item[\texttt 1.] Define the sample, the initial time $t=0$, the final time $T$, the time step $\tau$ and a parameter $\mu \in [0,1]$ .   
			\item[\texttt 2.] \texttt{while} $t<T$ 
			\begin{enumerate}
				\item Birth and competition events happen between $\floor{\mu N }$ individuals selected two by two \[
				s_Bs_E\xrightarrow{b}s_Bs_B,\quad s_As_B\xrightarrow{p_1} s_As_A, \quad s_As_B \xrightarrow{p_2}s_As_E,
				\]
				according to probability $b=b^r\tau$, $p_1^r\tau$ and $p_2^r\tau$, respectively.
				\item The remaining $N-\floor{\mu N }$ individuals are subjected to death events \[
				s_A\xrightarrow{d_1} s_E, \quad s_B \xrightarrow{d_2}s_E,
				\]
				according to probability  $d_1=d_1^r\tau$ and $d_2=d_2^r\tau$, respectively.
				\item Update the sample.
				\item Set $t \leftarrow t+\tau$.
			\end{enumerate}
			\texttt{repeat}
		\end{enumerate}
	\end{alg}
	\paragraph{Spatially heterogeneous case}\label{sec:algNH}
We extend Algorithm \ref{alg_MC_new} to the spatially heterogeneous case. Hence, the sample is divided in cells and populations are subjected also to migration events. Migrations can occur either between cell $C_\ell$ and cell $C_{\ell+1}$ or between cell $C_\ell$ and cell $C_{\ell-1}$ for any $\ell=1,\ldots, M_c$.  The spatially heterogeneous 
	Algorithm \ref{alg_MC_new} reads as follows
	%	\paragraph{Efficient Monte Carlo algorithm}
		\begin{alg}[Efficient ensemble stochastic algorithm]~ \label{alg_MC_new_heter}	
			\begin{enumerate}
				\item[\texttt 1.] Define the sample, the initial time $t=0$, the final time $T$, the time step $\tau$ and two parameters $q_1,q_2 \in [0,1]$ such that $q_1+q_2 \leq 1$.   
				\item[\texttt 2.] \texttt{while} $t<T$ 
				\begin{enumerate}
					\item In each cell $C_\ell$, $\floor{q_1 N_c}$ individuals interact two by two and are subjected to birth and competition events \eqref{eq:5a} that occur with probabilities $b=b^r\tau$, $p_1=p_1^r\tau$ and $p_2=p_2^r\tau$; 
					\item In each cell $C_\ell$,  $\floor{q_2N_c/2}$ individuals interact with $\floor{q_2N_c/2}$ individuals sampled in cell $C_{\ell+1}$ and other $\floor{q_2N_c/2}$ interact with $\floor{q_2N_c/2}$ individuals sampled in cell $C_{\ell-1}$ according to \eqref{eq:5b}, and migrate with probability $m_1=m_1^r \tau$ for predators and $m_2=m_2^r\tau$ for preys;
					\item In each cell $C_\ell$, the remaining $N_c-\floor{q_1 N_c} -\floor{q_2 N_c}$ individuals are subjected to death events \eqref{eq:5c} according to probabilities $d_1=d_1^r\tau$ and $d_2=d_2^r\tau$;
					\item Update the sample;
					\item Set $t\leftarrow t+\tau$. 
				\end{enumerate}
			\end{enumerate}
	\end{alg}
	%%%%%%%%%%%%%%%%%%%%%%%%%%%%%%%%%%%%%%%%%%%%%%%%%%%%%%%%%%%%%%%%%%%%%%%%%%%%%%%
	%%%%%%%%%%%%%%%%%%%%%%%%%%%%%%%%%%%%%%%%%%%%%%%%%%%%%%%%%%%%%%%%%%%%%%%%%%%%%%%
	%%%%%%%%%%%%%%%%%%%%%%%%%%%%%%%%%%%%%%%%%%%%%%%%%%%%%%%%%%%%%%%%%%%%%%%%%%%%%%%
	\newpage
	\section{Numerical experiments}\label{section4}
	\subsection{Test 1: Validation}\label{section4.1}
	In this section we present a comparison between the numerical solutions of the mean-field equations and the stochastic simulations. In both the homogeneous and heterogeneous case, stochastic simulations have been performed with the efficient version of the Monte Carlo algorithm presented in Section \ref{section3}. In the homogeneous case, the numerical solutions of the mean-field equations are computed using the Matlab function $\mathtt{ode45}$, \cite{MATLAB:2017}, while in the heterogeneous case with a combination of finite difference methods and numerical methods for ODEs assuming periodic boundary conditions, \cite{hundsdorfer2013numerical}.  The parameters choice for all the tests in Section \ref{section4.1} is specified in Table \ref{tab:all_parameters}. The sample size $N$ and the total number of individuals $N_c$ in any cell change in any test and will be defined later. 
	\begin{table}[!h]
		\begin{center}
			\caption{Model parameters for the different scenarios.}\label{tab:all_parameters}
			\begin{tabular}{c|cccccccccc}
				&$b^r$ & $d_1^r$ & $d_2^r$ & $p_1^r$ & $p_2^r$ & $m_1^r$ & $m_2^r$&$\mu$ & $q_1$& $q_2$\\
				\hline
				Homogeneous & 0.1 & 0.1&  0 &0.25&0.05&-&-&0.5&-&-\\
				\hline
				Heterogeneous & 0.1 & 0.1&  0 &0.25&0.05&0.5&0.5&-&0.3&0.3\\
				\hline
			\end{tabular}			
		\end{center}
	\end{table}
	\\Figure \ref{fig:figure1} shows a comparison between stochastic and mean-field solutions in the homogeneous case assuming the sample size to be $N=1000$ and the initial predators and preys density to be $A_0=N/4$ and $B_0=N/2$ respectively.  Note that the solution of the mean-field equations \eqref{eq:6}, without spatial dependence, converges to the stable equilibrium
	\vspace{0.5cm} \begin{equation}\label{equilibrium}
		f^*= \frac{2 \tilde{b}^r \tilde{p}^r_1 -\tilde{b}^r \tilde{d}^r_1 -\tilde{p}^r_1 \tilde{d}^r_2}{2 \tilde{p}^r_1 (\tilde{p}^r_1+\tilde{p}^r_2+\tilde{b}^r)}, \qquad g^*=\frac{\tilde{d}^r_1}{2 \tilde{p}^r_1},
	\end{equation} 
	that for the parameters choice of Table \ref{tab:all_parameters} takes the values $f^*=g^*=0.2$.
	
	Figure \ref{fig:figure2} shows that the error of the efficient Monte Carlo algorithm is proportional to $1/\sqrt{N}$, as the one of classic algorithms in both the predators (left) and preys (right) cases.  The errors $e_f^N$ and $e_g^N$ are defined as 
	\begin{equation} \label{error}
		\begin{split}
			e_f^N=\Vert f(t)-f^N(t)\Vert_{\infty},\qquad e_g^N=\Vert g(t)-g^N(t)\Vert_{\infty},
		\end{split}
	\end{equation}
	where $f(t)$, $g(t)$ denote the mean-field solutions and $f^N(t)$, $g^N(t)$ denote the stochastic solutions obtained with the efficient Monte Carlo algorithm at time $t$ for a certain value of $N$. Recall that $f(t)$, $f^N(t)$ refer to the predators population while $g(t)$, $g^N(t)$ to the preys one. 
	\begin{figure}[H]
		\centering
		\includegraphics[width=0.49\linewidth]{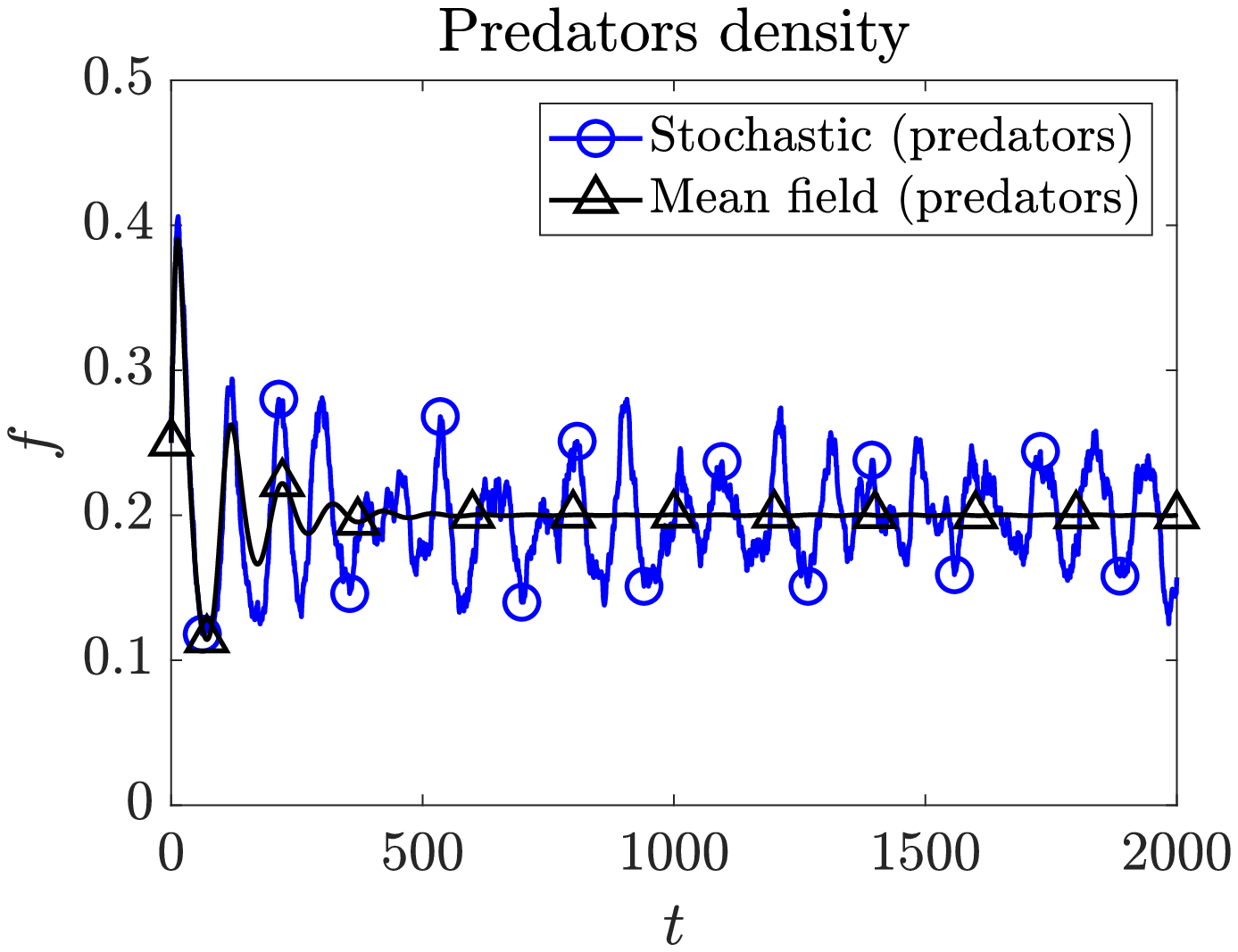}
		\includegraphics[width=0.49\linewidth]{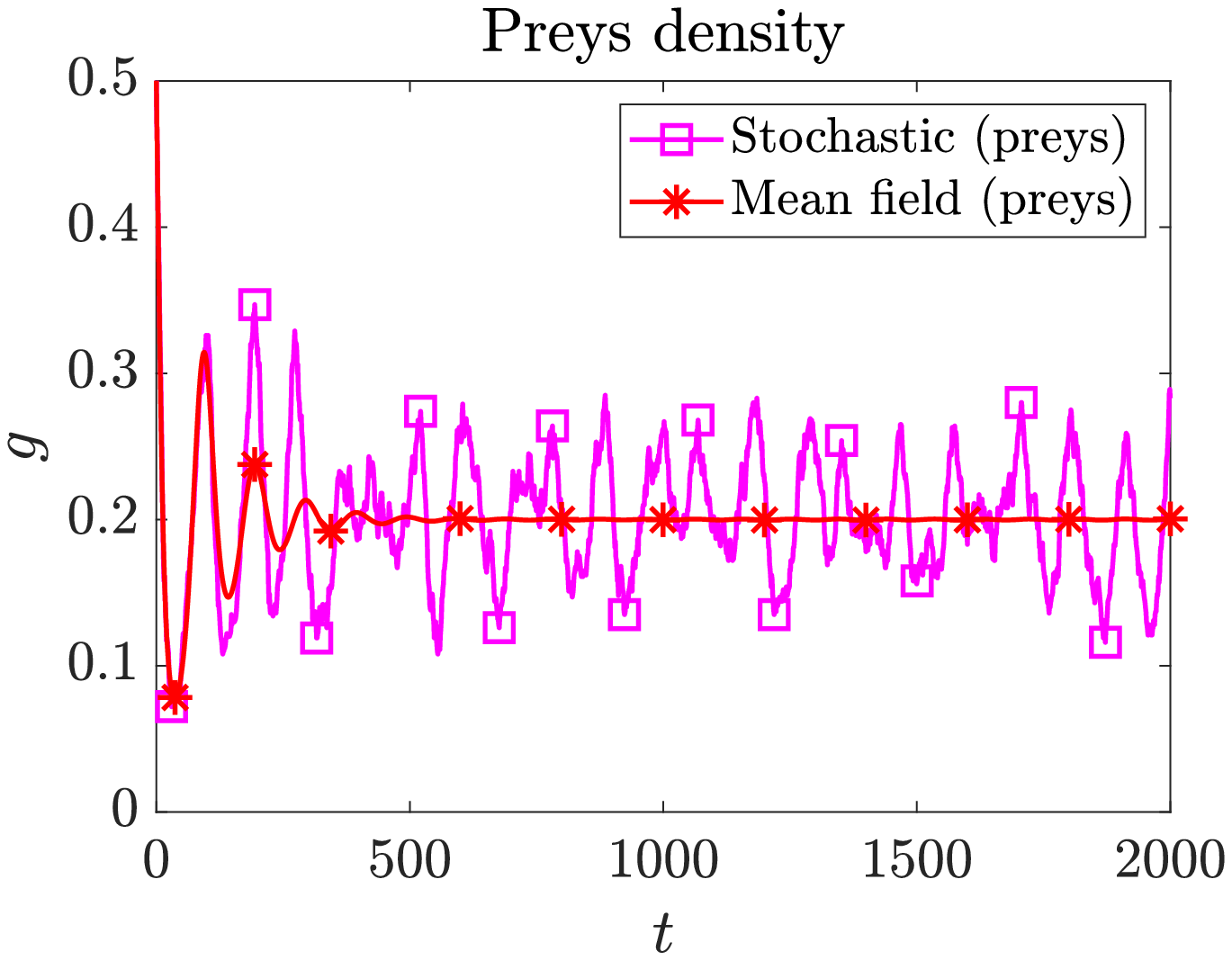}
		\caption{Homogeneous predator-prey model: simulation of the processes described in \eqref{eq:1a}-\eqref{eq:1b} with the efficient Monte Carlo algorithm and solutions of the spatial homogeneous mean-field equations \eqref{eq:6} for $N=1000$. Markers have been added just to indicate different lines.}
		\label{fig:figure1}
	\end{figure}
	\begin{figure}[H]
		\centering
		\includegraphics[width=0.49\linewidth]{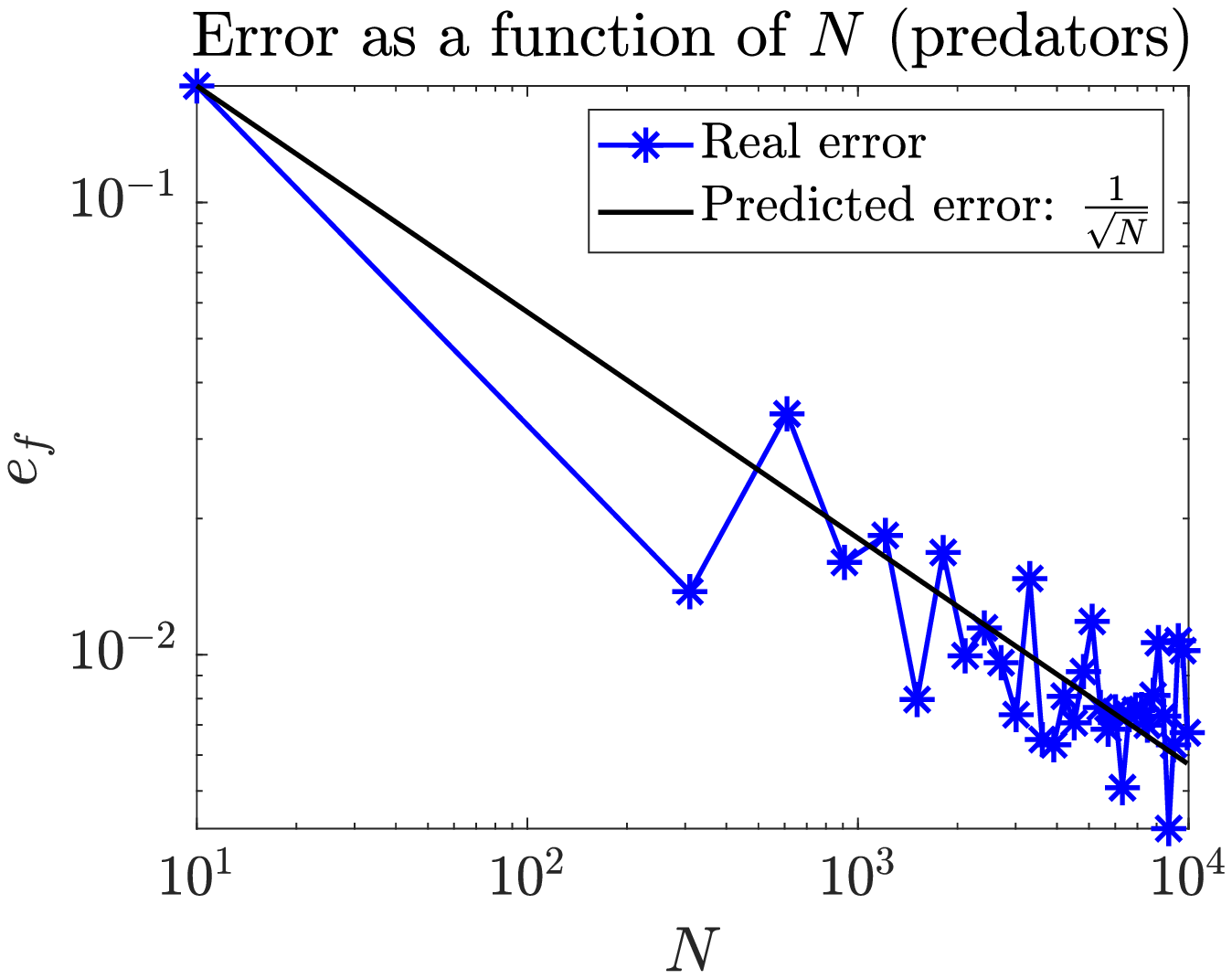}
		\includegraphics[width=0.49\linewidth]{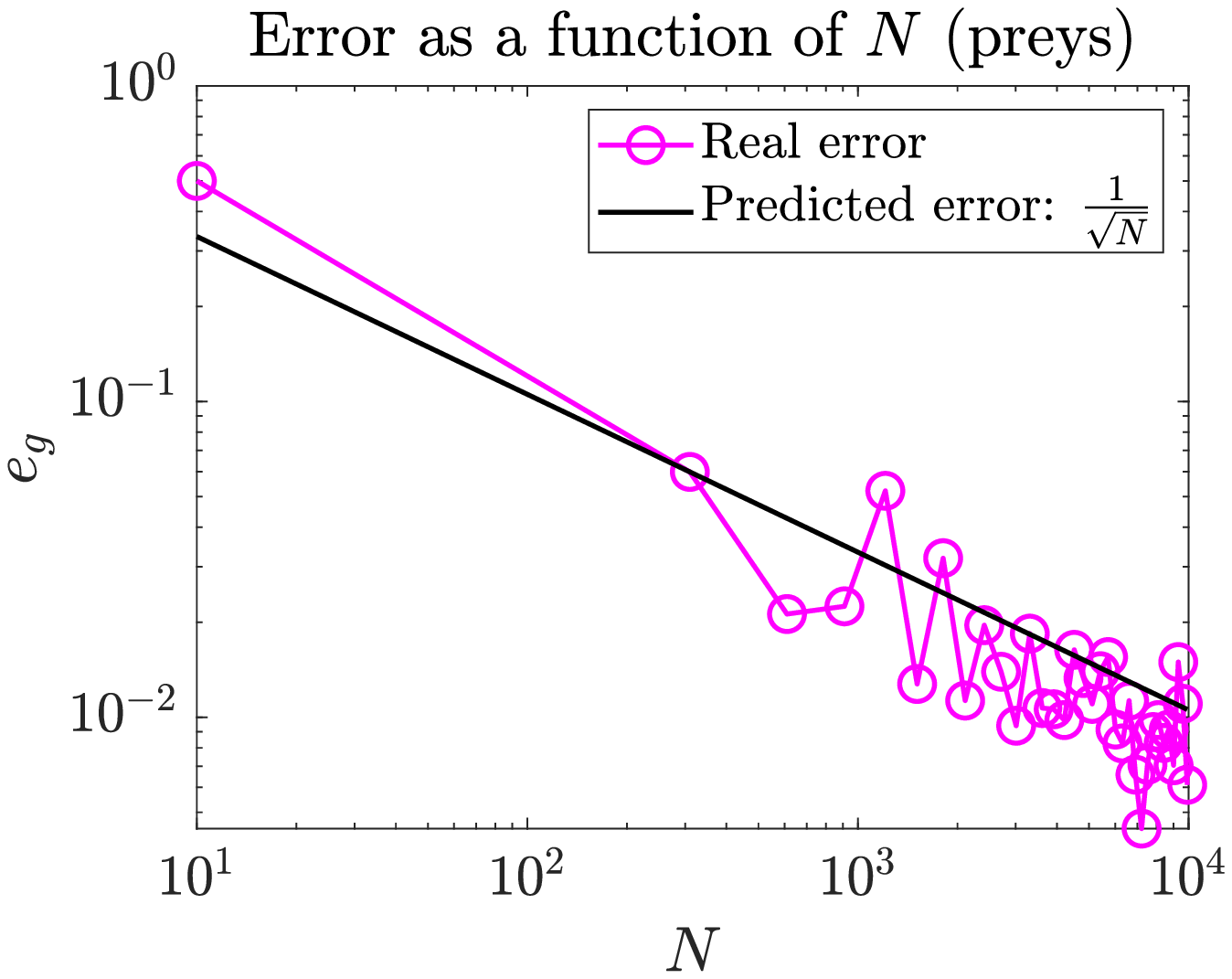}
		\caption{Homogeneous predator-prey model: efficient Monte Carlo algorithm's error computed as in equations \eqref{error} for $N=[10,\ldots,10^4]$. The error is proportional to $1/\sqrt{N}$ as the one of the classic Monte Carlo algorithm.  Markers correspond to the values $e_f^N$, $e_g^N$ for a fixed $N$.}
		\label{fig:figure2}
	\end{figure}
	Let us now focus on the heterogeneous model and consider first the one dimensional case. Assume that the dynamics evolves in an interval area of land $[0,L_x]$, $L_x>0$, divided in $M_c$ cells. Figure \ref{fig:figure3} shows three snapshots taken at time $t=5$, $t=50$ and $t=100$ in which the stochastic and mean-field solutions are compared. At time $t=5$, $A_0=N_c/4$ predators and $B_0=N_c/2$ preys are concentrated in the same central cells. At time $t=50$, preys migrate in regions where the predators concentration is lower and predators decrease their size and start to migrate to reach the regions occupied by preys. At time $t=100$, preys continue their migration increasing their size in the regions where the predators concentration is lower.  
	\begin{figure}[h!]
		\centering
		\includegraphics[width=0.327\linewidth]{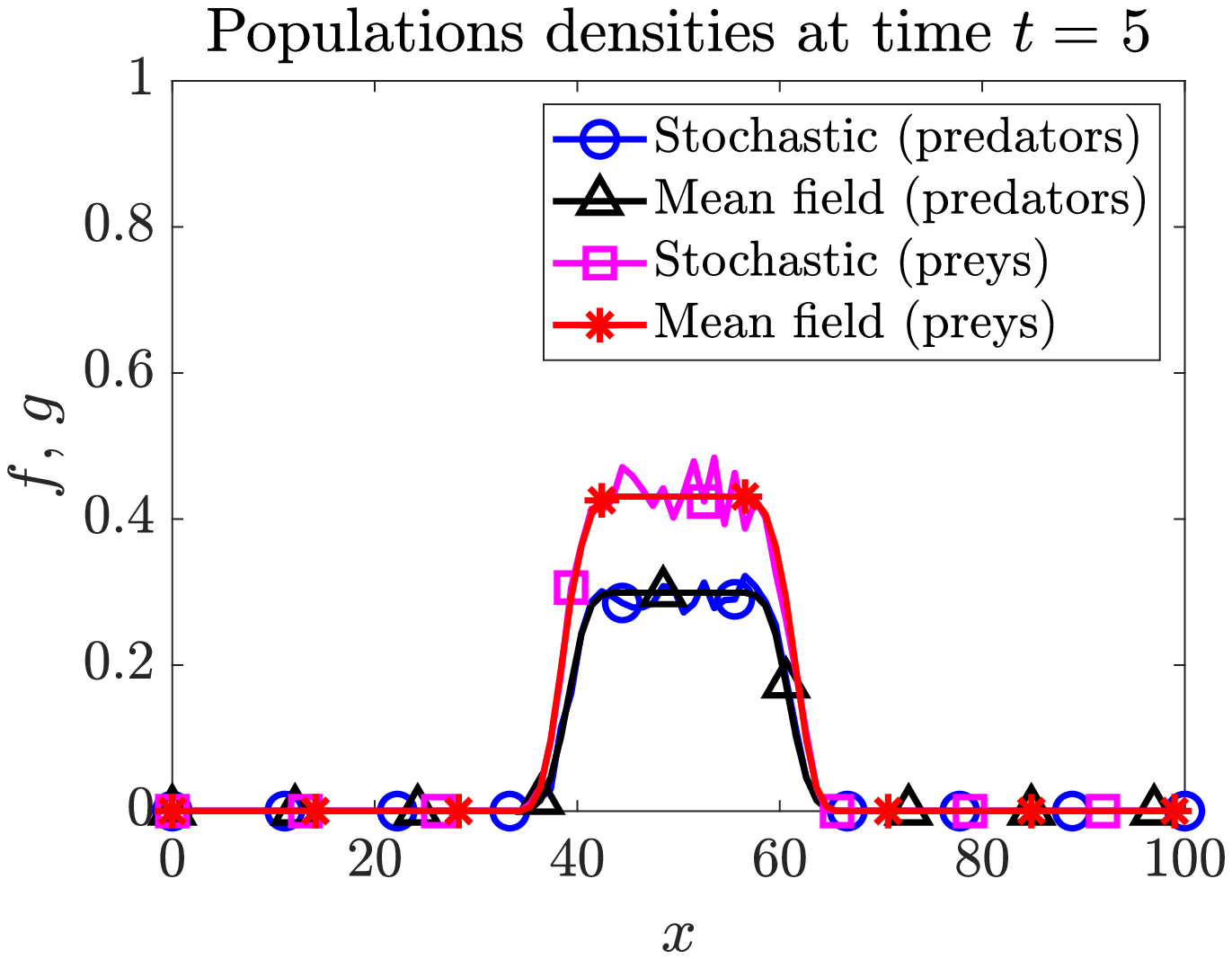}
		\includegraphics[width=0.327\linewidth]{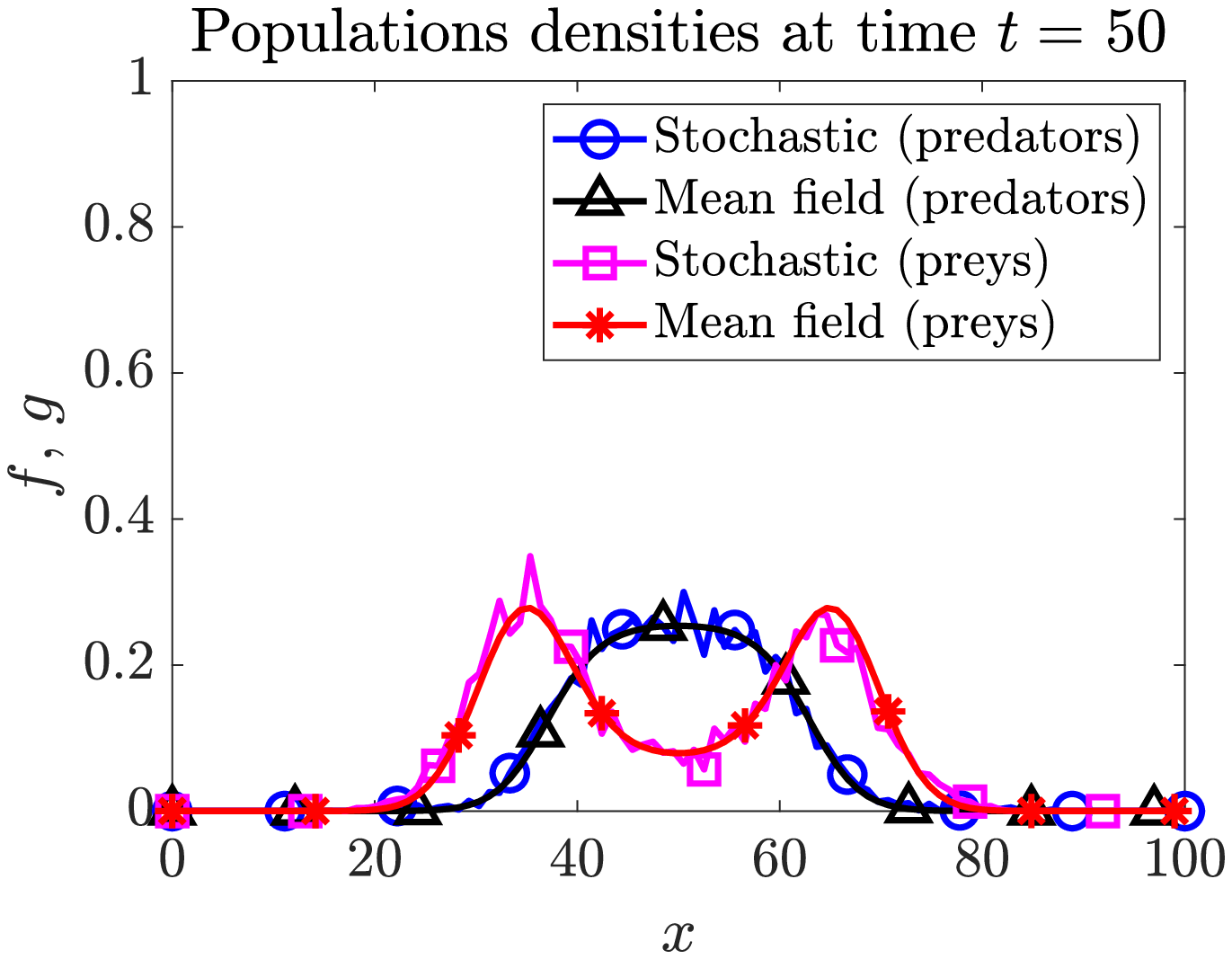}
		\includegraphics[width=0.327\linewidth]{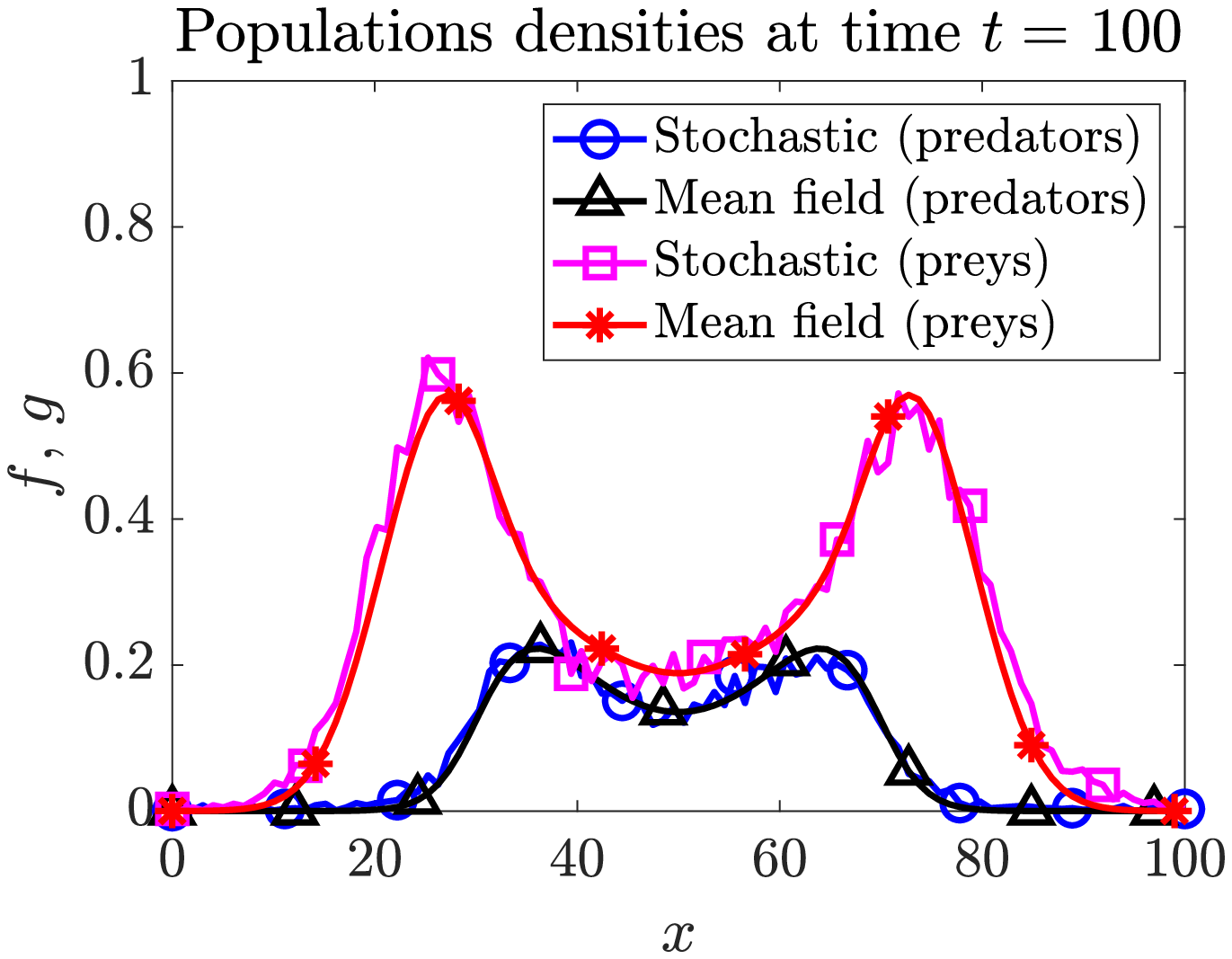}
		\caption{Heterogeneous one dimensional predator-prey model: simulation of the processes described in \eqref{eq:5a}-\eqref{eq:5b}-\eqref{eq:5c} with the efficient Monte Carlo algorithm and  solutions of the mean-field equations \eqref{eq:6} for $N_c=1000$. This figure shows three snapshots taken at time $t=5$ (left), $t=50$ (centre), $t=100$ (right). Markers have been added just to indicate different lines.}
		\label{fig:figure3}
	\end{figure}
	Figure \ref{fig:asymptotic_1d} shows the asymptotic behavior of the two populations at the mean-field and stochastic level. Predators and preys migrate in the whole available space reaching in each cell the value given by the equilibrium \eqref{equilibrium}. 
	\begin{figure}[h!]
		\centering
		\includegraphics[width=0.49\linewidth]{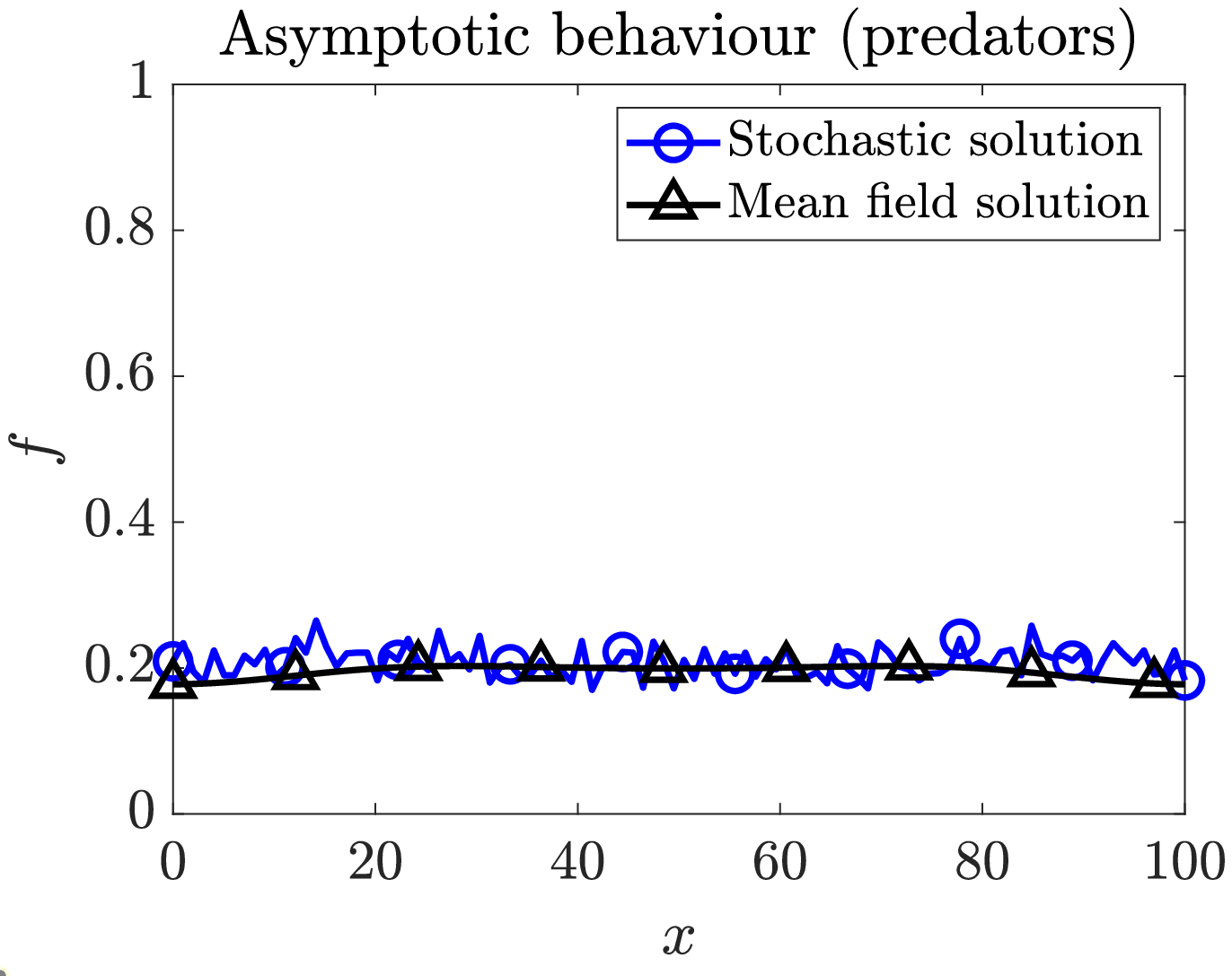}
		\includegraphics[width=0.49\linewidth]{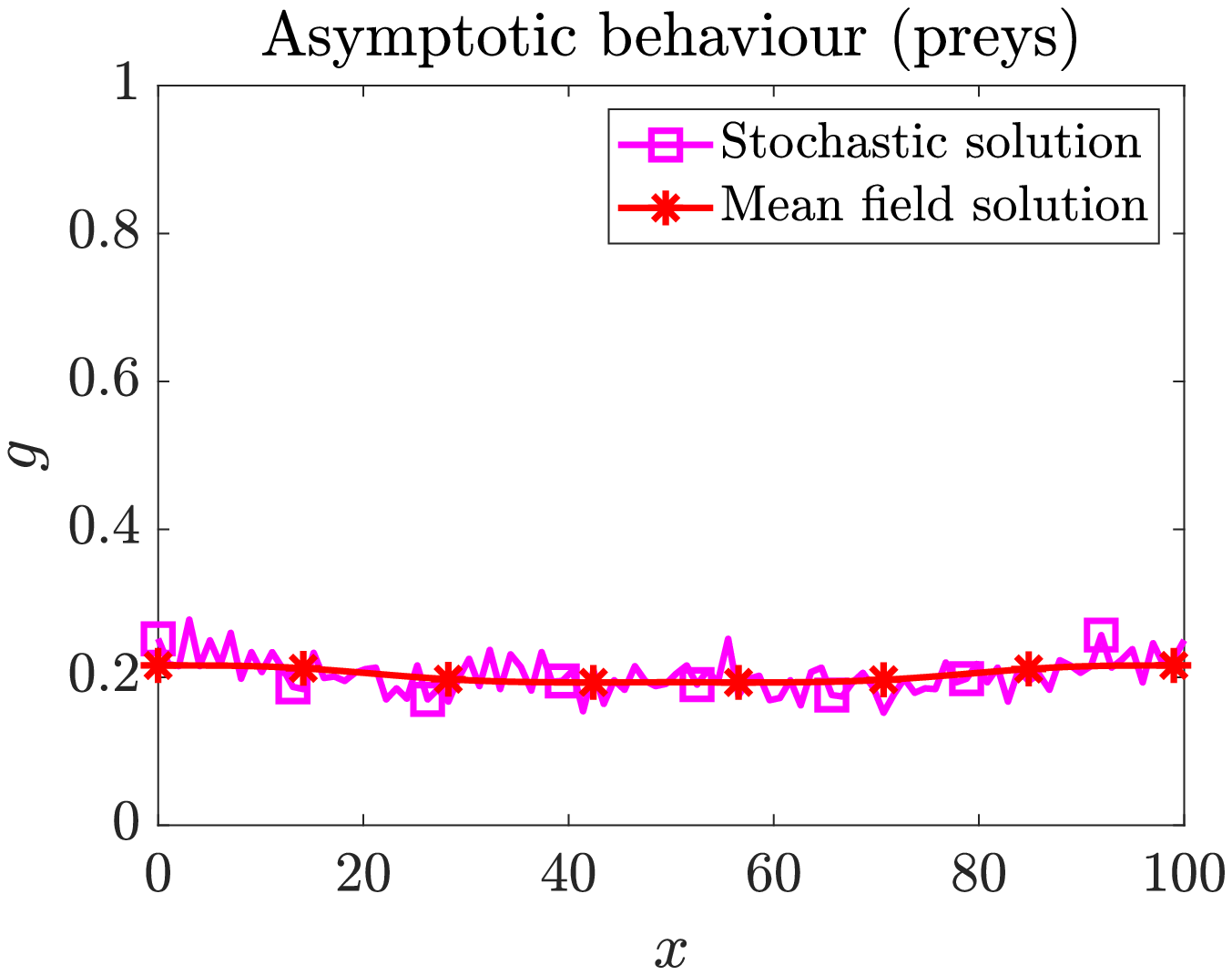}
		\caption{Heterogeneous one dimensional predator-prey model: asymptotic behavior (at time $t=500$) of predators (on the left) and preys (on the right) populations at the mean-field and stochastic level for $N_c=1000$. Markers have been added just to indicate different lines. }
		\label{fig:asymptotic_1d}
	\end{figure} 
	Figure \ref{fig:figure4} shows that the error of the efficient Monte Carlo algorithm as a function of $N_c$ is proportional to $1/\sqrt{N_c}$, as the one of classic algorithms.  The errors are computed as 
	\begin{equation} \label{errorNH}
		\begin{split} 
			&e_f^{N_c}=\left\langle  \max_{t} \vert f(x,t)-f^{N_c}(x,t) \vert \right\rangle_x,\,\,\,
			e_g^{N_c}=\left\langle  \max_{t} \vert g(x,t)-g^{N_c}(x,t) \vert \right\rangle_x,
		\end{split}
	\end{equation}
	where $\left\langle \cdot \right\rangle $ denotes the expected value with respect to $x$.   
	\begin{figure}
		\centering
		\includegraphics[width=0.49\linewidth]{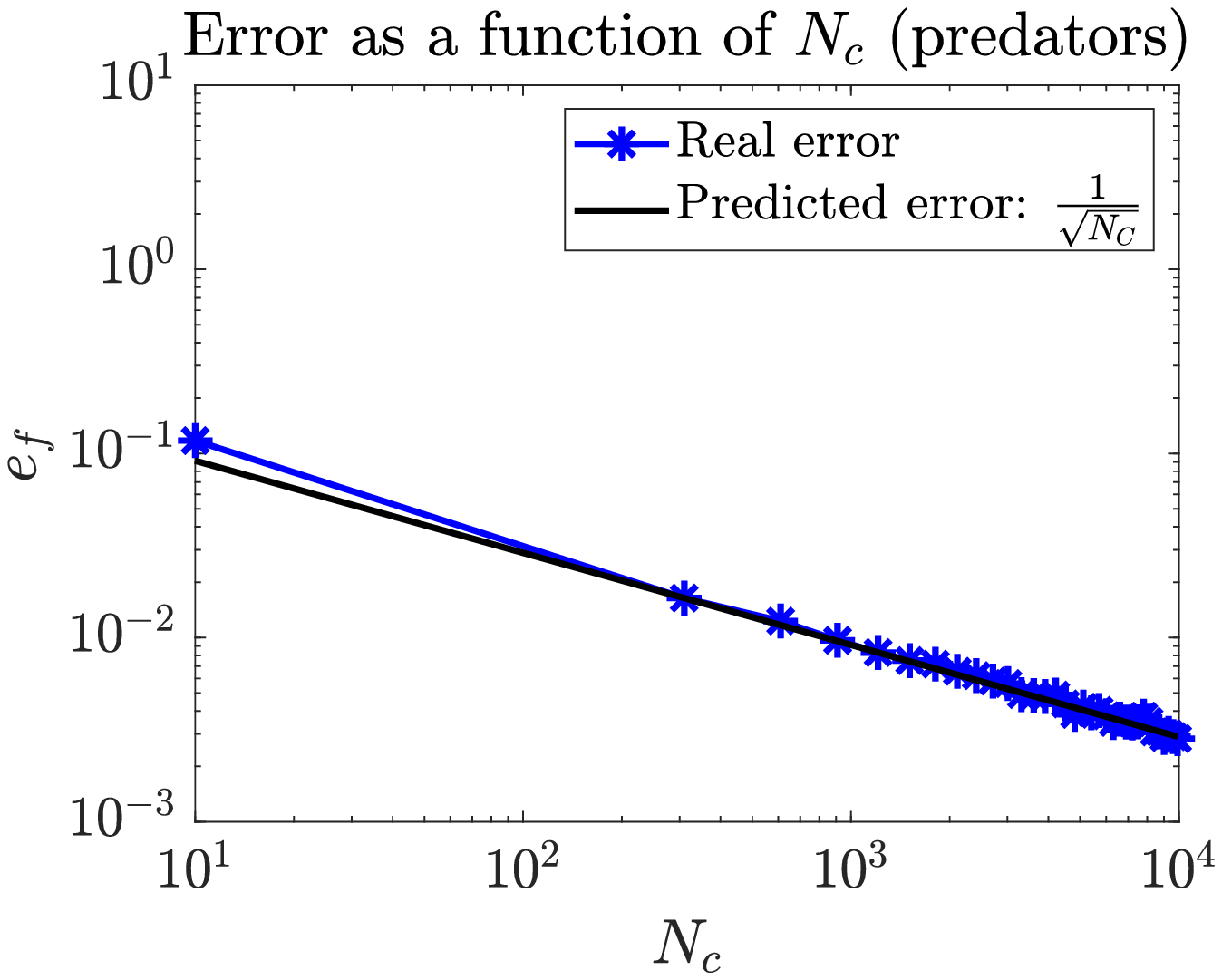}
		\includegraphics[width=0.49\linewidth]{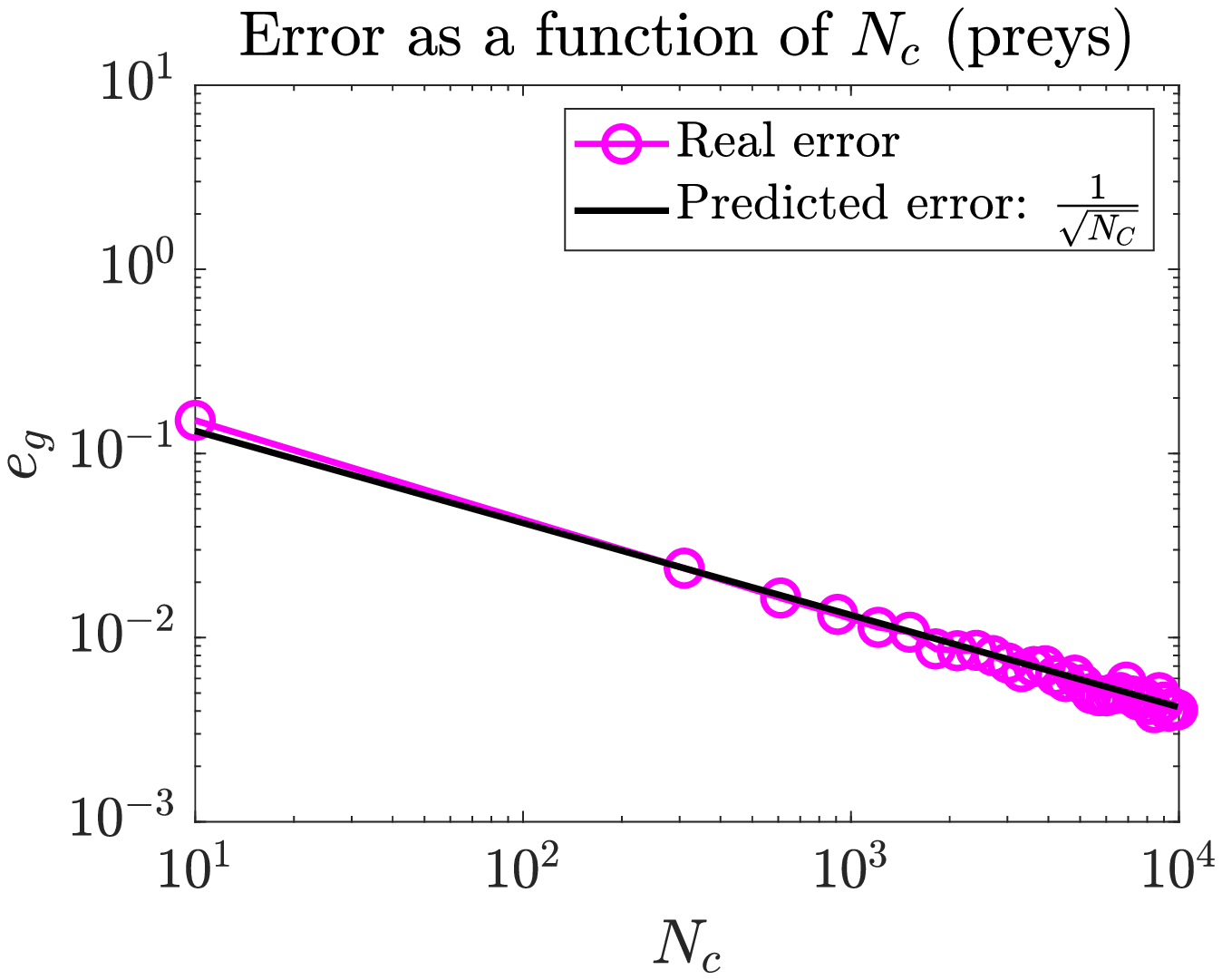}
		\caption{Heterogeneous one dimensional predator-prey model: efficient Monte Carlo algorithm's error computed as in equations \eqref{errorNH} for $N_c=[10,\ldots,10^4]$. The error is proportional to $1/\sqrt{N_c}$, as the one of the classic Monte Carlo algorithm. Markers correspond to the values $e_f^{N_c}$, $e_g^{N_c}$ for a fixed $N_c$.}
		\label{fig:figure4}
	\end{figure}
	Let us now focus on the two dimensional case. Assume that the dynamics evolves in a square area of land $[0,L_x]\times[0,L_y]$, $L_x,L_y>0$ divided in $C_{\ell_x\ell_y}$ cells, for $\ell_x=1,\ldots,M_{c}^x$ and $\ell_y=1,\ldots,M_c^y$. Populations can born, compete, die and migrate in one of the four nearest cells.
	The error computed as in equation \eqref{errorNH} assuming $x=(x,y)$ is still proportional to $1/\sqrt{N_c}$, as shown in Figure \ref{fig:figure5}.
	\begin{figure}
		\centering
		\includegraphics[width=0.49\linewidth]{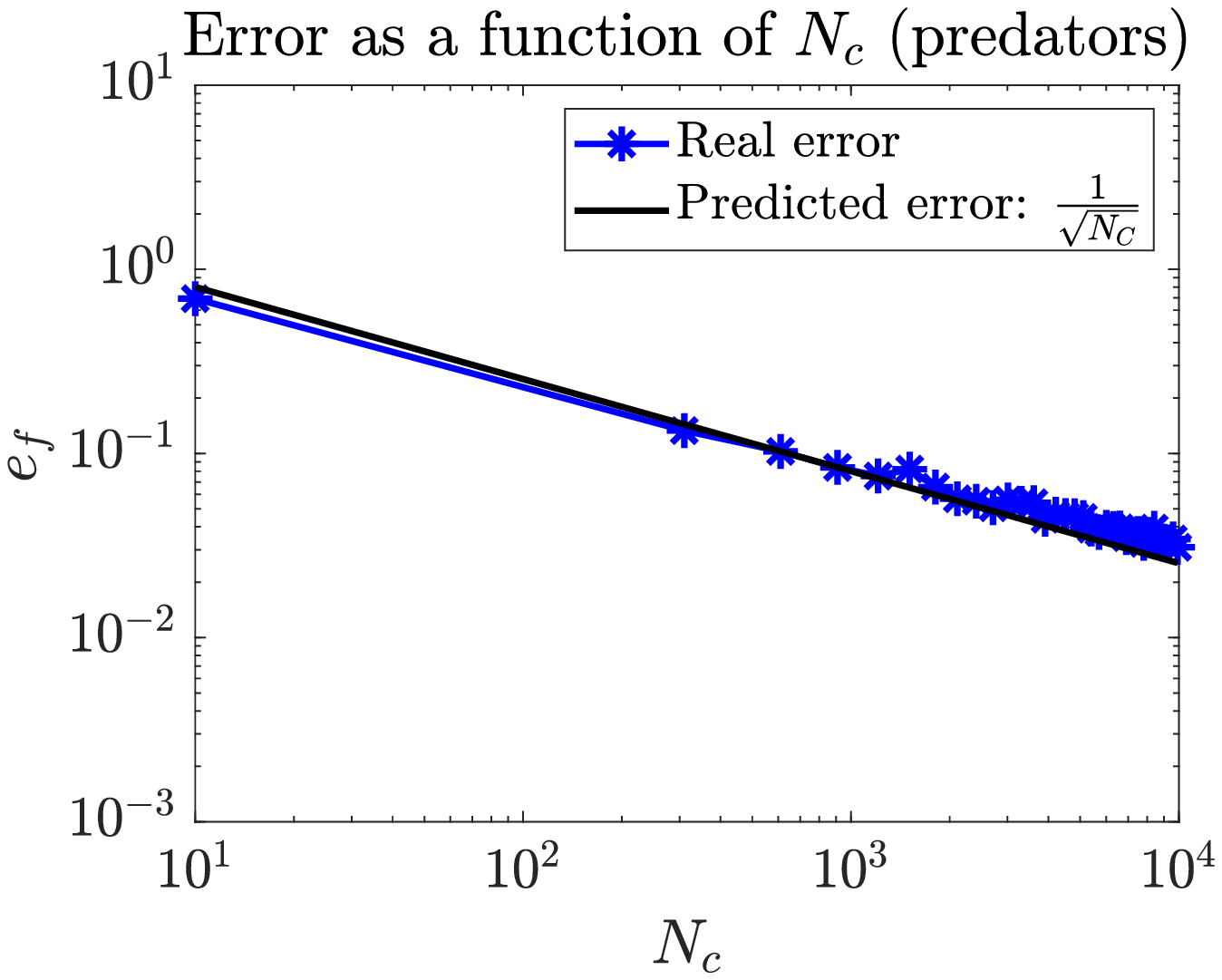}
		\includegraphics[width=0.49\linewidth]{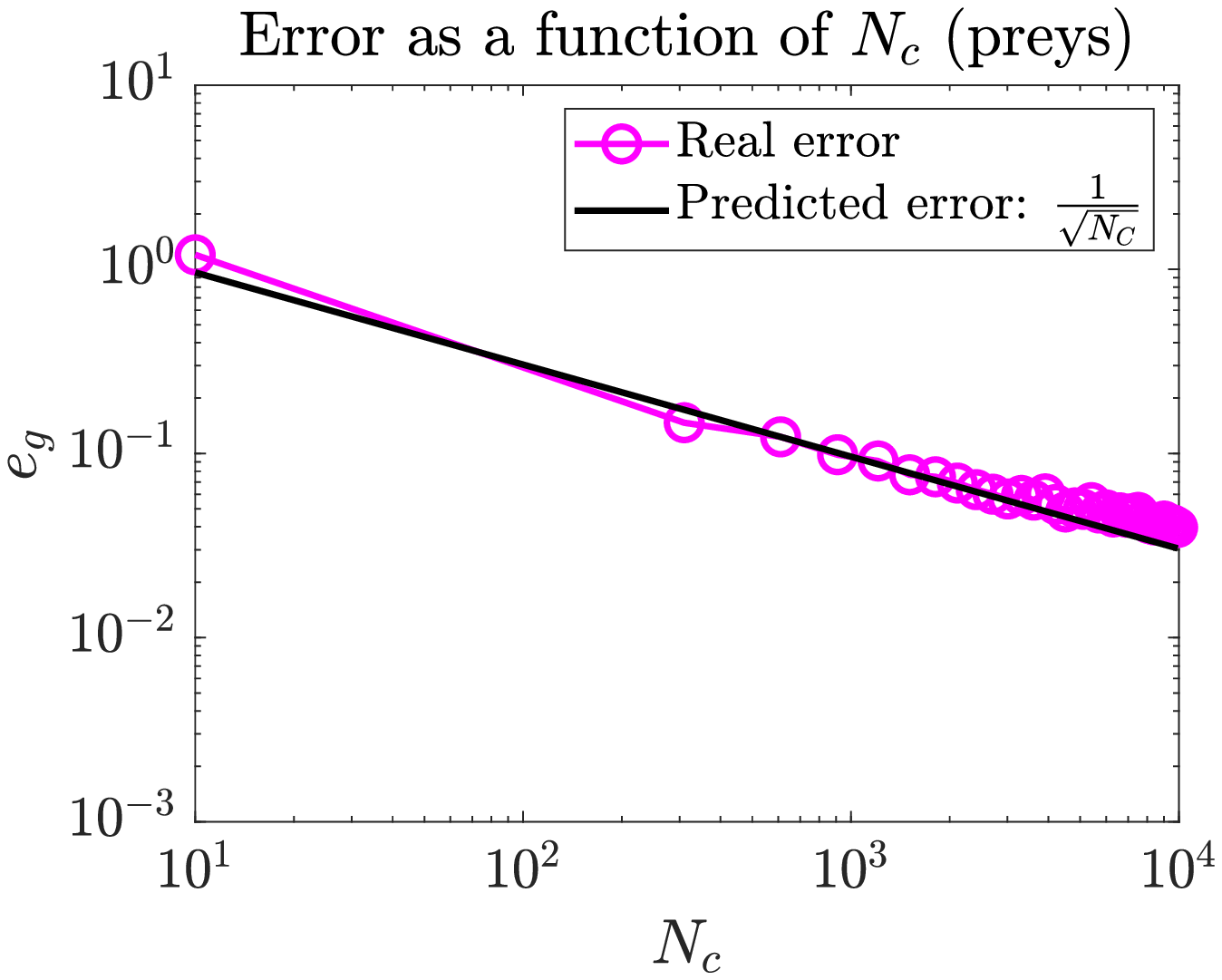}
		\caption{Heterogeneous two dimensional predator-prey model: efficient Monte Carlo algorithm's error computed as in equations \eqref{errorNH} for $N_c=[10,\ldots,10^4]$. The error is proportional to $1/\sqrt{N_c}$, as the one of the classic Monte Carlo algorithm. Markers correspond to the values $e_f^{N_c}$, $e_g^{N_c}$  for a fixed $N_c$.}
		\label{fig:figure5}
	\end{figure}
	Figure \ref{fig:figure6} and Figure \ref{fig:figure7} show three snapshots describing the time evolution of preys and predators densities at the mean-field and stochastic level in the heterogeneous two dimensional case for $N_c=1000$. 
	At time $t=5$, $B_0=N_c/2$ preys are concentrated in the central cells and surrounded by $A_0=N_c/4$ predators.  At time $t=100$ predators migrate in the central cells while preys reduce their size and start to migrate in the regions where the predators concentration is lower. At time $t=150$ preys are still migrating and increasing their size. Predators on the contrary are reducing their size and migrating to reach the regions in which preys are mainly concentrated.   
	\begin{figure}
		\centering
		\includegraphics[width=0.49\linewidth]{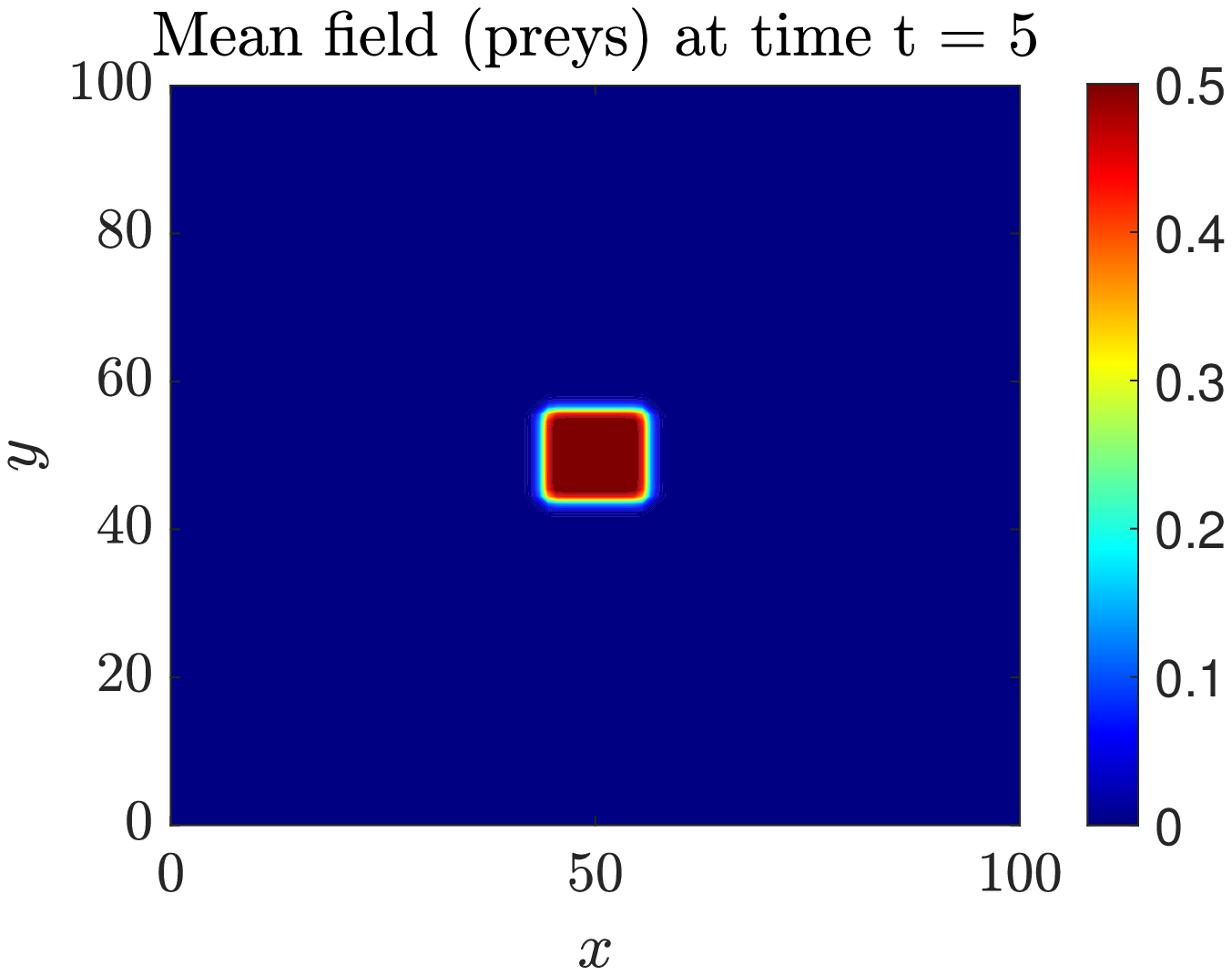}
		\includegraphics[width=0.49\linewidth]{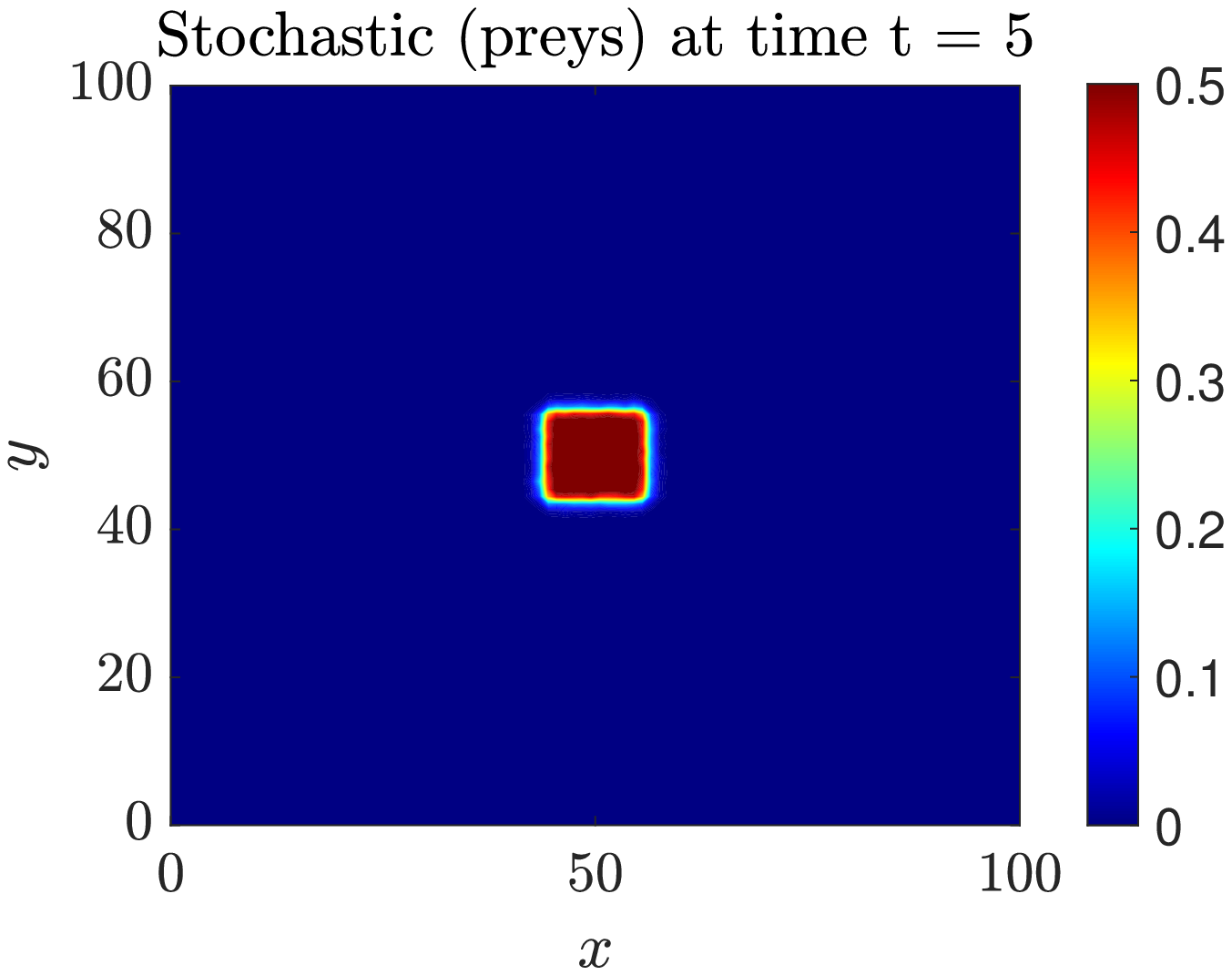}
		\includegraphics[width=0.49\linewidth]{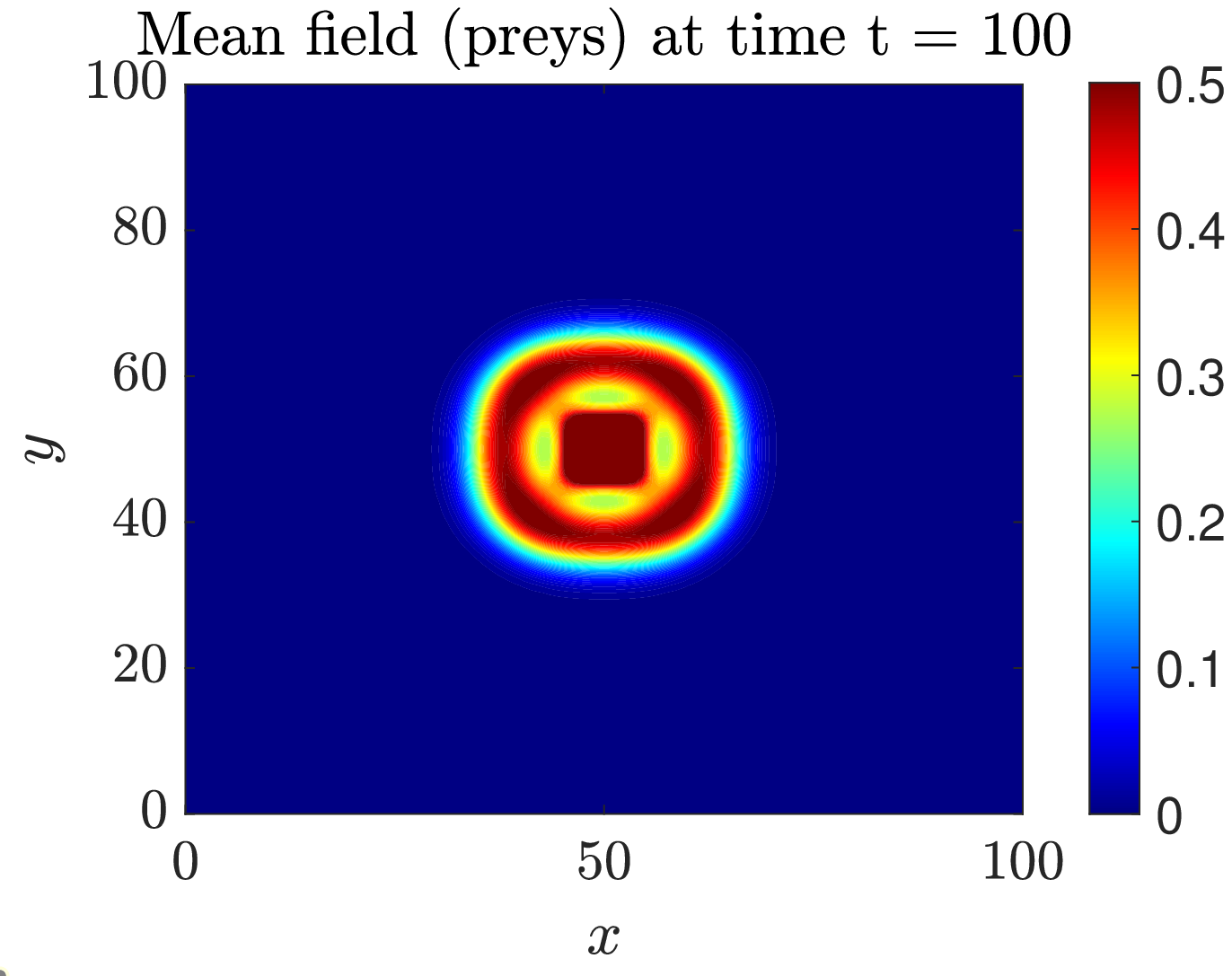}
		\includegraphics[width=0.49\linewidth]{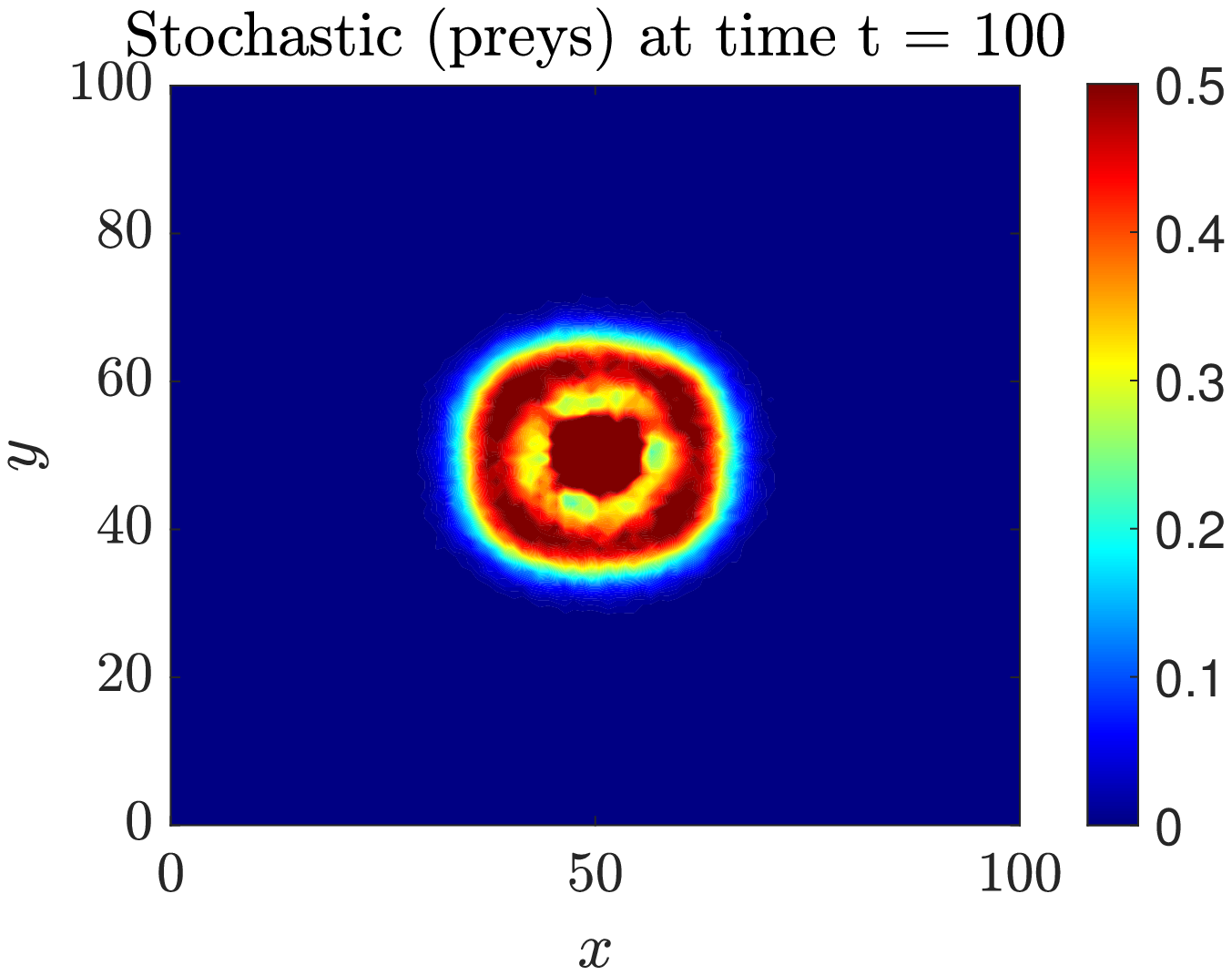}
		\includegraphics[width=0.49\linewidth]{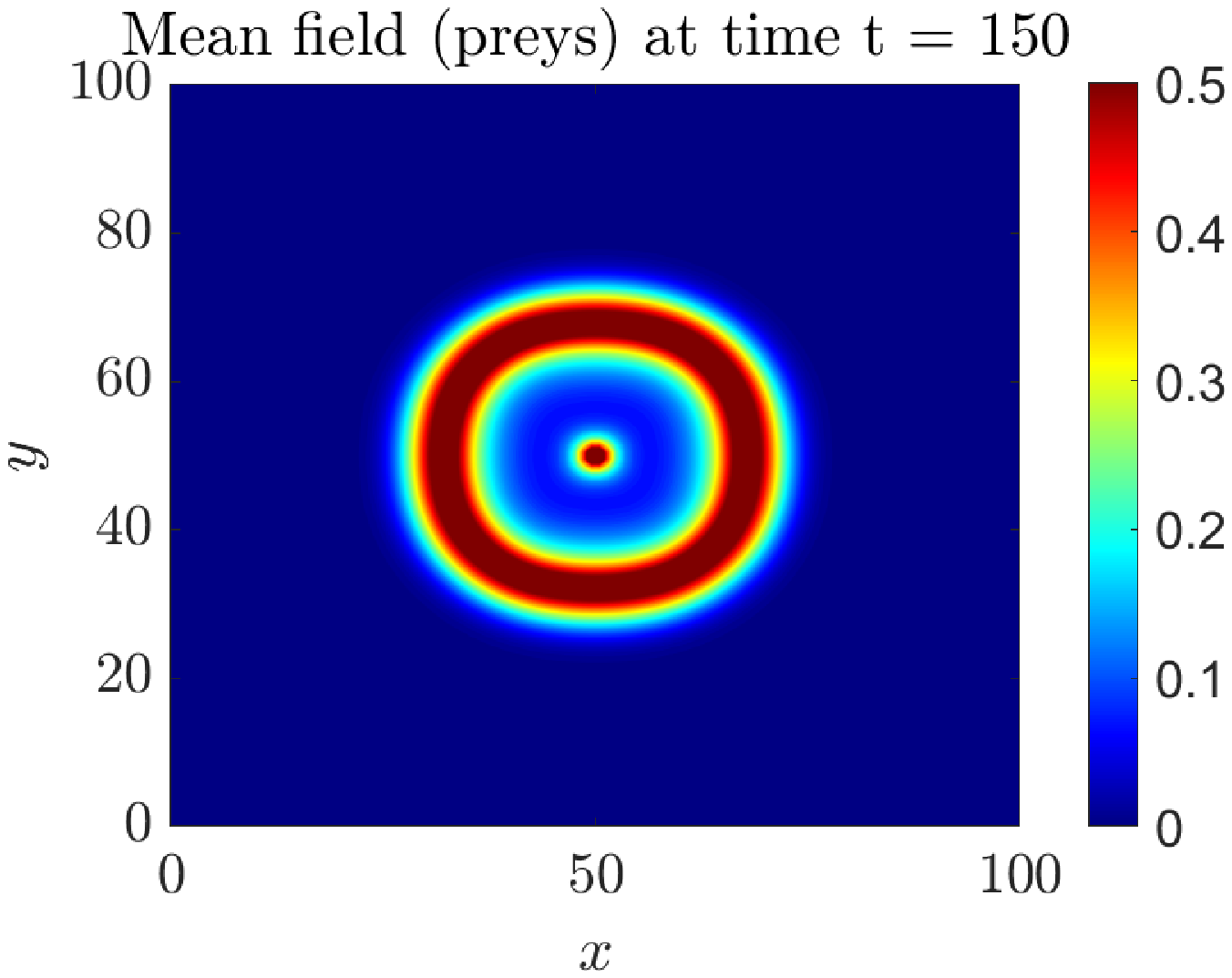}
		\includegraphics[width=0.49\linewidth]{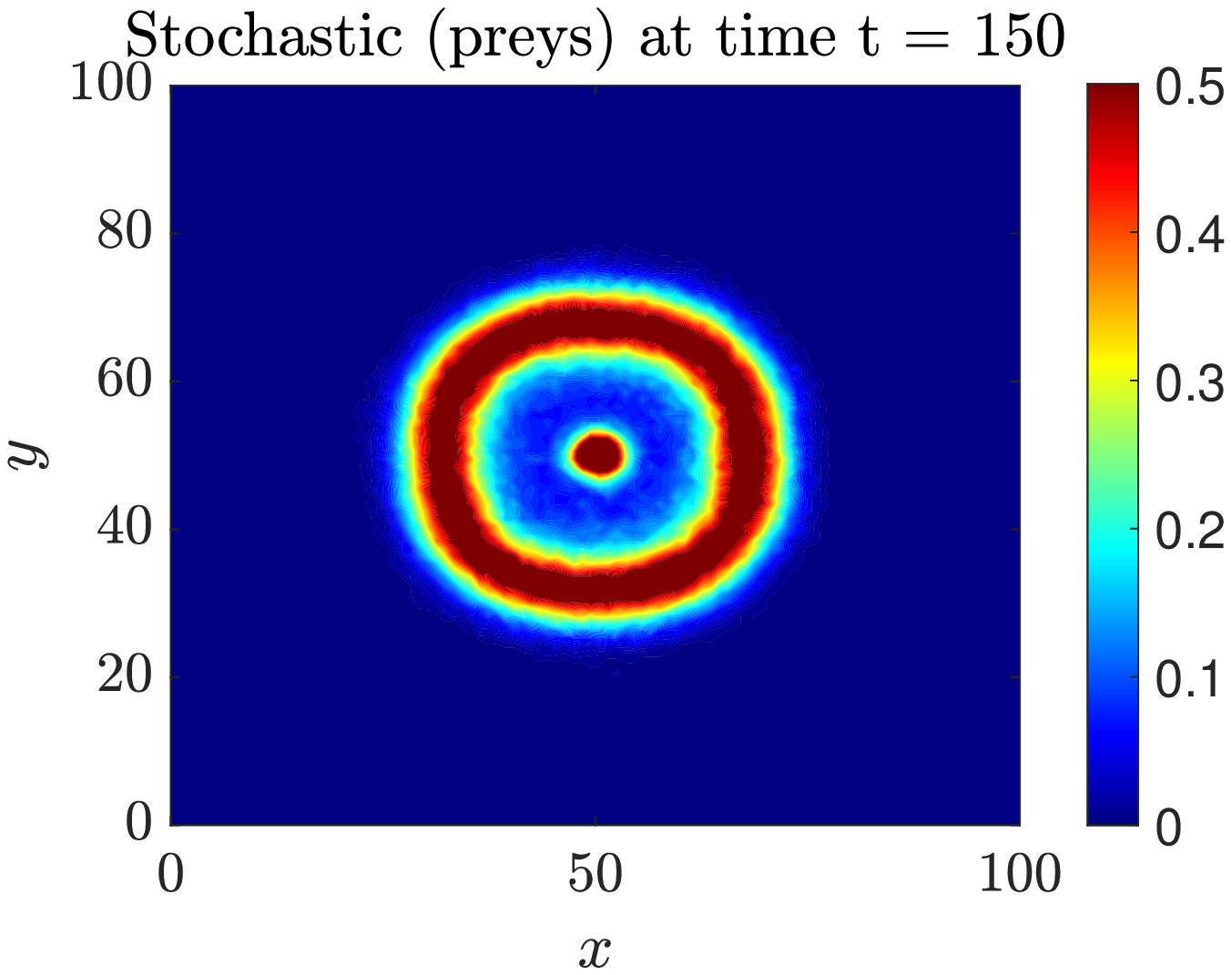}
		\caption{Heterogeneous two dimensional predator-prey model (preys population): simulation of the processes described in \eqref{eq:5a}-\eqref{eq:5b}-\eqref{eq:5c} with the efficient Monte Carlo algorithm and solutions of the mean-field equations \eqref{eq:6} in the two dimensional case for $N_c=1000$. This figure shows three snapshots taken at time $t=5$ (top), $t=100$ (middle), $t=150$ (bottom). On the left, mean-field solutions and on the right, stochastic simulations. }
		\label{fig:figure6}
	\end{figure}
	\begin{figure}
		\centering
		\includegraphics[width=0.49\linewidth]{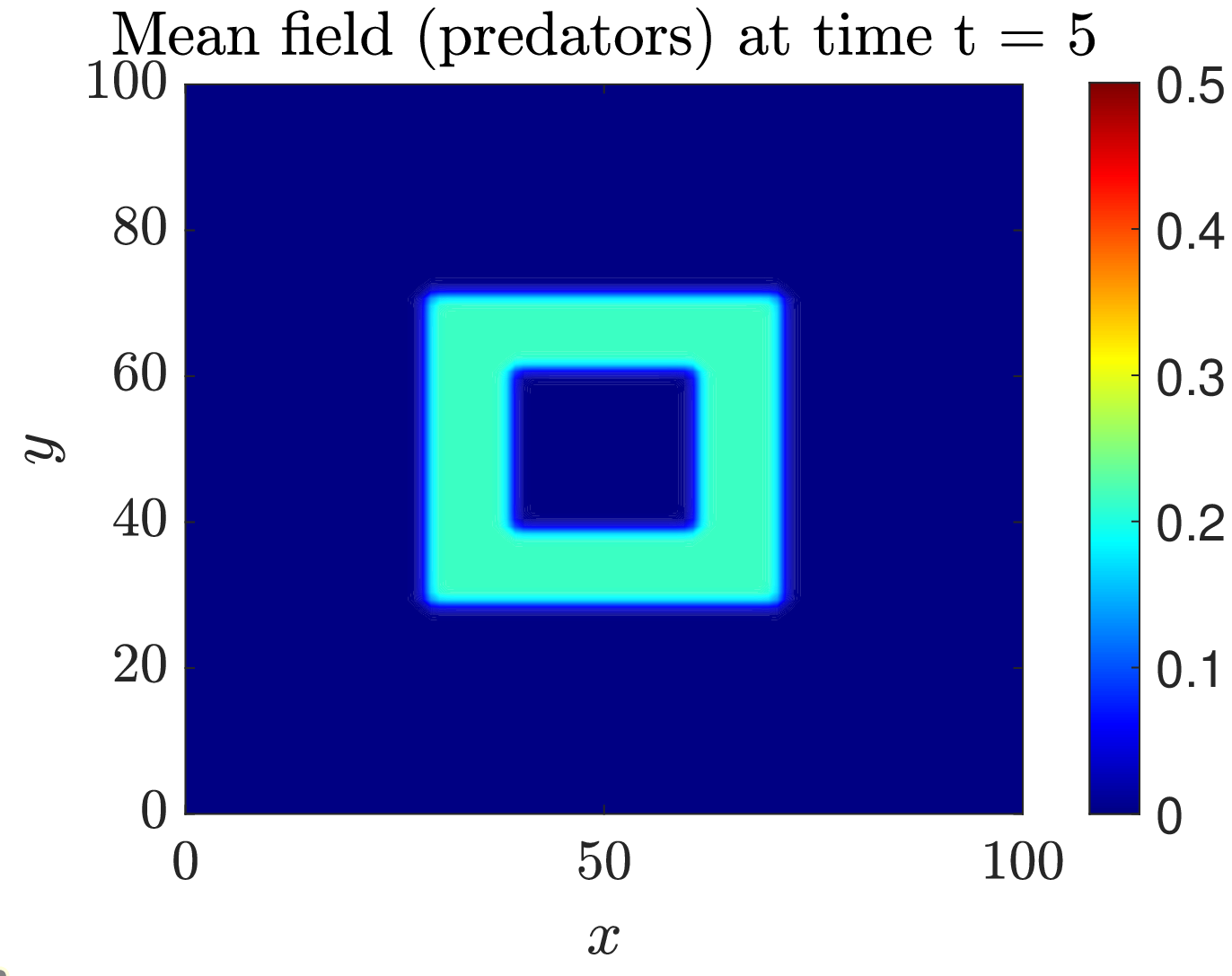}
		\includegraphics[width=0.49\linewidth]{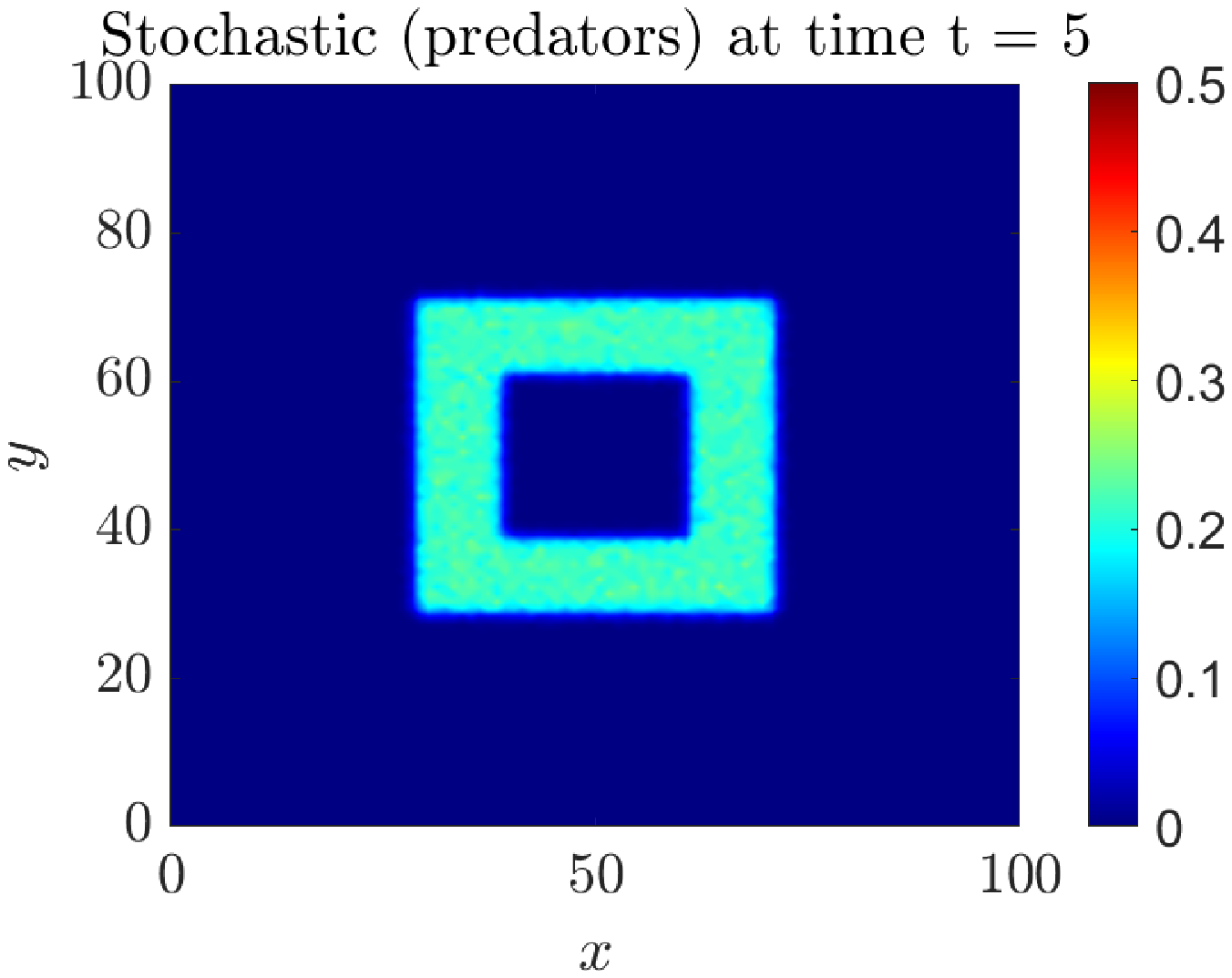}
		\includegraphics[width=0.49\linewidth]{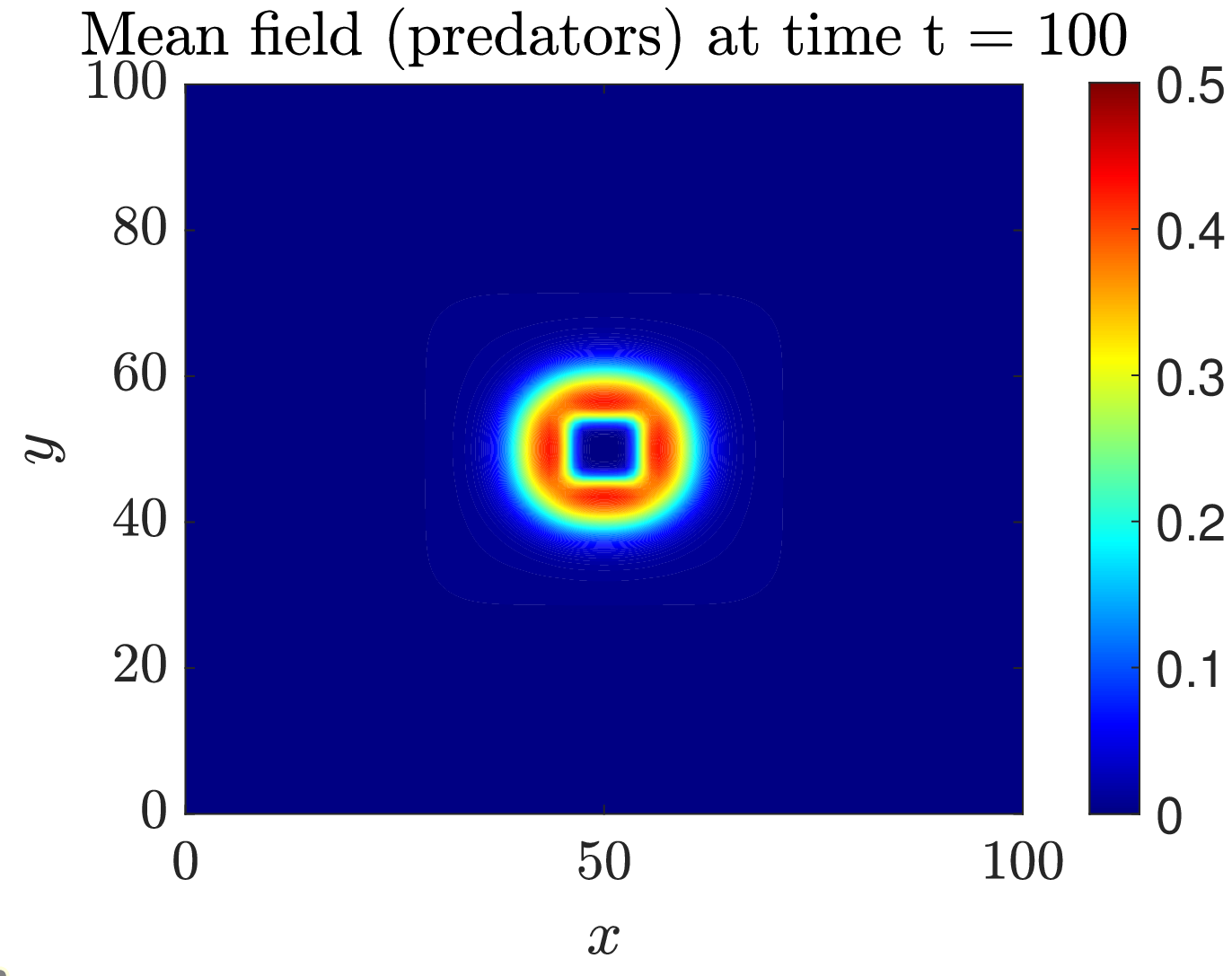}
		\includegraphics[width=0.49\linewidth]{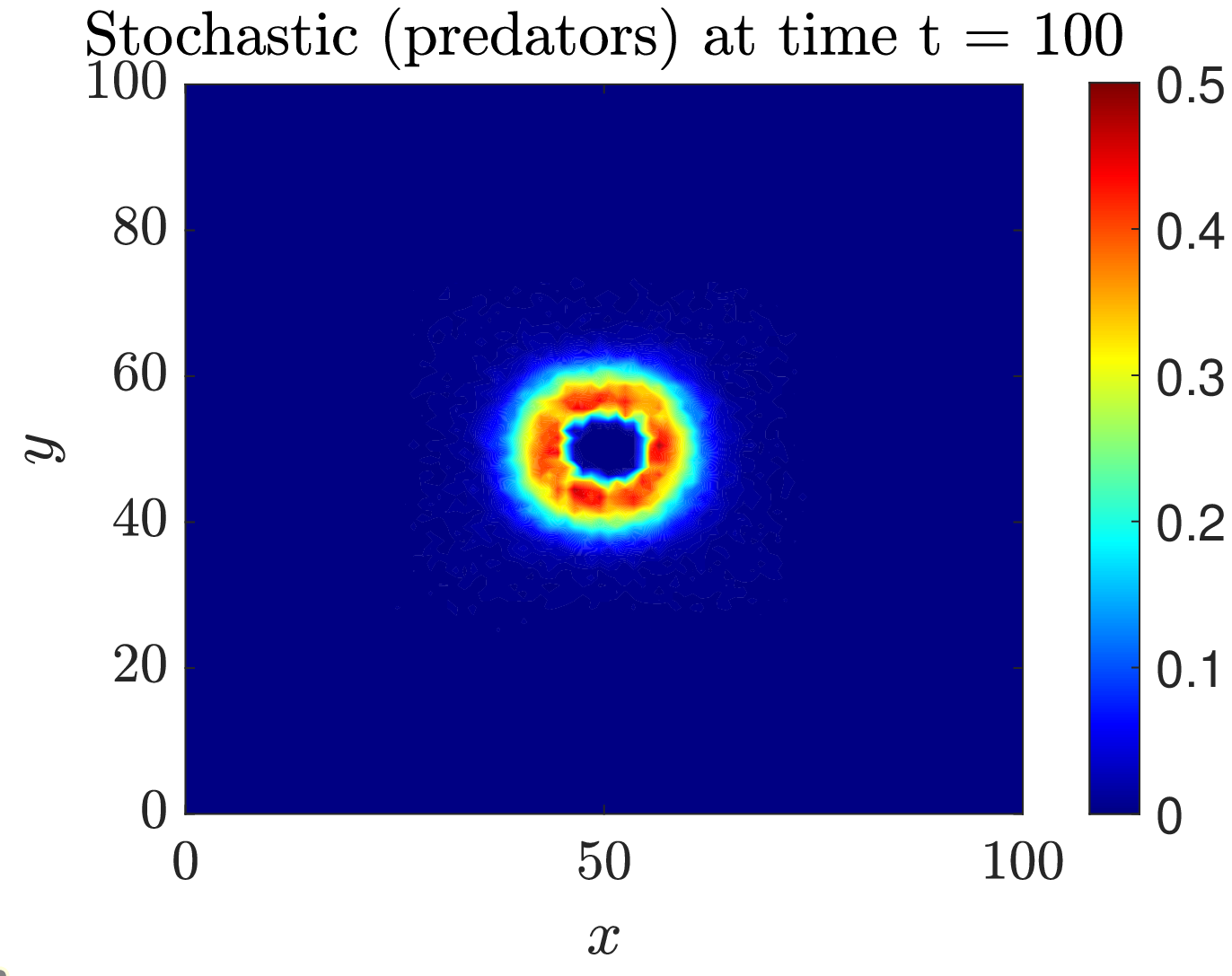}
		\includegraphics[width=0.49\linewidth]{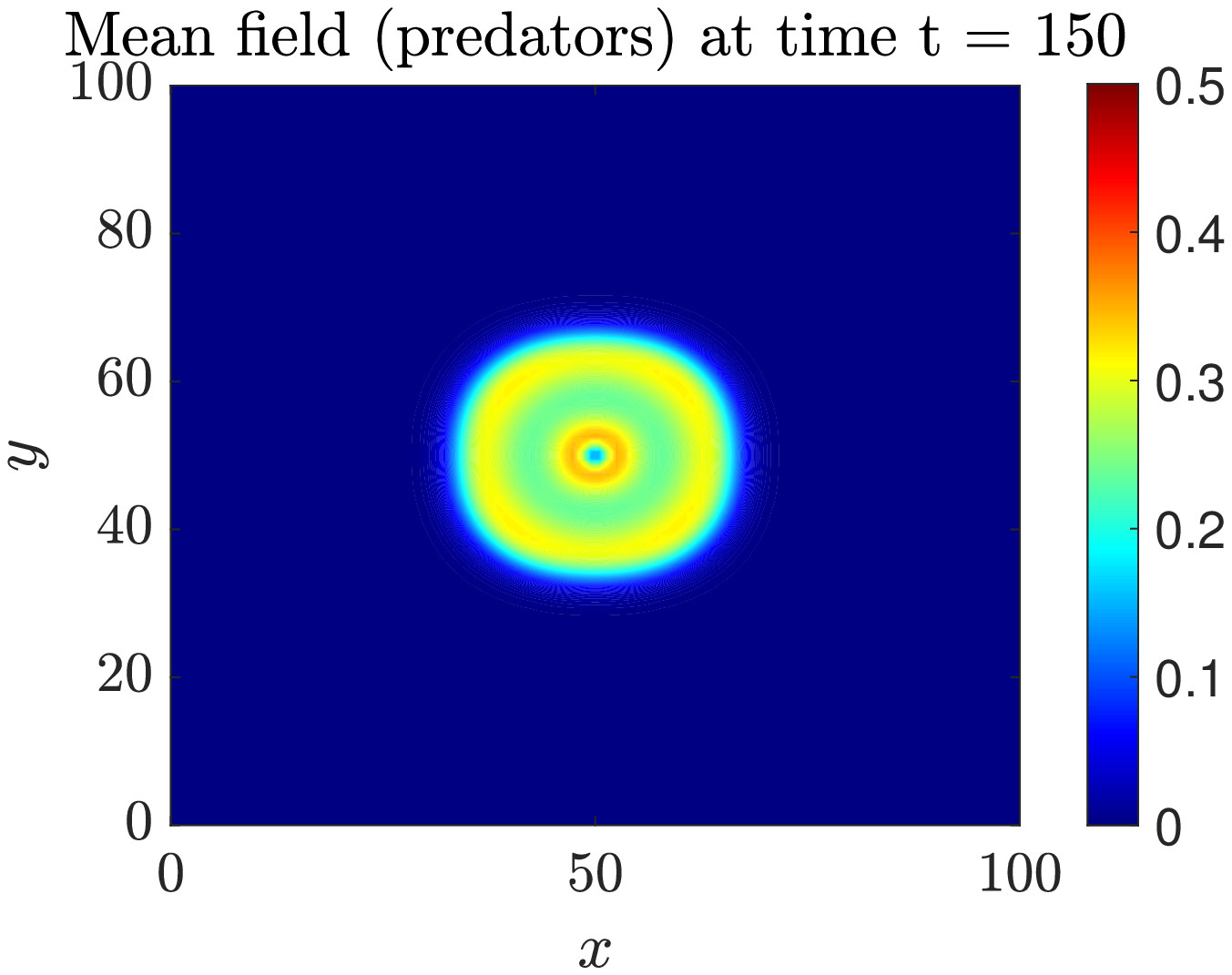}
		\includegraphics[width=0.49\linewidth]{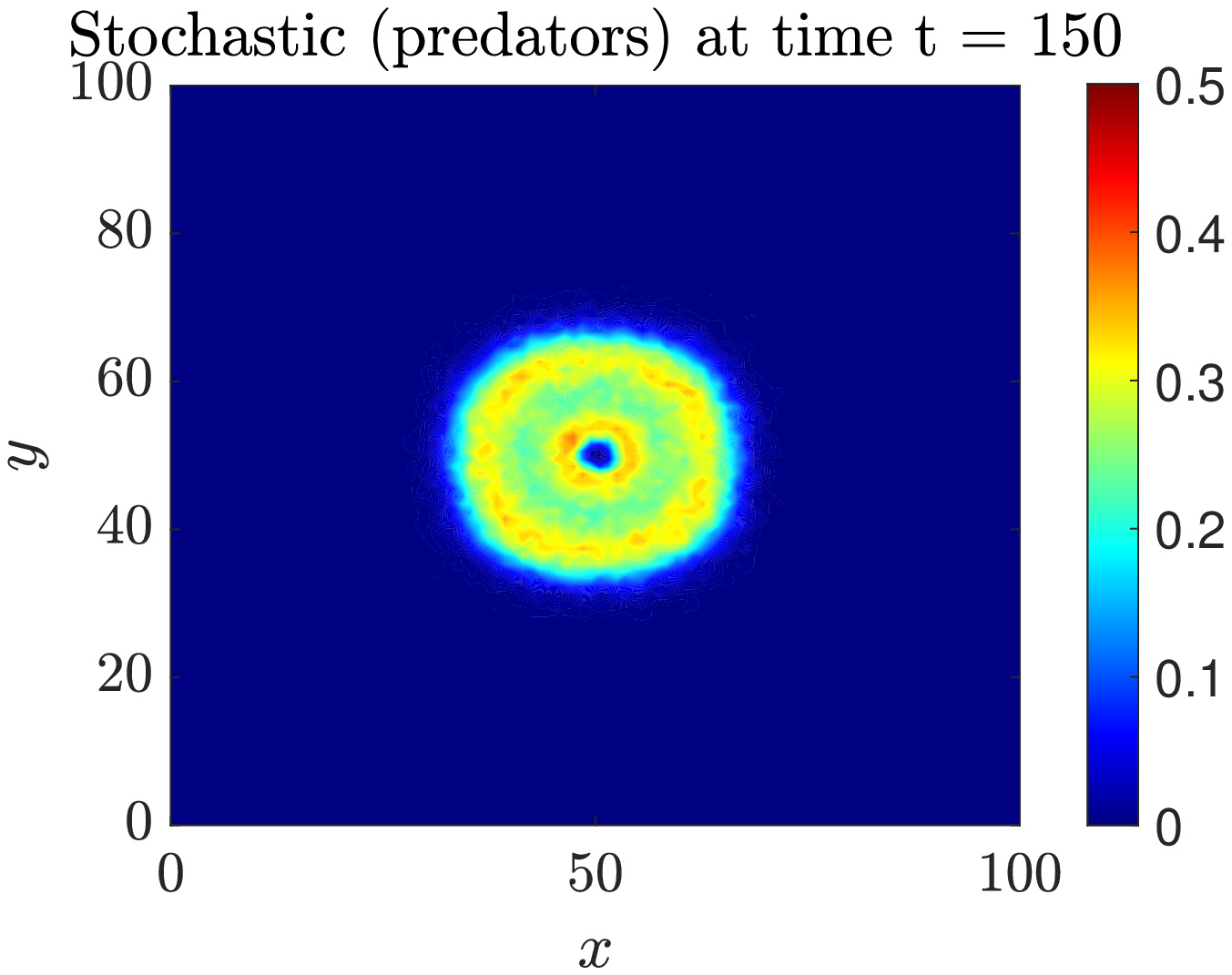}
		
		\caption{Heterogeneous two dimensional predator-prey model (predators population): simulation of the processes described in \eqref{eq:5a}-\eqref{eq:5b}-\eqref{eq:5c} with the efficient Monte Carlo algorithm and solutions of the mean-field equations \eqref{eq:6} in the two dimensional case for $N_c=1000$. This figure shows three snapshots taken at time $t=5$ (top), $t=100$ (middle), $t=150$ (bottom). On the left, mean-field solutions and on the right, stochastic simulations.  }
		\label{fig:figure7}
	\end{figure}  
	In Figure \ref{fig:asymptotic_2d} we see for simplicity just the asymptotic behavior of the predators population. One can show that both populations in long time migrate in the whole available space reaching in each cell the value given by the equilibrium \eqref{equilibrium}. 
	\begin{figure}
		\centering
		\includegraphics[width=0.49\linewidth]{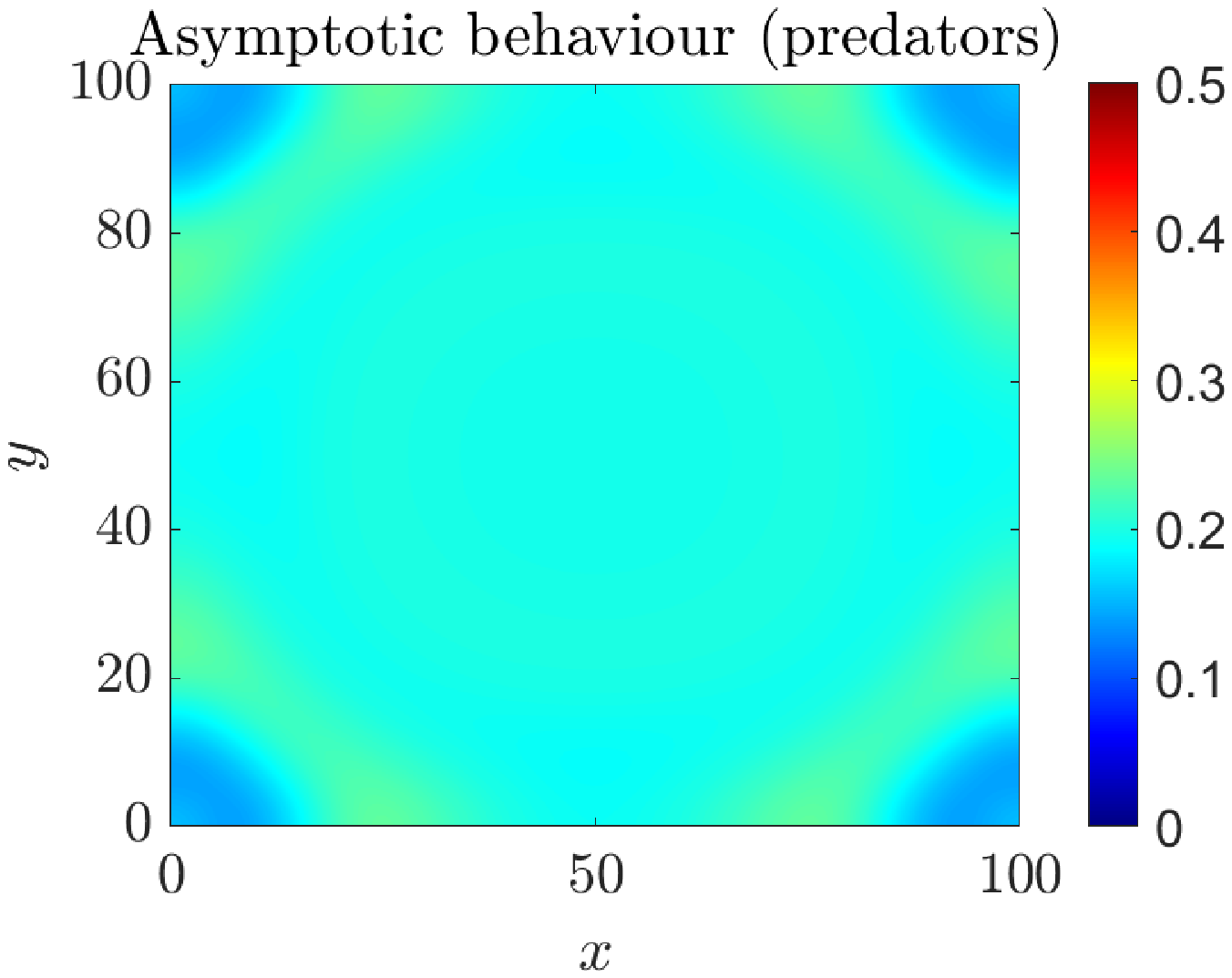}
		\includegraphics[width=0.49\linewidth]{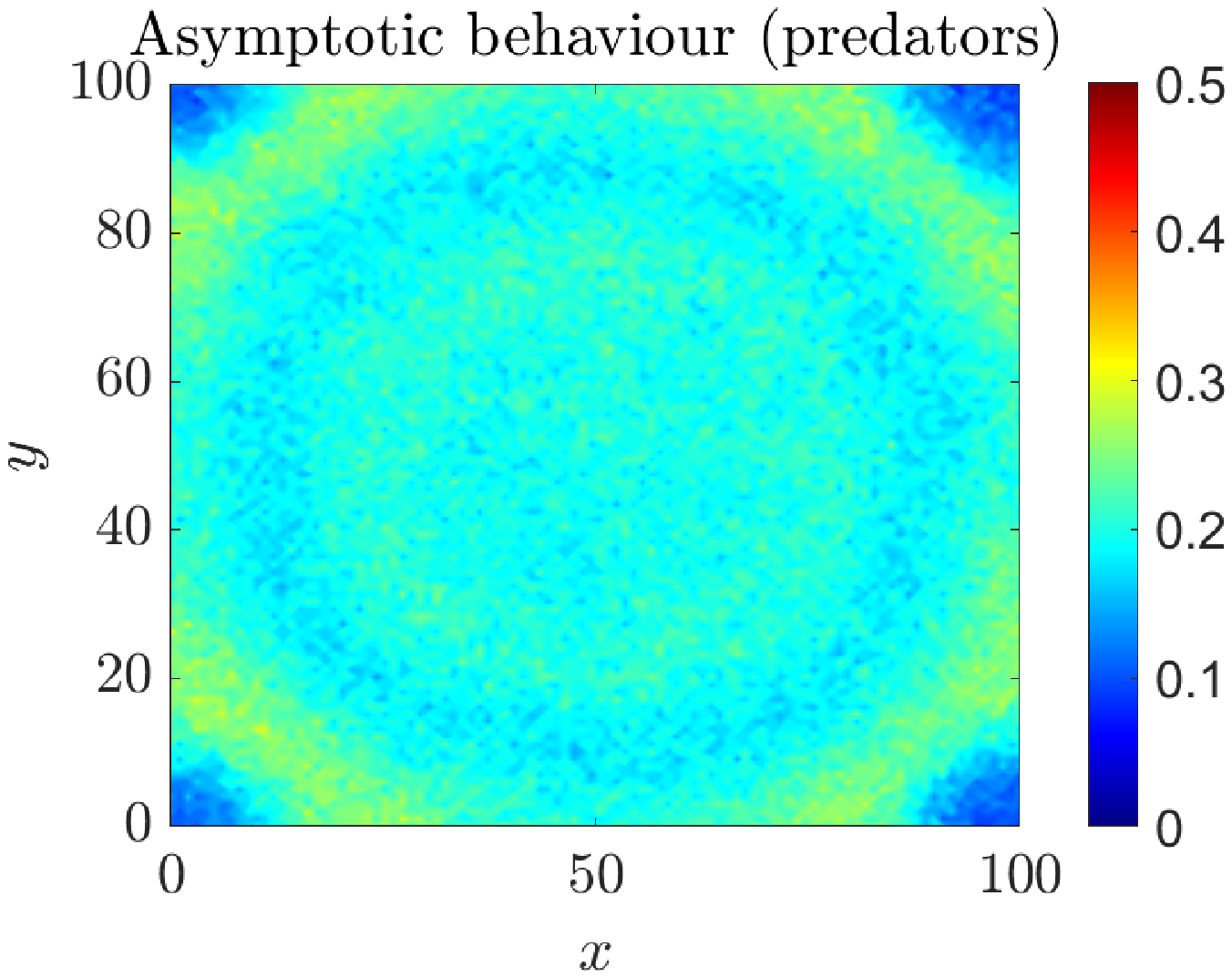}
		\caption{Heterogeneous two dimensional predator-prey model: asymptotic behaviour (at time $t=500$) of predators population at the mean-field (on the right) and stochastic (on the left) level for $N_c=1000$.}
		\label{fig:asymptotic_2d}
	\end{figure} 
	\begin{rmk}
		Note that the mean-field solutions of equations \eqref{eq:6} present a damped behavior in time while the stochastic solutions obtained with the efficient Monte Carlo algorithm have a persistent behavior in time. One can prove that the stochastic persistency is due to a resonant effect,
		\cite{mckane2005predator}.  We refer to section \ref{section5} for a dedicated test to analyzed this behavior.	
	\end{rmk} 
	%%%%%%%%%%%%%%%%%%%%%%%%%%%%%%%%%%%%%%%%%%%%%%%%%%%%%%%%%%%%%%%%%%%%%%%%%%%%%%%%%%%%%
	%%%%%%%%%%%%%%%%%%%%%%%%%%%%%%%%%%%%%%%%%%%%%%%%%%%%%%%%%%%%%%%%%%%%%%%%%%%%%%%%%%%%%
%	\newpage
	\subsection{Test 2: Computational cost}
	Let us consider the homogeneous case and assume to fix the parameters $\mu=0.5$, $b^r=1$, $d_1^r=d_2^r=0.3$, and to let the competition parameters $p^r_1, p^r_2$ to vary between $0.1$ and $0.9$. Figure \ref{fig:figure8} shows that the computational cost of the efficient Monte Carlo algorithm is lower than the one of the other algorithms.
	\begin{figure}[h!]
		\centering
		\includegraphics[width=0.49\linewidth]{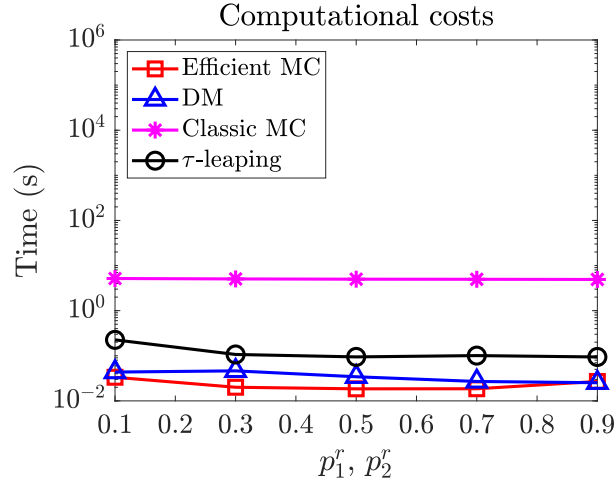}
		\caption{Homogeneous predator-prey model: computational cost of efficient and classic Monte Carlo algorithms, direct method and $\tau$-leaping method as the competition parameters vary for $N$ fixed. The dynamics in \eqref{eq:1a}-\eqref{eq:1b} is simulated for $N=500$, $\mu=0.5$, $b^r=1$, $d_1^r=d_2^r=0.3$, $p_{1}^r=p_{2}^r=0.1,\ldots,0.9$. Markers represent the computational costs relative to the parameters choice indicated.}
		\label{fig:figure8}
	\end{figure}
	Figure \ref{fig:figure8bis} shows that the efficient Monte Carlo algorithm, the direct method and the $\tau$-leaping algorithm have a computational complexity of order $N$ in time while the one of the classic Monte Carlo algorithm is of order $N^2$. 
	\begin{figure}[h!]
		\centering
		\includegraphics[width=0.49\linewidth]{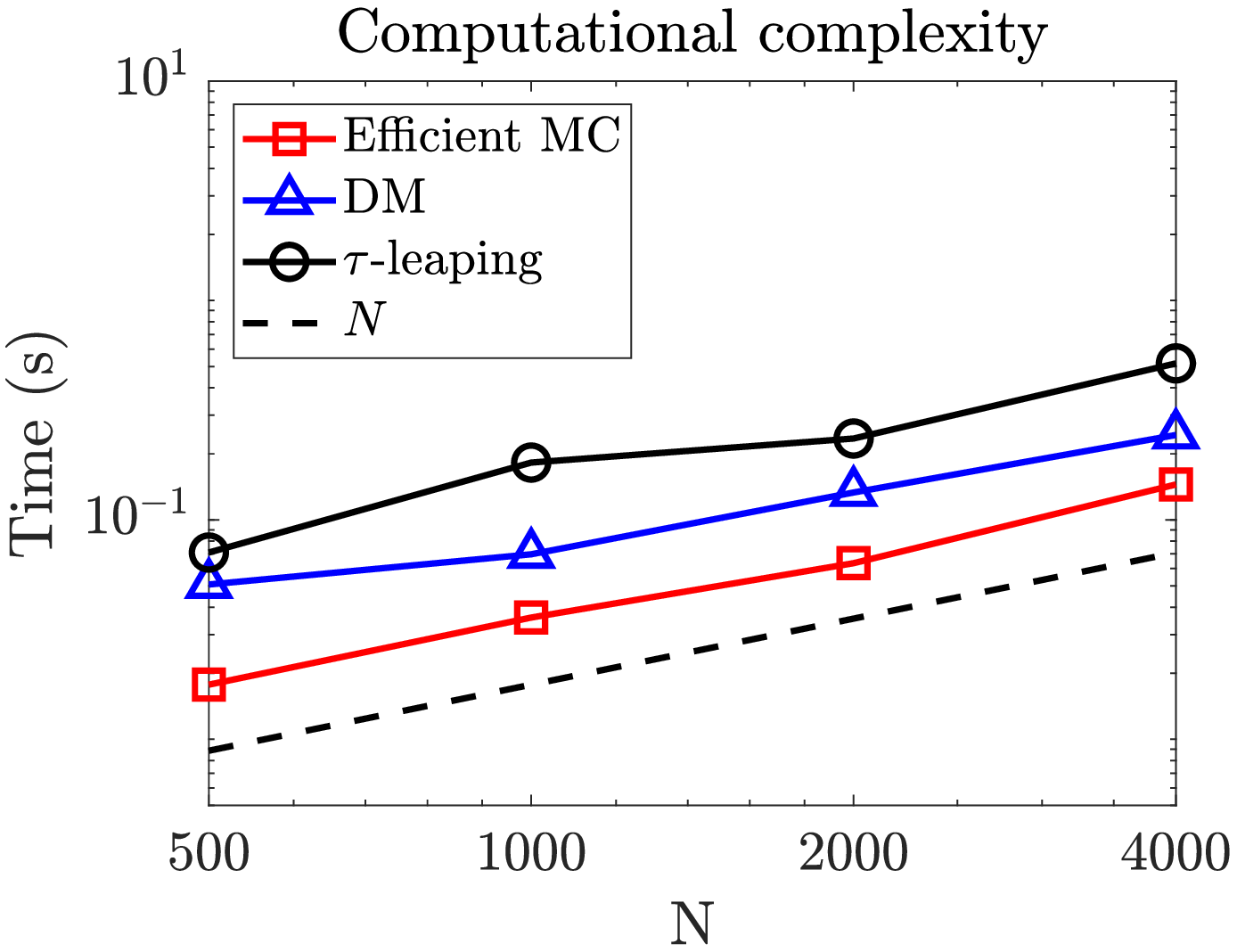}
		\includegraphics[width=0.49\linewidth]{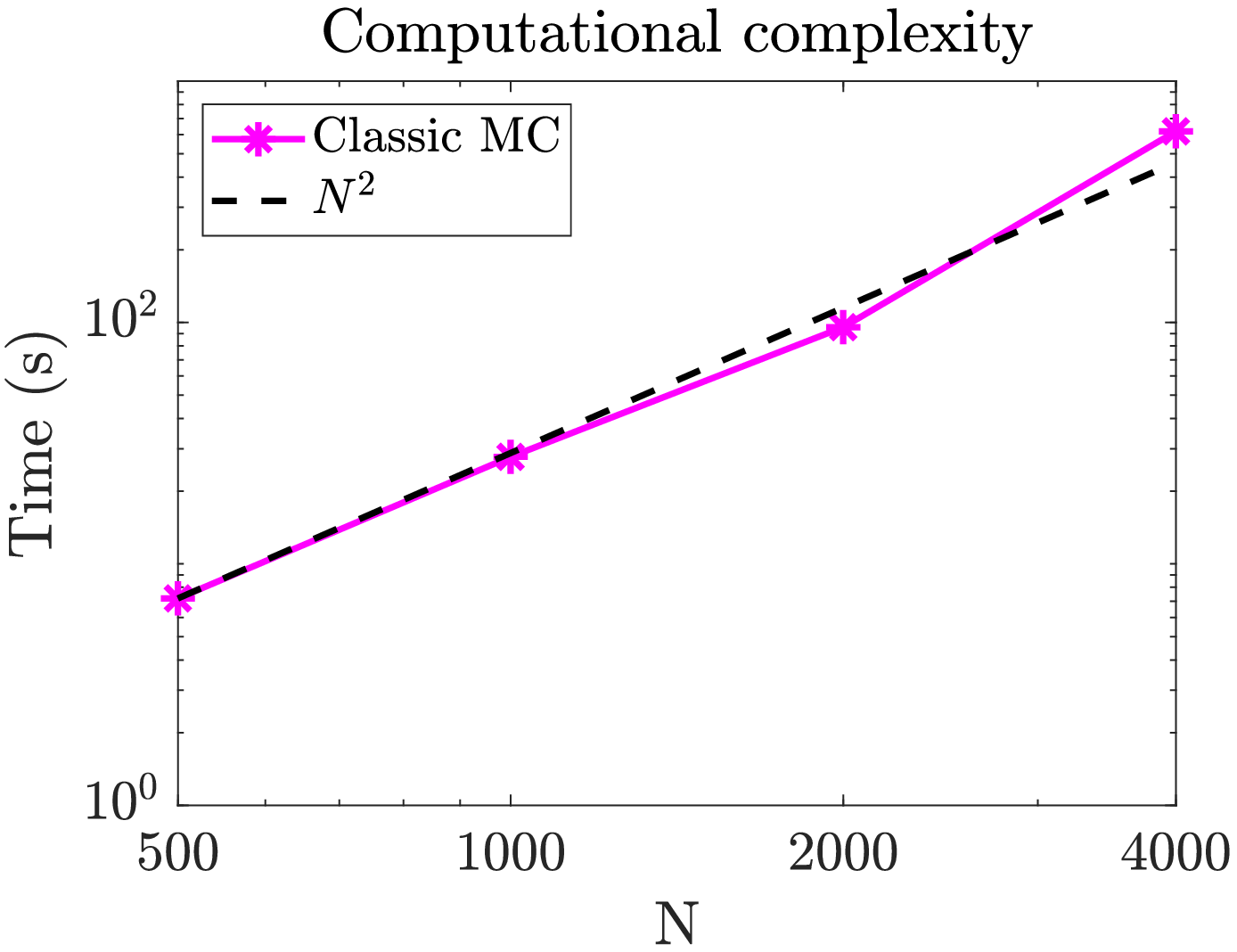}
		\caption{Homogeneous predator-prey model: computational complexity of efficient Monte Carlo algorithm, direct method and $\tau$-leaping method (on the left),  and classic Monte Carlo algorithms (on the right) as $N$ varies.The dynamics in \eqref{eq:1a}-\eqref{eq:1b} is simulated for $N=500,\ldots,4000$, $\mu=0.5$, $b^r=1$, $d_1^r=d_2^r=0.3$, $p_{1}^r=p_{2}^r=0.5$. Markers represent the computational costs relative to the parameters choice indicated. }
		\label{fig:figure8bis}
	\end{figure}
	Figure \ref{fig:figure9} shows the comparison between the computational costs of the stochastic algorithms in the one dimensional heterogeneous case. The dynamics in \eqref{eq:5a}-\eqref{eq:5b}-\eqref{eq:5c} is simulated for $N_c=100$, $q_1=0.3$, $q_2=0.3$ and $b^r=0.1$, $d_1^r=0.1$, $d_2^r=0$, $p_1^r=0.25$, $p_2^r=0.05$, letting the migration parameters $m_1^r$, $m_2^r$ to vary between $0.1$ and $0.9$. The computational cost of the efficient Monte Carlo algorithm is lower than the one of classic algorithms.
	\begin{figure}[H]
		\centering
		\includegraphics[width=0.49\linewidth]{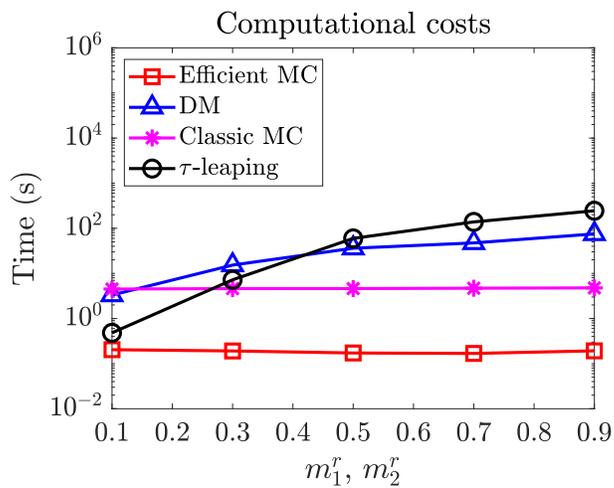}
		\caption{Heterogeneous one dimensional predator-prey model: computational cost of efficient and classic Monte Carlo algorithms, direct method and $\tau$-leaping method as the migration parameters vary, for $N_c$ fixed. The dynamics in \eqref{eq:5a}-\eqref{eq:5b}-\eqref{eq:5c} is simulated for $N_c=100$, $q_1=0.3$, $q_2=0.3$, $b^r=0.1$, $d_1^r=0.1$, $d_2^r=0$, $p_1^r = 0.25$, $p_2^r=0.05$, $m_{1}^r=m_{2}^r=0.1,\ldots,0.9$. Markers represent the computational costs relative to the parameters choice indicated. }
		\label{fig:figure9}
	\end{figure}
	Figure \ref{fig:figure10} shows the computational cost of the stochastic algorithms for different values of the migration rates as $N_c$ varies. On the left the computational costs for fixed migration rates $m_1^r=m_2^r=0.1$ and on the right for $m_1^r=m_2^r=0.9$. Note that the computational cost of the efficient Monte Carlo algorithm is always lower than the one of the classic algorithms.
	\begin{figure}[h!]
		\centering
		\includegraphics[width=0.49\linewidth]{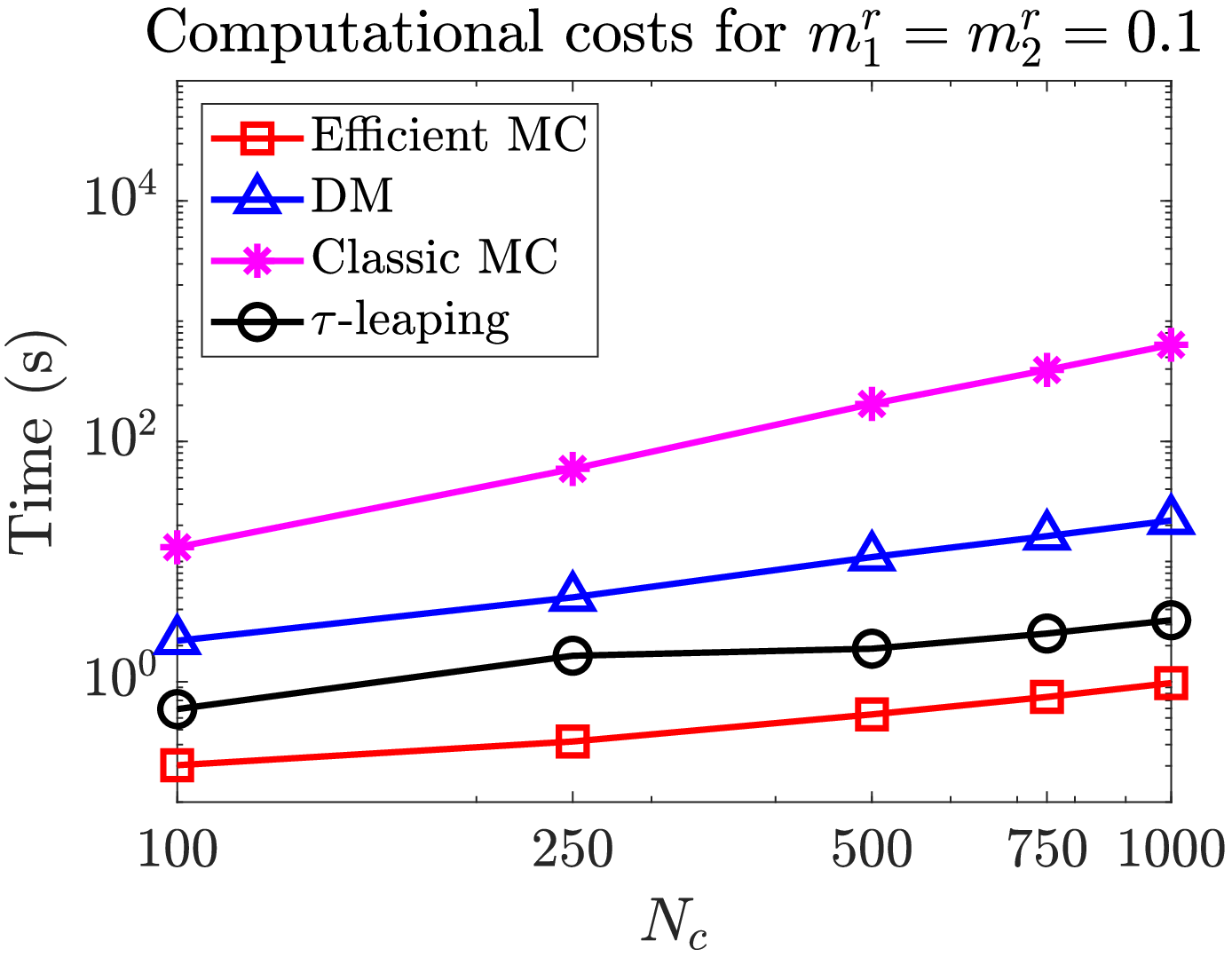}
		\includegraphics[width=0.49\linewidth]{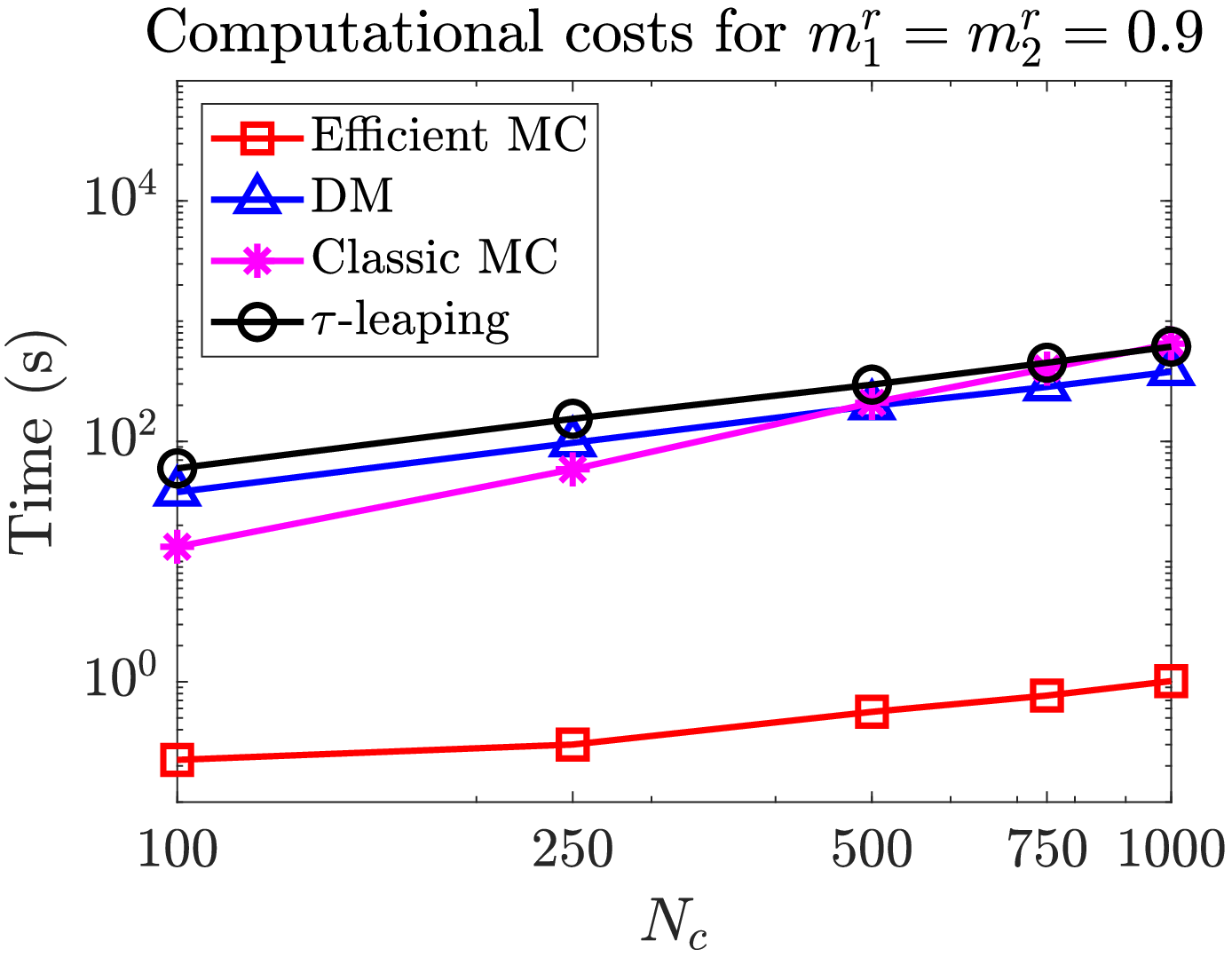}
		\caption{Heterogeneous one dimensional predator-prey model: computational cost of efficient and classic Monte Carlo algorithms, direct method and $\tau$-leaping method as $N_c$ varies. The dynamics in \eqref{eq:5a}-\eqref{eq:5b}-\eqref{eq:5c} is simulated for $N_c=100,\ldots,1000$, $q_1=0.3$, $q_2=0.3$,  $b^r=0.1$, $d_1^r=0.1$, $d_2^r=0$, $p_1^r = 0.25$, $p_2^r=0.05$, $m_{1}^r=m_{2}^r=0.1$(on the left), $m_1^r=m_2^r=0.9$ (on the right). Markers represent the computational costs relative to the parameters choice indicated.}
		\label{fig:figure10}
	\end{figure}
	In Figure \ref{fig:figure11} a comparison between the computational costs of the stochastic algorithms in the two dimensional heterogeneous case. The dynamics is simulated  for $N_c=50$ fixed, $q_1=0.3$, $q_2=0.3$,  $b^r=0.1$, $d_1^r=0.1$ $d_2^r=0$, $p_1^r=0.25$, $p_2^r=0.05$, letting the migration parameters to vary between $0.1$ and $0.9$. Also in this case the efficient Monte Carlo algorithm has a computational cost lower than the one of classic algorithms.
	\begin{figure}[H]
		\centering
		\includegraphics[width=0.49\linewidth]{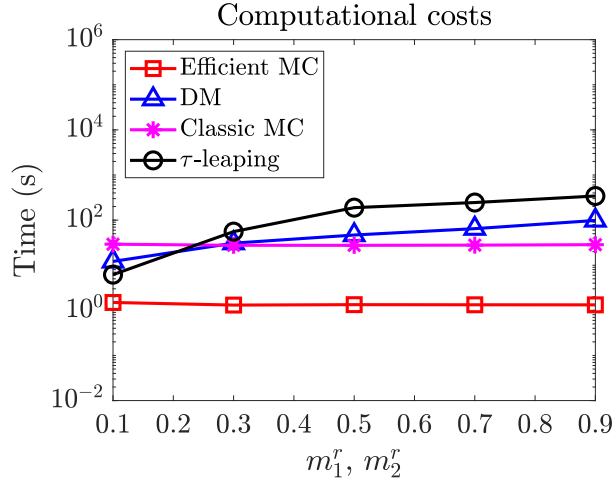}
		\caption{Heterogeneous two dimensional predator-prey model: computational cost of efficient and classic Monte Carlo algorithms, direct method and $\tau$-leaping method as the migration parameters vary, for $N_c$ fixed. The dynamics in \eqref{eq:5a}-\eqref{eq:5b}-\eqref{eq:5c} is simulated for $N_c=50$, $q_1=0.3$, $q_2=0.3$, $b^r=0.1$, $d_1^r=0.1$, $d_2^r=0$, $p_1^r = 0.25$, $p_2^r=0.05$, $m_{1}^r=m_{2}^r=0.1,\ldots,0.9$. Markers represent the computational costs relative to the parameters choice indicated. }
		\label{fig:figure11}
	\end{figure}
	%%%%%%%%%%%%%%%%%%%%%%%%%%%%%%%%%%%%%%%%%%%%%%%%
	%%%%%%%%%%%%%%%%%%%%%%%%%%%%%%%%%%%%%%%%%%%%%%%
		\subsection{Test 3: Accuracy \& performances}
		In the following we will compare the accuracy of the direct method with the one of the efficient Monte Carlo algorithm and of the $\tau$-leaping method. We first focus on the homogeneous case and we define the errors as \begin{equation}\label{eq:errorH_DM} 
			E_f^N = \Vert  f^N_{DM}(t)-f^N \Vert_{\infty},\quad 	E_g^N = \Vert  g^N_{DM}(t)-g^N \Vert_{\infty},
		\end{equation}
		where $f^N_{DM}(t)$, $g^N_{DM}(t)$ denote the simulations obtained with the direct method and $f^N(t)$, $g^N(t)$ denote the simulations obtained either with the efficient Monte Carlo or with the $\tau$-leaping algorithms for the predators and preys populations, respectively. Figure	\ref{fig:figure_error}
		shows the errors $E_f^N$, $E_g^N$ as $N$ varies and for fixed birth, competition and death parameters (see Table \ref{tab:all_parameters}). Note that the error related to the simulations obtained with the efficient Monte Carlo algorithm is lower than the one related to the simulated $\tau$-leaping solutions for almost every $N$. Note also that both errors decrease as the sample size increases reaching a constant value when $N$ is large enough.
		\begin{figure}[H]
			\centering
			\includegraphics[width=0.49\linewidth]{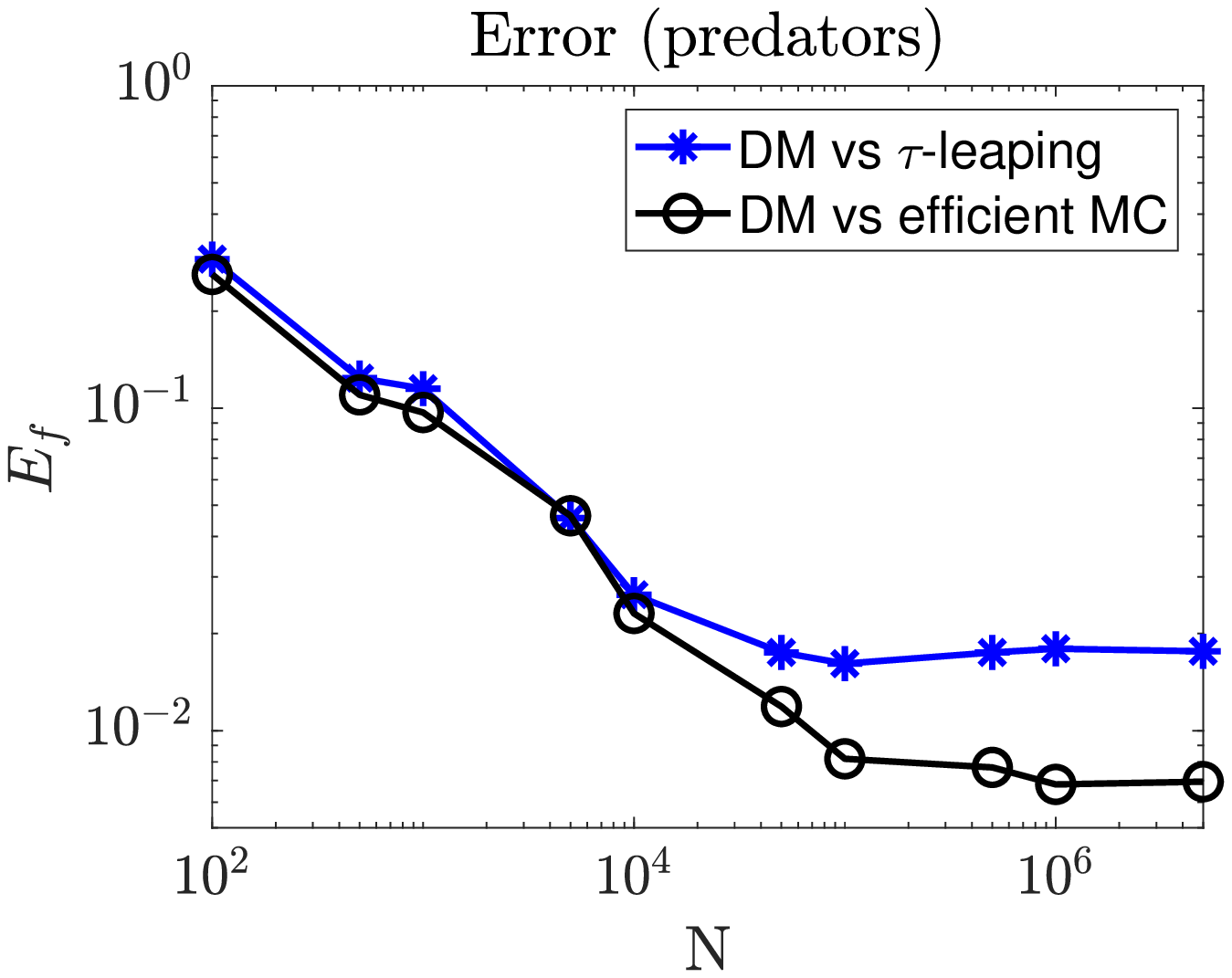}
			\includegraphics[width=0.49\linewidth]{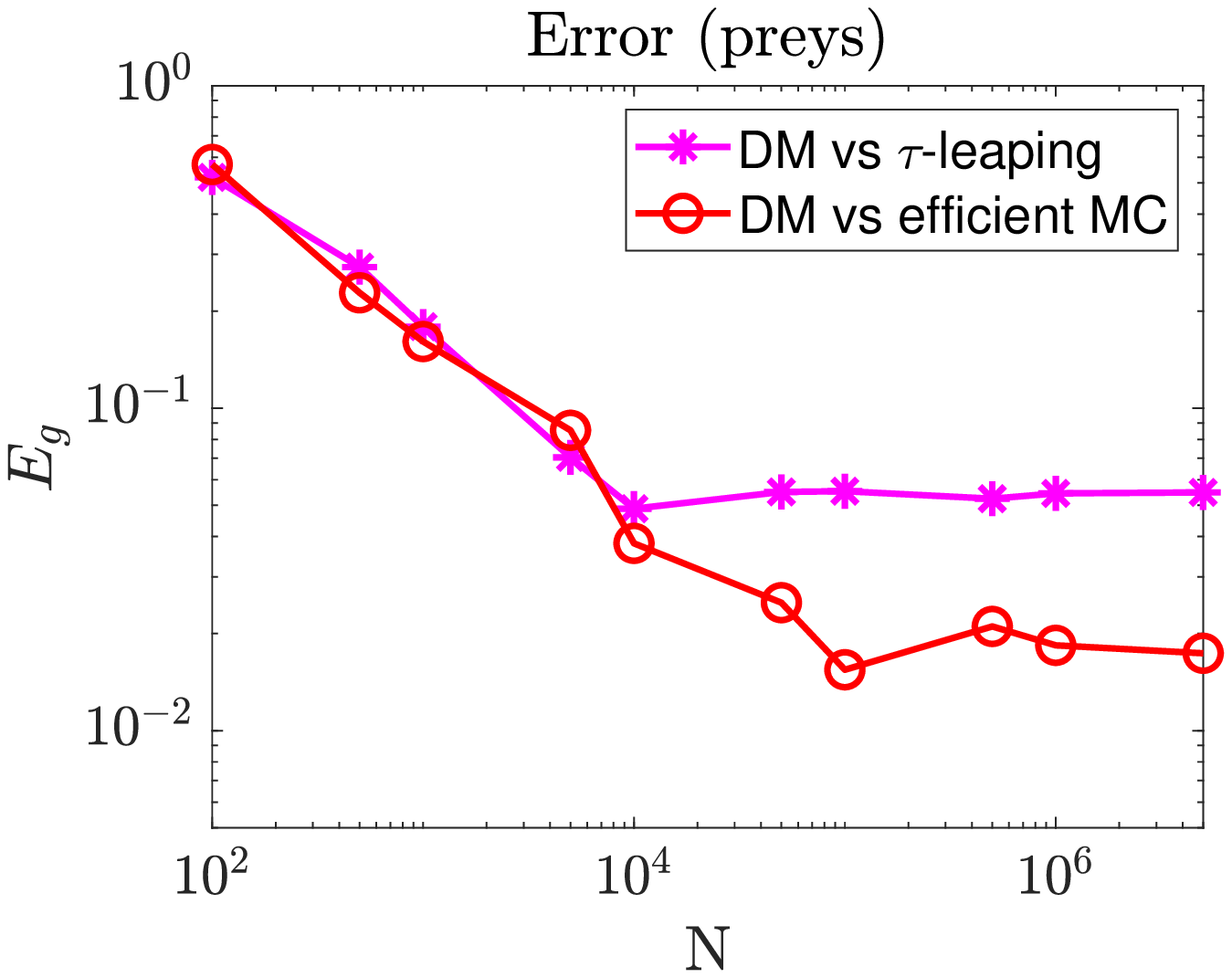}
			\caption{Homogeneous predator-prey model: accuracy of the efficient Monte Carlo and of the $\tau$-leaping algorithms with respect to the direct method as $N$ varies. The dynamics in \eqref{eq:1a}-\eqref{eq:1b} is simulated for $N=10^2,\ldots, 5 \times 10^6$, $\mu=0.5$, $b^r=0.1$, $d_1^r=0.1$, $d_2^r=0$, $p_1^r = 0.25$, $p_2^r=0.05$. Markers correspond to the values $E_f^N, E_g^N$ computed as in equations \eqref{eq:errorH_DM} for a fixed $N$. }
			\label{fig:figure_error}
		\end{figure} The same results can be obtained in the heterogeneous case with the parameters choice reported in Table \ref{tab:all_parameters}, as show in Figure \ref{fig:figure_error_NH}. Here, for simplicity we focus on the one dimensional case and we define the errors as 
		\begin{equation}\label{eq:errorNH_DM}
			\begin{split}
				&E_f^{N_c} = \left\langle  \max_{t} \vert f^{N_c}(x,t) -f^{N_c}(x,t) \vert \right\rangle_x,\quad E_g^{N_c} = \left\langle  \max_{t} \vert g^{N_c}(x,t) -g^{N_c}(x,t) \vert \right\rangle_x,
			\end{split}
		\end{equation} where $\left\langle \cdot \right\rangle $ denotes the expected value with respect to $x$.	\begin{figure}[H]
			\centering
			\includegraphics[width=0.49\linewidth]{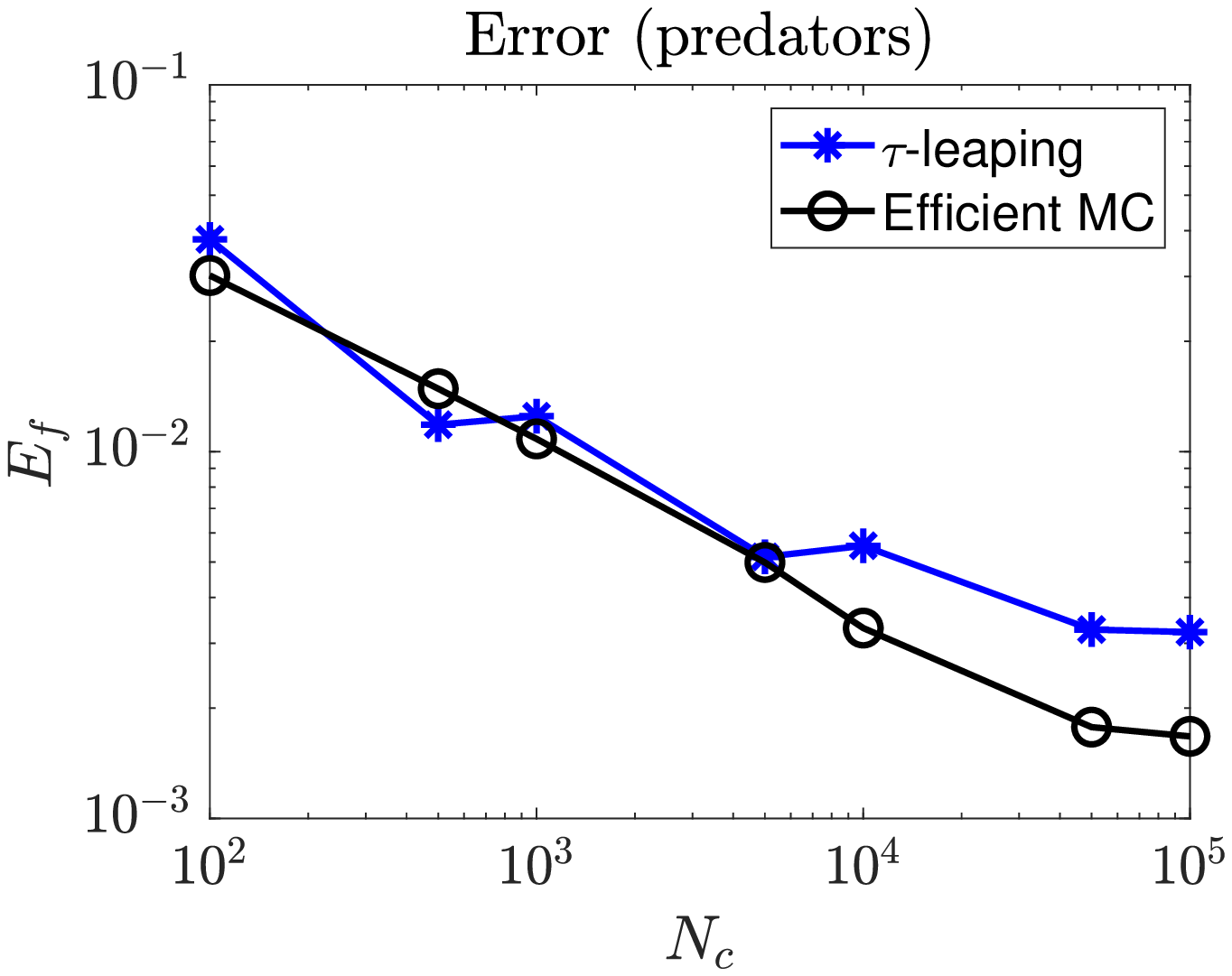}
			\includegraphics[width=0.49\linewidth]{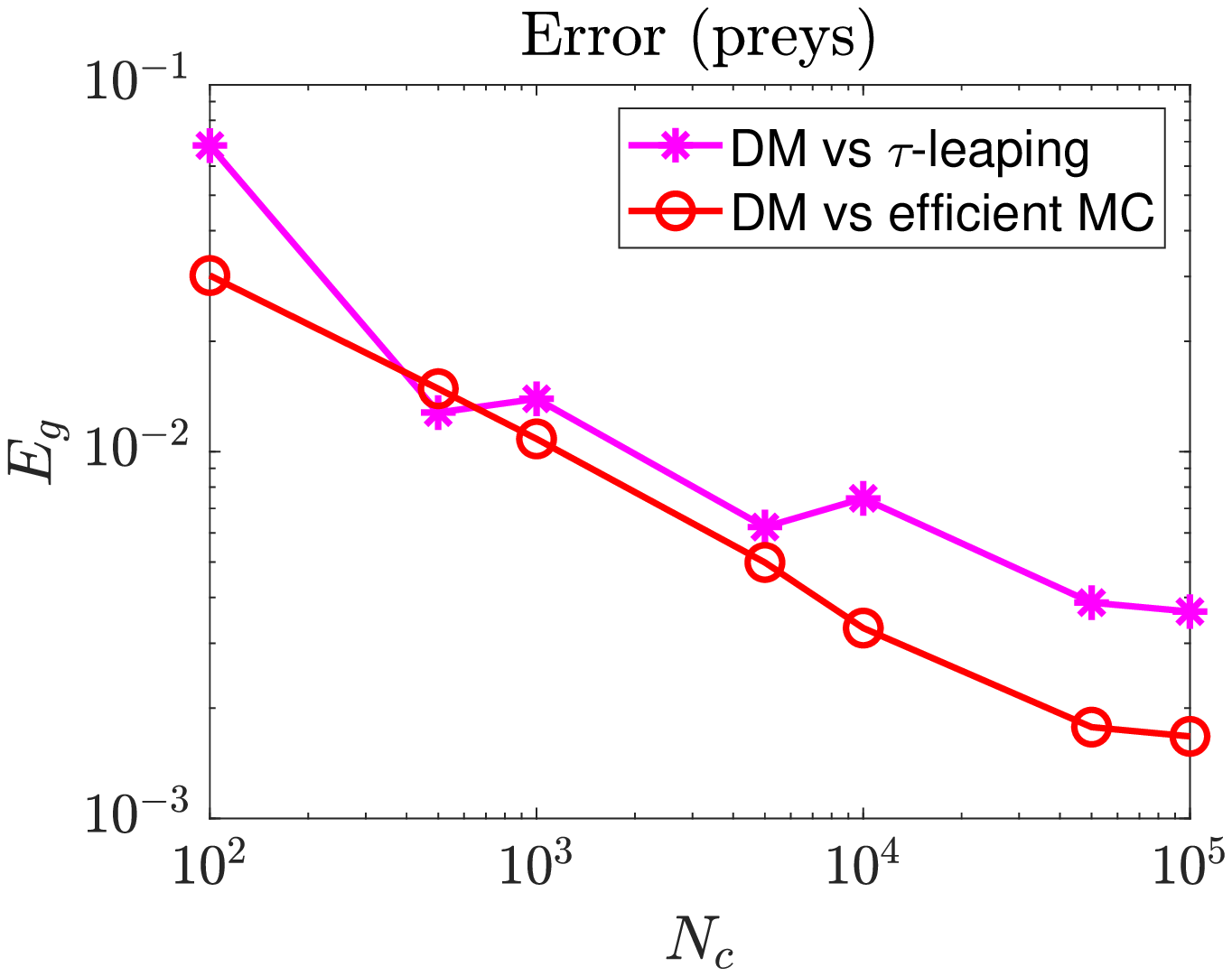}
			\caption{Heterogeneous predator-prey model: accuracy of the efficient Monte Carlo and of the $\tau$-leaping algorithms with respect to the direct method as $N_c$ varies. The dynamics in \eqref{eq:5a}-\eqref{eq:5b}-\eqref{eq:5c} is simulated for $N_c=10^2,\ldots, 10^5$, $q_1=q_2=0.3$, $b^r=0.1$, $d_1^r=0.1$, $d_2^r=0$, $p_1^r = 0.25$, $p_2^r=0.05$, $m_1=0.5$, $m_2=0.5$. Markers correspond to the values $E_f(N_c), E_g(N_c)$ computed as in equations \eqref{eq:errorNH_DM} for a fixed $N_c$. }
			\label{fig:figure_error_NH}
		\end{figure} 
		Figure \ref{fig:figure_tradeoff} 
		shows the trade off between the error $E_f$ and the computational costs of the efficient Monte Carlo algorithm and of the $\tau$-leaping method in both the homogeneous and heterogeneous case for the predators dynamics. Note that both errors decrease as the computational time increases, that is as the sample size increases.  Note also that in the homogeneous case the computational cost is lower but the error is higher than in the heterogeneous case. Here, for simplicity we focus on the predators dynamics but the same results can be obtained if we focus on the preys dynamics.
		\begin{figure}[H]
			\centering
			\includegraphics[width=0.49\linewidth]{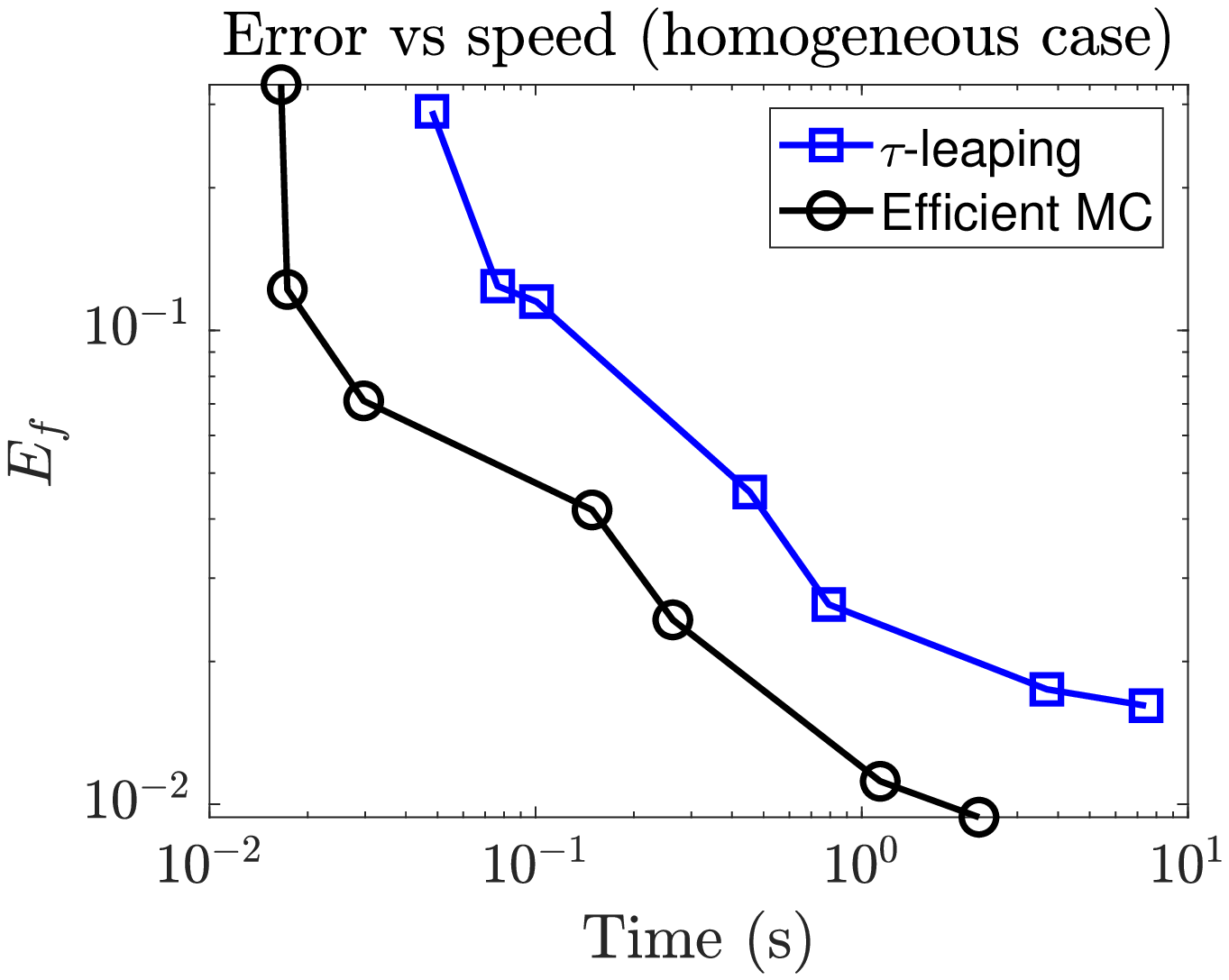}
			\includegraphics[width=0.49\linewidth]{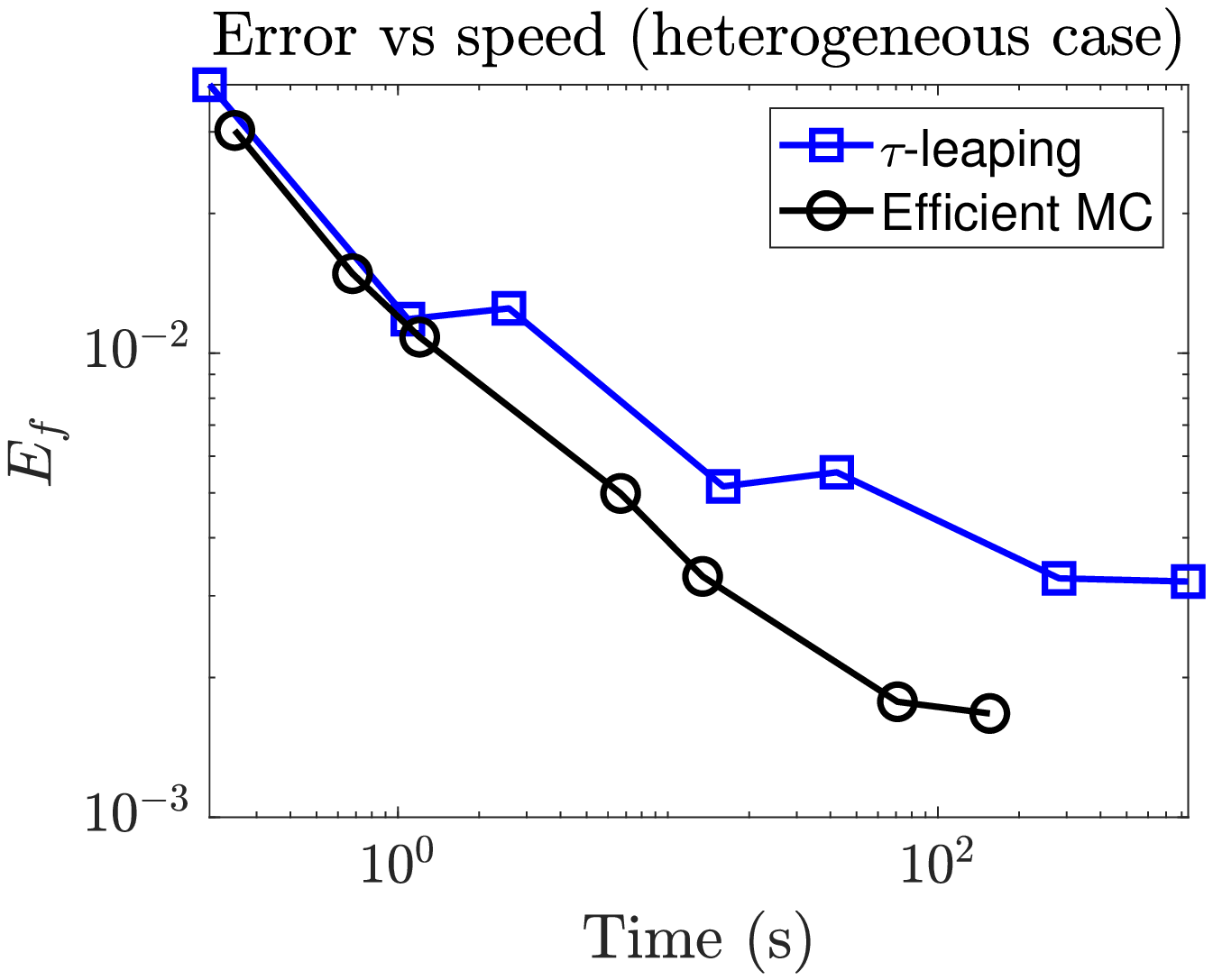}
			\caption{Homogeneous and heterogeneous one dimensional predator-prey model: trade-off of error $E_f^{N_c}$ computed as in equation \eqref{eq:errorH_DM}-\eqref{eq:errorNH_DM} vs computational cost of the efficient Monte Carlo and of the $\tau$-leaping algorithms as the computational cost varies. On the left the trade-off in the  homogeneous case, on the right in the heterogeneous case. Markers correspond to the values $E_f^{N_c}$ in relation to the computational costs. }
			\label{fig:figure_tradeoff} 
	\end{figure} 
	
	%%%%%%%%%%%%%%%%%%%%%%%%%%%%%%%%%%%%%%%%%%%%%
	%%%%%%%%%%%%%%%%%%%%%%%%%%%%%%%%%%%%%%%%%%%%%
	\subsection{Test 4: Stochastic persistency}\label{section5}
	We have observed in Figure \ref{fig:figure1} that both mean-field solutions and stochastic simulations present an oscillatory behavior. However, the nature of the oscillations seems to be different. Following the idea in \cite{mckane2005predator}, we can prove that oscillations in predators and preys stochastic simulations should be of order $1/\sqrt{N}$, as the error of the Monte Carlo method, but indeed are amplified by a large factor due to noise resonant effect. 
	
		In order to analyze the nature of oscillations we introduce the following linearized model corresponding to a perturbation of the mean-field dynamics \eqref{eq:A3b} in absence of spatial diffusion, see also Appendix \ref{app:homo},
		\begin{equation} \label{eq:C1b} 
			\frac{d}{dt}	\begin{bmatrix}
				{f^N} \\ 
				{g^N}
			\end{bmatrix} = \Psi \begin{bmatrix}
				f^N \\ 
				g^N 
			\end{bmatrix} + \Phi \xi, 
		\end{equation}
		where $\Psi$ is the Jacobian matrix of the mean-field equations evaluated at the equilibrium $(f^*,g^*)$, defined in \eqref{equilibrium}, $\Phi$ corresponds to the following matrix
		\[
		\Phi= \begin{bmatrix}
			0 &  \sqrt{2\tilde{p}_1^rf^*g^*} & 0 & -\sqrt{\tilde{d}_1^rf^*}  \\
			\sqrt{2\tilde{b}^rg^*(1-f^*-g^*)} & -\sqrt{2\tilde{p}_1^rf^*g^*} & -\sqrt{2\tilde{p}_2^rf^*g^*+\tilde{d}_2^rg^*} & 0
		\end{bmatrix},
		\] 
		associated to the white noise vector $\xi=[\xi_1,\xi_2,\xi_3,\xi_4]^T$, with $\xi_i$ i.i.d distributed random numbers with zero mean and variance $1/\sqrt{N}$, \cite{gillespie2000chemical},\cite{toral2014stochastic}.
		We focus on the predator density, we compute the Fourier transform of the mean-field density $f(t)$, i.e. $\tilde f(\omega)$, and we derive the power spectrum
		\begin{equation}\label{eq:spectrum2}
			P (\omega) = \left\langle\vert \tilde{f}(\omega)\vert^2\right\rangle =\frac{\Theta + \Lambda \omega^2}{(\omega^2-\Omega_0^2)^2+\Gamma^2\omega^2},
		\end{equation} 
		where $\left\langle\cdot\right\rangle$ is the expected value w.r.t. $\xi$. The parameters $\Theta, \Lambda,\Omega_0,\Gamma$ are defined as follows
		\begin{align*}
			\Theta &= \left\langle[\Psi_{12}(\Phi_{21}\xi_1+\Phi_{22}\xi_2+\Phi_{23}\xi_3)-\Psi_{22}(\Phi_{12}\xi_2+\Phi_{14}\xi_4)]^2\right\rangle
			\cr
			\Lambda &= \left\langle[(\Phi_{12}\xi_2+\Phi_{14}\xi_4]^2\right\rangle,\quad\Omega_0^2=\Psi_{12}\vert\Psi_{21}\vert,\quad \Gamma = \vert\Psi_{22}\vert.
		\end{align*}
		%We refer to \ref{app:homo} for further the details on 
	
	Note that $P(\omega)$ resembles the one of a damped harmonic oscillator, \cite{hauer2015nonlinear}.   As shown in Figure \ref{fig:noise} on the left, the solution $f(t)$ of the system of equations \eqref{eq:C1b} without the influence of noise behaves as an under-damped harmonic oscillator. The damped term is $\Gamma<1$ and oscillations decrease exponentially as $\tilde{A}\exp\left (-\frac{\Gamma}{2} t\right)$ where $\tilde{A}$ is the maximum oscillations amplitude. This behavior recalls the one of the mean-field solutions. If we introduce the noise term, we see in Figure \ref{fig:noise} on the right that the oscillations amplitude is influenced both by damped and resonant effects and the oscillations become sustained as the one of the stochastic simulations shown in Section \ref{section4}. Hence the resonant effect is a consequence of the white noise that is not an external factor but it is due to the stochasticity of birth, death and competitions events.
	\begin{figure}[H]
		\centering
		\includegraphics[width=0.49\linewidth]{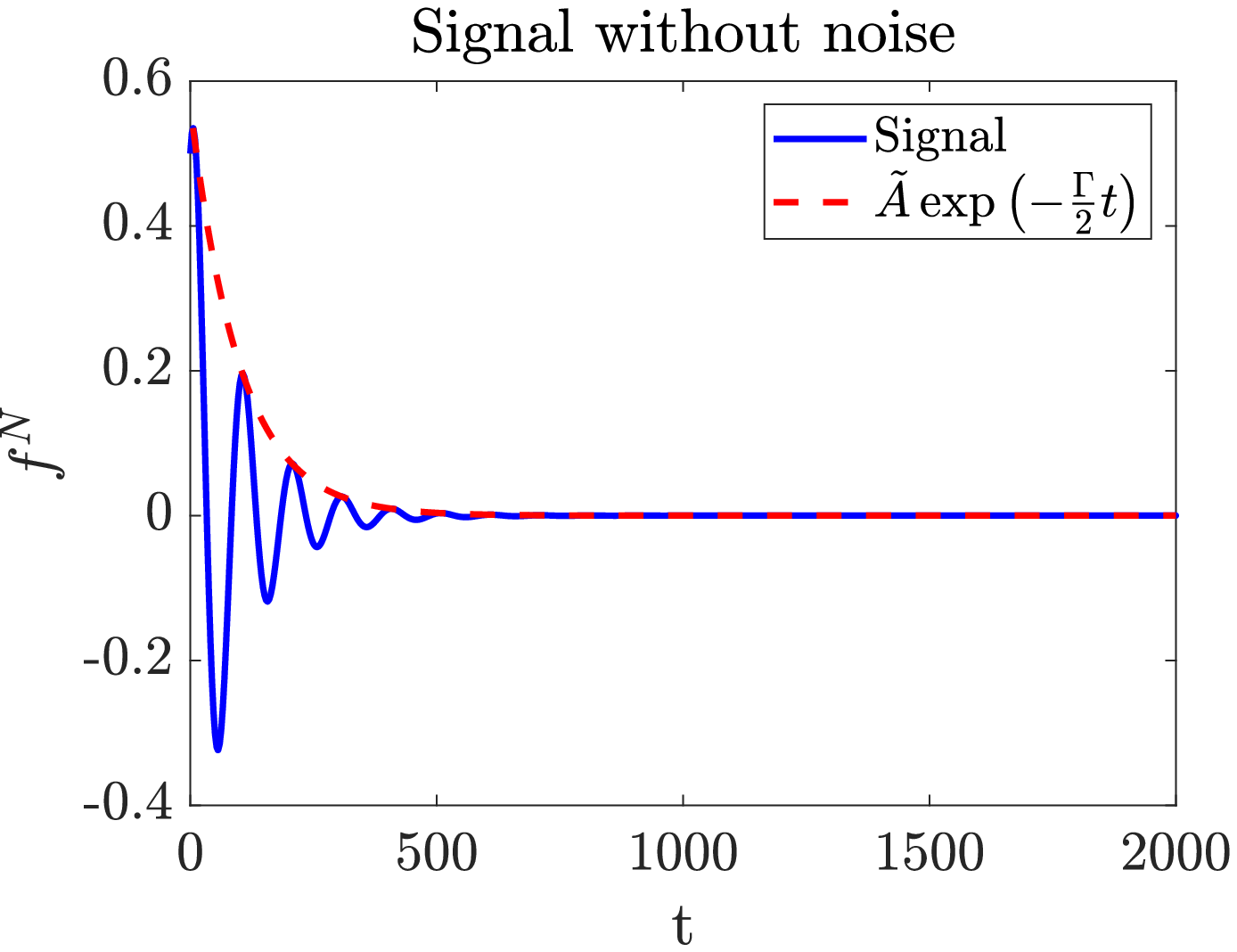}
		\includegraphics[width=0.49\linewidth]{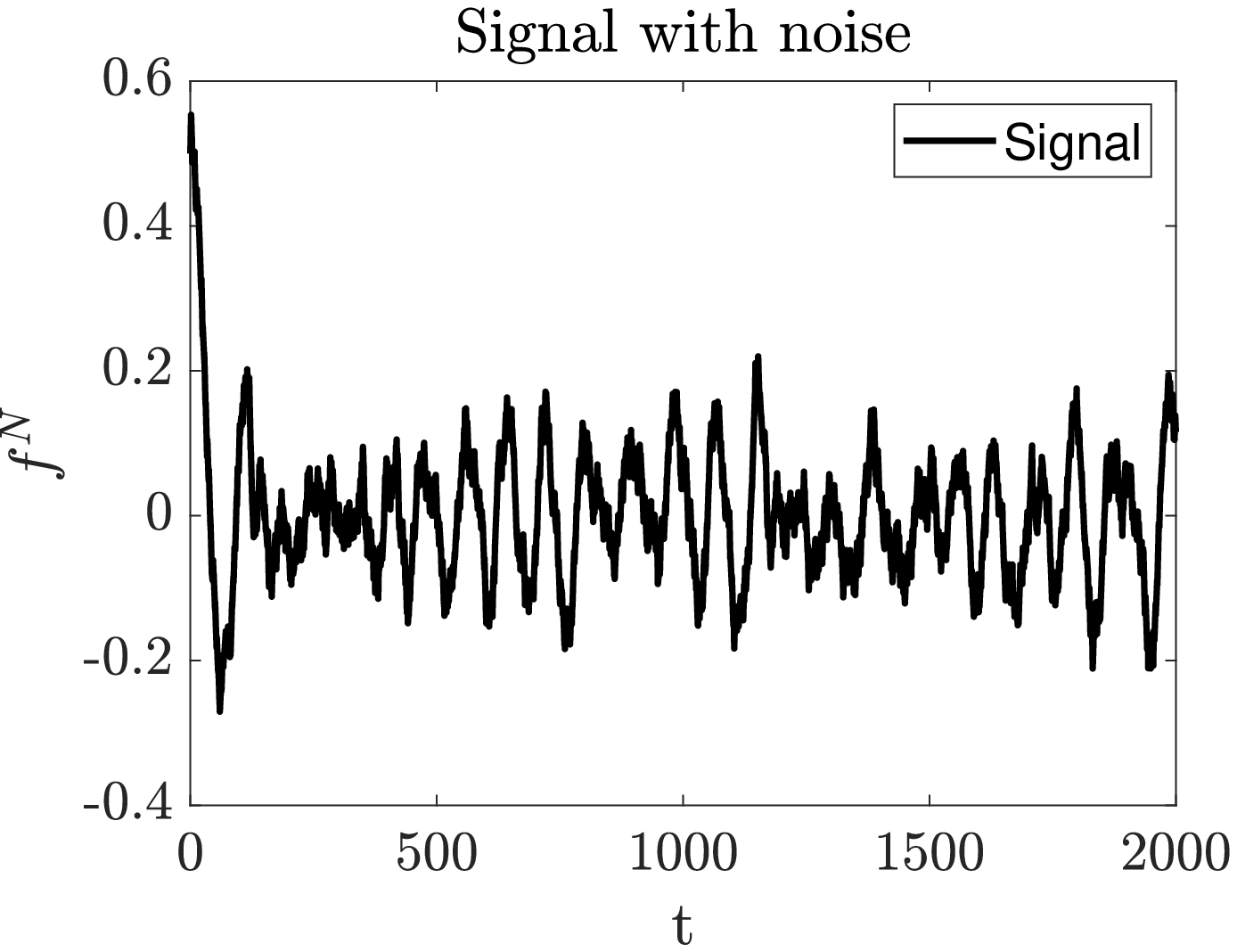}
		\caption{Predators density without noise (on the left) and with noise (on the right) as solution of equations \eqref{eq:C1b}. }
		\label{fig:noise}
	\end{figure}
In Figure	\ref{fig:spectrumH} the power spectrum computed as in equation \eqref{eq:spectrum2} and the one obtained averaging the results of 500 simulations of the processes described in \eqref{eq:1a}-\eqref{eq:1b} computed with the efficient Monte Carlo algorithm and with the direct method for different values of $N$ and for the parameters choice reported in the first line of Table \ref{tab:all_parameters}. Note that the two simulated power spectra agree with the one computed as in equation \eqref{eq:spectrum2}, especially when $N$ is large. Note also that the amplitude of the oscillations decreases as $N$ increases. 
		\begin{figure}[H]
			\centering
			\includegraphics[width=0.49\linewidth]{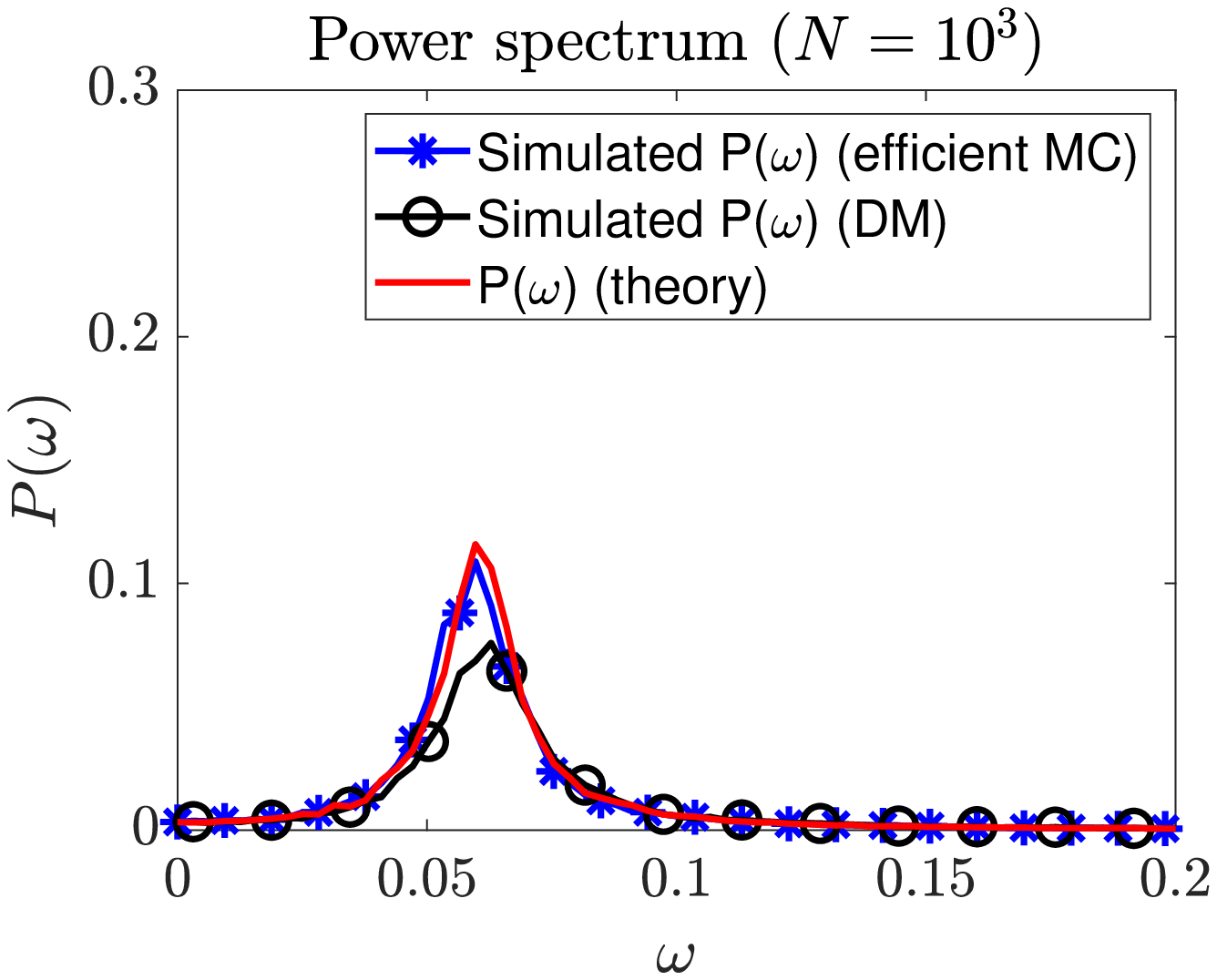}
			%	%\includegraphics[width=0.327\linewidth]{figure/spectrumH_predators_N1e4}		
			\includegraphics[width=0.49\linewidth]{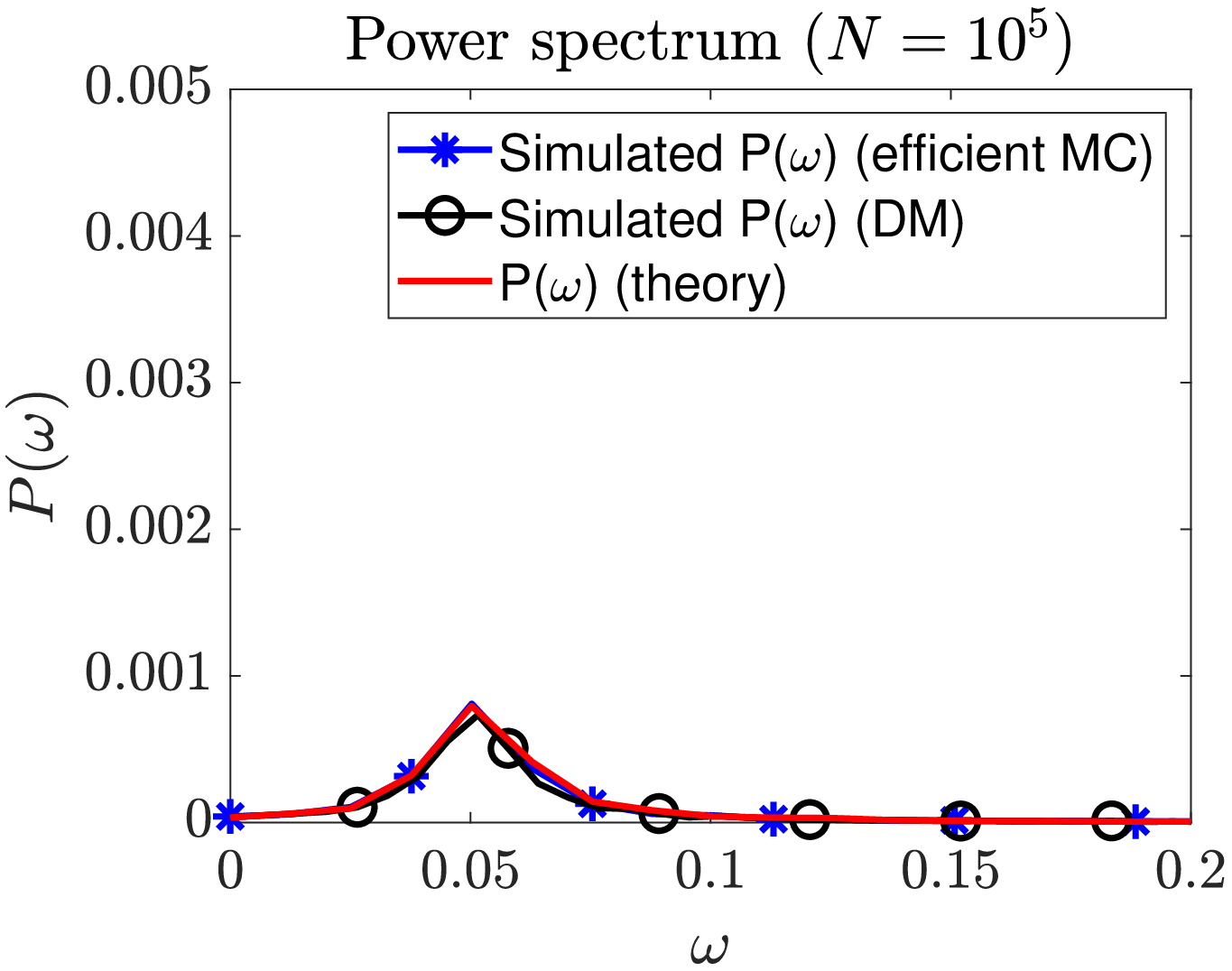}
			\caption{Homogeneous case: power spectrum computed as in equation \eqref{eq:spectrum2} in red, averaging 500 simulations of the processes described in \eqref{eq:1a}-\eqref{eq:1b} computed with the efficient Monte Carlo algorithm in blue and with the direct method in black. On the left $N=10^3$ and on the right $N=10^5$. Markers have been added just to indicate different lines. }
			\label{fig:spectrumH}
		\end{figure}
	
	The same results can be obtained in the heterogeneous case in each cell $C_\ell$, $\ell=1,\ldots,M_c$, as the time $t$ varies. In Figure \ref{fig:figureC2}, we see an example of persistent oscillatory behavior in stochastic simulations compared with the damped one of the mean-field solutions in cell $C_\ell$, $\ell=25,50$ for $N_c=100$, $M_c=100$.   
	\begin{figure}[H]
		\centering
		\includegraphics[width=0.49\linewidth]{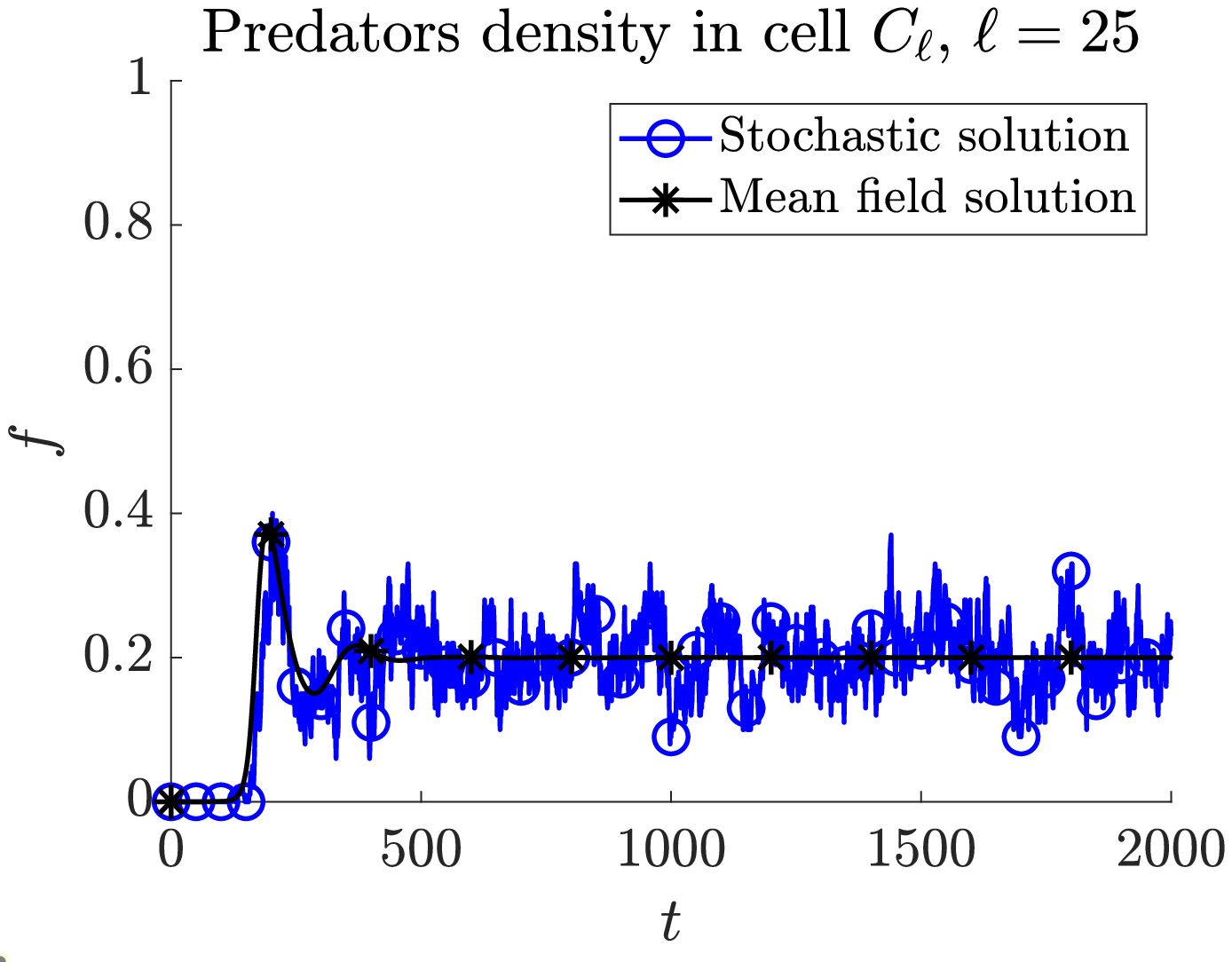}
		\includegraphics[width=0.49\linewidth]{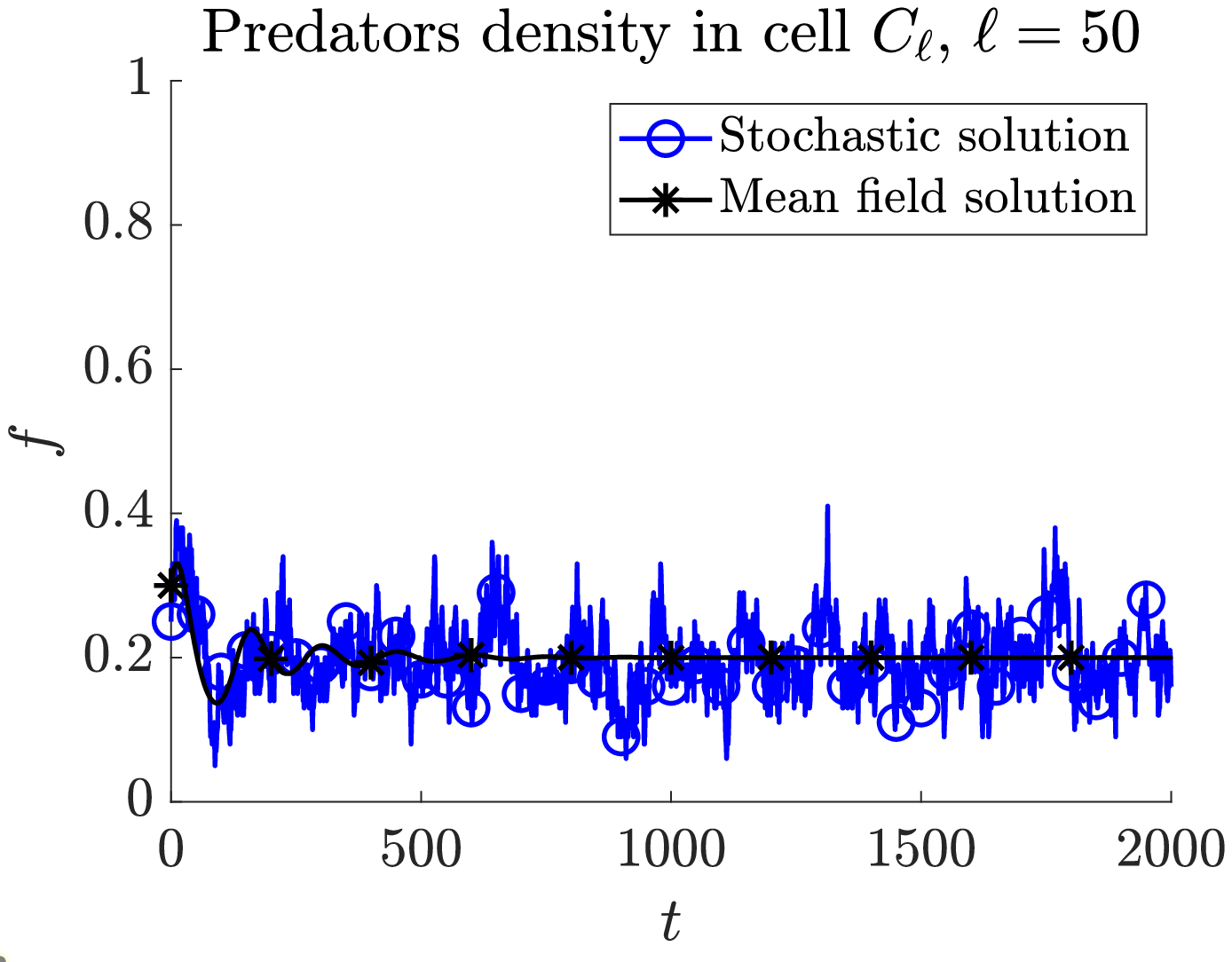}
		\caption{Heterogeneous case: time evolution of the predators density in cell $C_\ell$, $\ell=25,50$. Parameters: $N_c=100$, $b^r=0.1$, $d_1^r=0.1$, $d_2^r=0$, $p_1^r=0.25$, $p_2^r=0.05$, $m_1^r=m_2^r=0.5$, cell $C_\ell$, $\ell=25,50$.   Markers have been added just to indicate different lines.}
		\label{fig:figureC2}
	\end{figure}
	In Figure \ref{fig:spectrumNH} on the left the power spectrum computed averaging the results of 500 simulations of the processes described in \eqref{eq:5a}-\eqref{eq:5b}-\eqref{eq:5c} in cell $C_\ell$, $\ell=25,50$ with the efficient Monte Carlo algorithm for the parameters choice reported in the second line of Table \ref{tab:all_parameters} for $N_c=100$ and $T=2000$.	 On the right, the average between the power spectra computed in cell $C_\ell$, $\ell=1,\ldots,M_c$ obtained with the efficient Monte Carlo algorithm and with the direct method. Note that also in the heterogeneous case the power spectrum recalls the one of a damped harmonic oscillator. Note also that the two average power spectra agree and so that the oscillations that characterize the simulations obtained with the two algorithms are of the same nature.   
		\begin{figure}[H]
			\centering
			\includegraphics[width=0.49\linewidth]{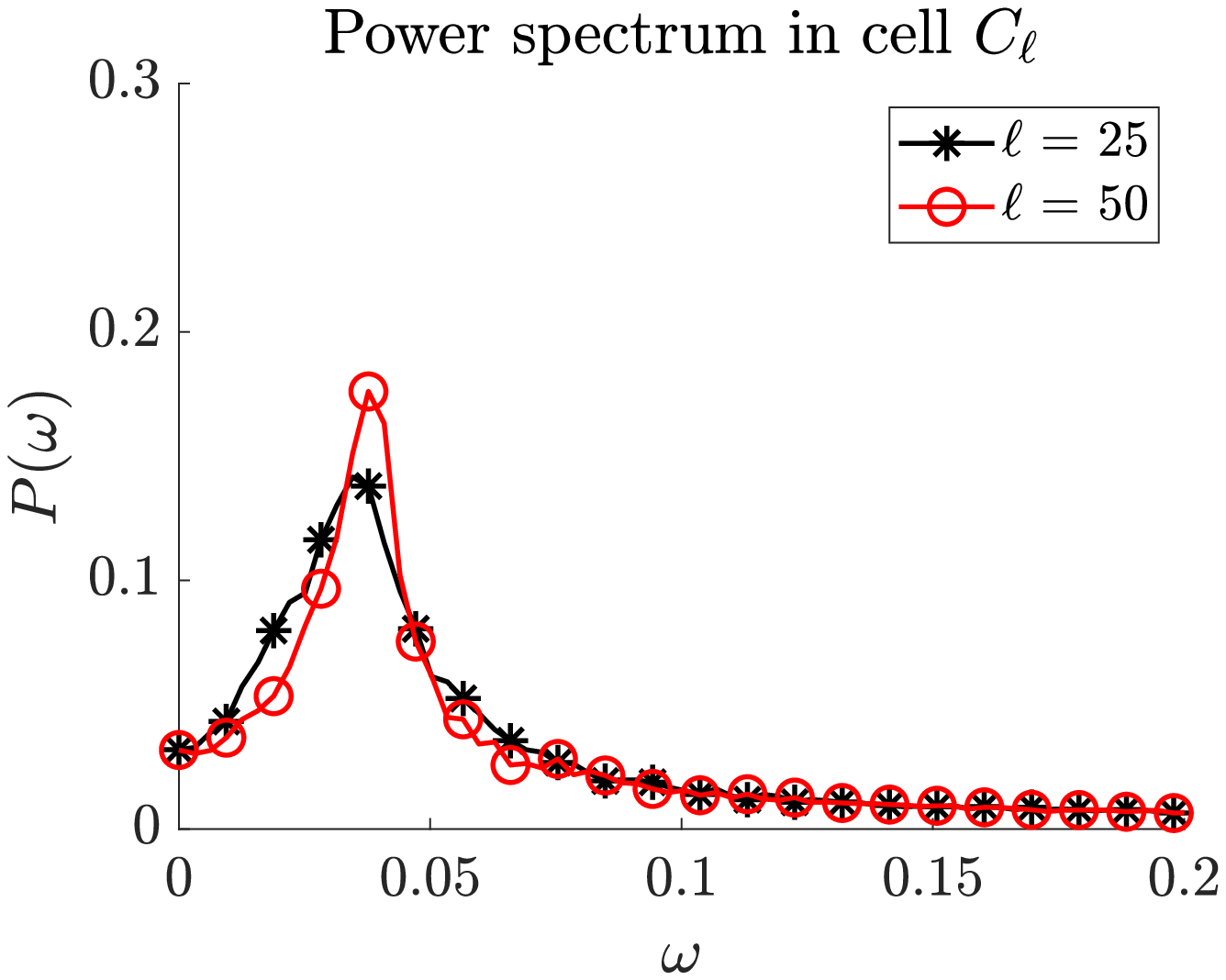}
			\includegraphics[width=0.49\linewidth]{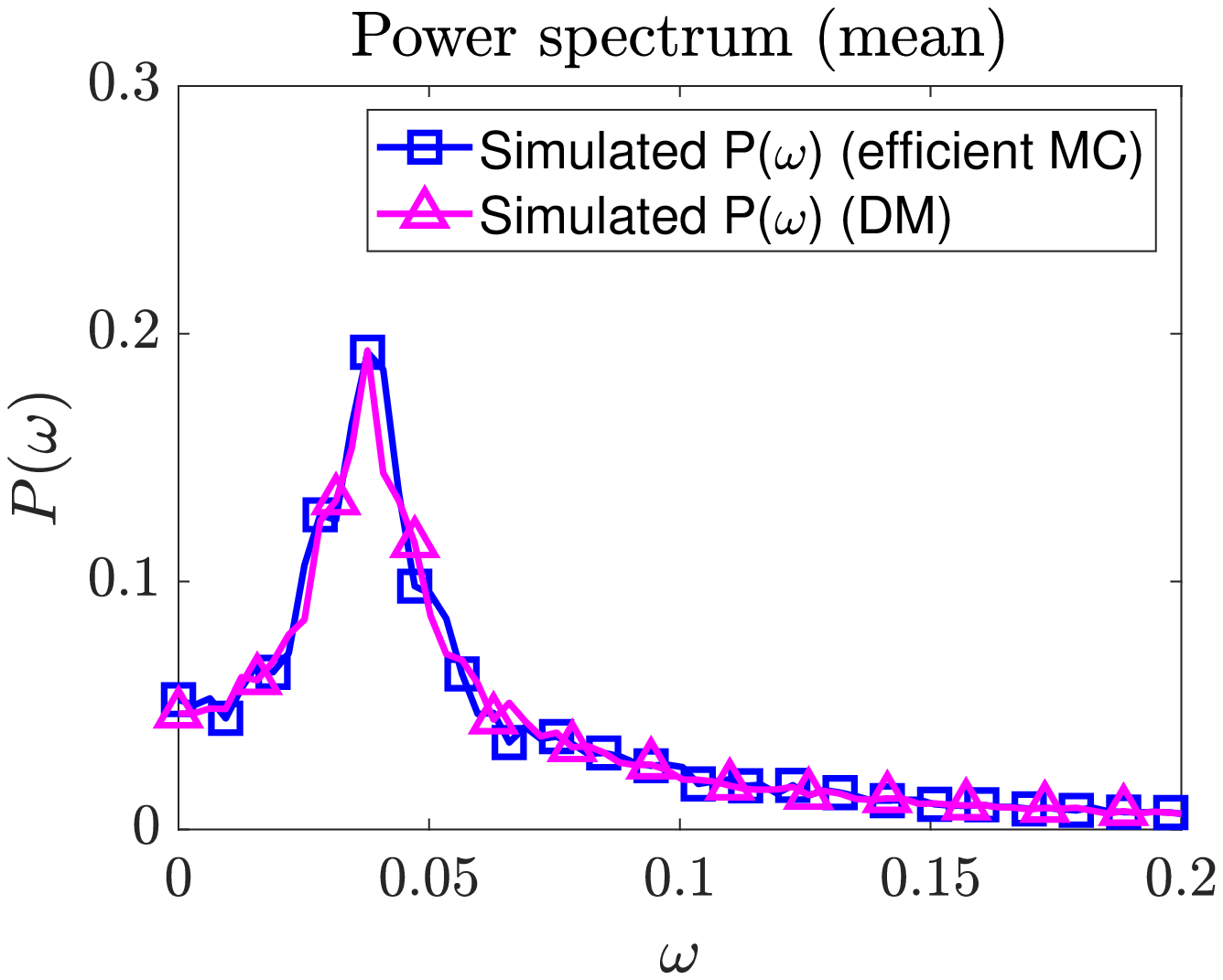}
			\caption{Heterogeneous case: power spectrum. On the left, power spectrum computed averaging 500 simulations of the processes described in \eqref{eq:5a}-\eqref{eq:5b}-\eqref{eq:5c} in cell $C_\ell$, $\ell=25,50$ obtained with the efficient Monte Carlo algorithm. On the right the average between the power spectra computed in cell $C_\ell$, $\ell=1,\ldots,M_c$ obtained with the efficient Monte Carlo algorithm (in blue) and with the direct method (in magenta).  Markers have been added just to indicate different lines.  }
			\label{fig:spectrumNH}
	\end{figure}
	
	%In Figure	\ref{fig:spectrum} on the left the power spectrum computed as in equation \eqref{eq:spectrum} and the one obtained averaging the results of 500 simulations of the processes described in \eqref{eq:1a}-\eqref{eq:1b} for the homogeneous case. 
	%In Figure \ref{fig:spectrum} on the right the power spectrum computed averaging the results of 500 simulations of the processes described in \eqref{eq:5a}-\eqref{eq:5b}-\eqref{eq:5c} for the heterogeneous case. Note that also in the heterogeneous case the power spectrum recalls the one of a damped harmonic oscillator.  
	%\begin{figure}[H]
	%	\centering
	%	\includegraphics[width=0.49\linewidth]{figure/spectrumH_predators}
	%		\includegraphics[width=0.49\linewidth]{figure/spectrumNH_predators}
	%	\caption{Power spectrum. On the left, homogeneous case: power spectrum computed as in equation \eqref{eq:spectrum} in red and averaging 500 simulations of the processes described in \eqref{eq:1a}-\eqref{eq:1b} in blue. On the right, heterogeneous case:power spectrum computed averaging 500 simulations of the processes described in \eqref{eq:5a}-\eqref{eq:5b}-\eqref{eq:5c}.}
	%	\label{fig:spectrum}
	%\end{figure}
	\section{Conclusion} Starting from the microscopic level it is possible to derive the master equation that describes the time evolution of the probability of being in a certain state and the mean-field equations that provide an average description of the behavior of the system which is valid only when the sample size $N$ is arbitrary large. Stochastic models instead describe the individuals behavior of interacting species in a more realistic context. However, classic stochastic algorithms have high computational costs, especially when the sample size increases.  In this paper, we have presented a valid efficient stochastic algorithm that, compared to other algorithms, has lower computational costs,and equivalent, or even better performances. First we have proved its consistency providing a reformulation of the model and showing that the mean-field equations are still valid. Then with different numerical experiments we have proved that the efficient Monte Carlo algorithm produces solutions that oscillate around the mean-field ones and that the error between stochastic simulations and mean-field solutions as a function of $N$ is still proportional to $1/\sqrt{N}$. We have also shown which is the main advantage of this new method proving that its computational cost is lower than the one of classical algorithms for a huge set of parameters but its accuracy is preserved.
	In the end, we have compared the oscillatory behavior of the stochastic and mean-field solutions underling the differences in their oscillatory behavior.
		Due to the flexible formulation of the ensemble stochastic algorithm it can be used not only to simulate the microscopic dynamics of interacting biological entities, but further generalized to other agent-based models, with different interaction dynamics and different macroscopic limiting models.

	\appendix
	\section{Spatial homogeneous predator-prey model} \label{app:homo}
	We report here the agent-based  model associated to spatial homogeneous dynamics.
	\paragraph{Agent-based dynamics}
	At each instant of time we select one individual and with probability $\mu\in[0,1]$ we assume that an interaction occurs with another individual, randomly chosen among the remaining $N-1$, according to the following rules
		\begin{equation} \label{eq:1a} 
			\begin{split} 
				&s_Bs_E \xrightarrow{b^r} s_Bs_B,\quad
				s_As_B \xrightarrow{p^r_1} s_As_A, \quad s_As_B \xrightarrow{p^r_2} s_As_E,
		\end{split}\end{equation}
		where $b^r,p^r_1,p^r_2$ are birth, and competition rates, and
		when interaction occurs among two empty components no change is accounted.
		Otherwise, with probability $1-\mu$, we assume that the sampled individual changes according to death events with rates $d^r_1,d^r_2$, as follows
		\begin{equation} \label{eq:1b} 
			\begin{split} 
				s_A \xrightarrow{d^r_1 } s_E,\quad s_B \xrightarrow{d^r_2} s_E.
		\end{split}\end{equation} 
		We suppose that if the selected component is empty then no changes in the populations sizes happen. We assume that at most one birth/competition and one death event can occur at each time step.
	%	After updating the sample, we repeat the same processes until respectively $\floor {\mu N}$ and $N-\floor{\mu N}$ components have been extracted, for $\mu \in (0,1)$.
	
	We focus on the evolution of the total number of predators and preys considering the transition rates from state $\xx=(A,B,E)$ to the states $\xx+\vv_j$, where $\vv_j$ denotes the $j$-th row of the stoichiometry matrix 
	\begin{equation}\label{eq:stoichH}
		V =\begin{bmatrix}
			0 &  1 & -1 \cr
			1 & -1 & 0 \cr
			0 & -1 & 1 \cr
			-1 & 0 & 1 \cr
			0 & -1 & 1
		\end{bmatrix},
	\end{equation}
	whose columns represent predators, preys and empty components respectively and whose rows represent the changes in the populations sizes due to the birth, competition and death events described in \eqref{eq:1a}-\eqref{eq:1b}.
		For such predator-prey dynamics the transition rates write as follows
		\begin{equation}\label{eq:2}  
			\begin{split}
				&\pi_{\vv_1}(\xx) = 2 \mu b^r \frac{B}{N} \frac{E}{N-1}, \quad \pi_{\vv_2}(\xx)  =  2 \mu p_1^r\frac{A}{N} \frac{B}{N-1},\quad \pi_{\vv_3}(\xx)  = 2\mu p_2^r\frac{A}{N} \frac{B}{N-1},\cr& \pi_{\vv_4}(\xx) = (1-\mu) d_1^r\frac{A}{N},\quad \pi_{\vv_5}(\xx) = (1-\mu) d_2^r\frac{B}{N},	
			\end{split} 
		\end{equation}
		where we define the operators $\pi_{\vv_j}(\xx)=\pi({\bf{x}}+{\bf v}_j|{\bf{x}})$ for any $j = 1,\ldots, M$. Here, $M=5$ denotes the total number of events described in \eqref{eq:1a}-\eqref{eq:1b}.
	In Figure \ref{fig:scheme} we depicted for $A$ predators, $B$ preys and $E$ empty states all possible states $\xx+\vv_j$ after one step of process \eqref{eq:1a}-\eqref{eq:1b}. 
	Neighbor states $\xx+\vv_j$ are connected to state $\xx$ either with outgoing arrows (red) for loss events, or with incoming arrows (blue) for gain events.
	\begin{figure}
		\centering
		\begin{tikzpicture}
			[
			mybox/.style={rectangle, draw=black!60, thick, minimum width=25mm, minimum height=10mm, rounded corners},
			]
			%Nodes
			\node[mybox]      (X)                              {$(A,B,E)$};
			\node[mybox]        (R)       [right=of X] {$(A,B+1,E)$};
			\node[mybox]     (L)       [left=of X] {$(A,B-1,E+1)$};
			\node[mybox]         (B)       [below=of X] {$(A+1,B,E-1)$};
			\node[mybox]          (A)       [above=of X] {$(A-1,B,E+1)$};
			\node[mybox]         (BL)       [below=of L] {$(A+1,B-1,E)$};
			\node[mybox]          (RU)       [above=of R] {$(A-1,B+1,E)$};
			%		%Lines
			\draw[>=latex,->,red,thick]                ([yshift= 5pt] X.east) -- ([yshift= 5pt] R.west);
			\draw[>=latex,->,blue,thick]                ([yshift= -5pt] R.west) -- ([yshift= -5pt] X.east);
			
			\draw[<-, blue, thick] (X.north east) --(RU.south west);
			\draw[->, red, thick] (X.north) --(A.south);
			\draw[<-, blue, thick] (X.south) --(B.north);
			
			\draw[>=latex,->,blue,thick]                ([yshift= 5pt] L.east) -- ([yshift= 5pt] X.west);
			\draw[>=latex,->,red,thick]                ([yshift= -5pt] X.west) -- ([yshift= -5pt] L.east);
			\draw[->, red, thick] (X.south west) --(BL.north east);
		\end{tikzpicture}
		\caption{Neighbor states are connected to state $(A,B,E)$ either with outgoing arrows (red) for loss events, or with incoming arrows (blue) for gain events. }\label{fig:scheme}
	\end{figure}
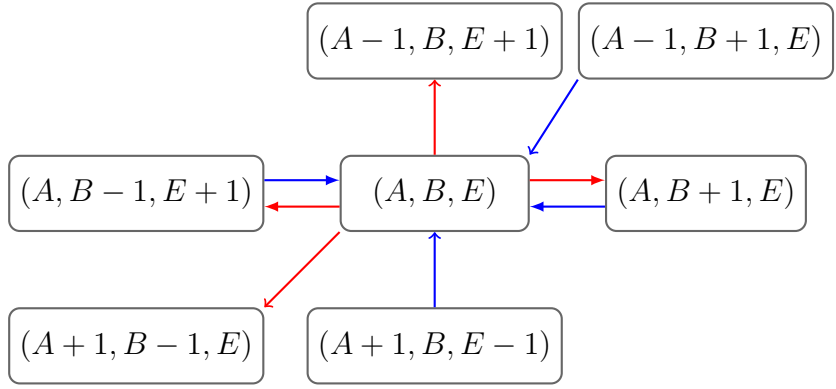
	Given the probability $P(\xx,t)$ to be in the state $\xx$ with $A$ predators, $B$ preys and $E$ empty spaces at time $t$, its time evolution is governed by the associated master equation as follows
	%	\begin{equation} \label{eq:3}
	%	\frac{d P({\bf x},t)}{dt} =\sum_{j=1}^M \Biggl [ P({\bf x}-{\bf v}_j,t) \pi(\xx|{\bf x}-{\bf v}_j)-P({\bf x},t)   \pi({\bf x}+{\bf v}_j|\xx)\Biggr ].
	%	\end{equation}	
	%		\textcolor{blue}{We introduce the master equation for the spatial homogeneous stochastic process in the following form}
	\begin{equation}\label{eq:master}
		\frac{d P({\bf x},t)}{dt} =\sum_{j=1}^M \Biggl[ (\mathcal E^{\vv_j}-1) \left( P({\bf x},t)   \pi({\bf x}+{\bf v}_j|\xx)\right) \Biggr ],
	\end{equation}
	where the step operator $\mathcal E^{\vv}$ acts on an arbitrary function $f$ according to 
	\begin{equation}\label{eq:step_operator}
		\mathcal E^{\vv_j}\left( f(\xx)\right) = f(\xx-\vv_j).
	\end{equation}

		\paragraph{Mean-field type approximations}
		In analogy to Section \ref{section2}, the spatial homogeneous mean-field approximation \eqref{eq:6} can be derived from the master equation \eqref{eq:master} introducing the empirical densities $f^N,g^N$ as in \eqref{eq:empirical}.
		A different approximation is obtained considering a Taylor expansion of \eqref{eq:step_operator} up to second order. In this framework, the master equation \eqref{eq:master} is approximated by the following Fokker-Planck equation 
		\begin{equation}\label{eq:fokker}
			\begin{aligned}
				&\frac{d P({\bf x},t)}{dt} =-\sum_{i=1}^3 \partial_{x_i}\Biggl(  \sum_{j=1}^{M-1} A_{ij}(\xx)P(\xx,t) \Biggr)+\frac{1}{2} \sum_{i=1}^3\partial_{x_ix_i}\Biggl(  \sum_{j=1}^{M-1} D_{ij}(\xx)P(\xx,t) \Biggr)+\\
				&\qquad\qquad\quad \sum_{i=1}^3 \sum_{k=1 \atop k \neq i}^3 \partial_{x_i x_k}\Biggl(  \sum_{j=1}^{M-1} C_{ikj}(\xx)P(\xx,t) \Biggr),
			\end{aligned}
			%\begin{aligned}
			%&\frac{d P({\bf x},t)}{dt} =\sum_{j=1}^M \sum_{i=1}^3 \Biggl[ \vv_j^i \partial_{x_i} (P(\xx,t) \pi(\xx+\vv_j|\xx))\Biggr.\cr
			%&\qquad\qquad\qquad \Biggl.+\frac{1}{2} \sum_{k=1}^3\Biggl( \vv_j^i \vv_j^k \partial_{x_i}\partial_{x_k} (P(\xx,t) \pi(\xx+\vv_j|\xx)) \Biggr) \Biggr],
			%\end{aligned}
		\end{equation}
		where \begin{equation}\label{eq:A8}
			\begin{split}
				&A_{ij} (\xx) =\vv_j^i \Pi_{\vv_j}(\xx), \quad D_{ij}(\xx)= (\vv_j^i)^2 \Pi_{\vv_j}(\xx),\quad C_{ikj} (\xx) =  \vv_j^i\vv_j^k \Pi_{\vv_j}(\xx).
			\end{split}
		\end{equation} Here $\vv_j^i$ denotes the $(i,j)$ component of the stoichiometry matrix $V$ and 
		$\Pi_{\vv_l}(\xx) = \pi_{\vv_l}(\xx)$ for $l\in \{1,2,4\}$ and $\Pi_{\vv_3}(\xx) = \pi_{\vv_3}(\xx) + \pi_{\vv_5}(\xx)$.
		As stated in \cite{gillespie2000chemical}, \cite{toral2014stochastic}, \cite{gillespie1996multivariate}, equation \eqref{eq:fokker} is the Fokker-Planck equation associated to the Langevin equations
		\begin{equation}\label{eq:langevin_general}
			\frac{dx_i}{dt} = \sum_{j=1}^{M-1} \left( A_{ij}(\xx)+ \sqrt{ D_{ij}(\xx)} \xi_j\right),
		\end{equation}
		for $i=1,\ldots,3$, where  $\xi_j$, $j=1,\ldots,M-1$ are normally distributed random numbers with zero mean and variance equal to $1/\sqrt{N}$ . Hence, multiplying both sides of equation \eqref{eq:langevin_general} by $P(\xx,t)/N$ and summing over all the values of $x_i$, $i=1,\ldots,3$  we can write 
		\begin{equation}\label{eq:langevin}
			\begin{aligned}
				\frac{df^N}{d\tau} &= (\beta  g^N -\delta)f^N +\sqrt{2\tilde{p}_1^r f^N g^N} \xi_2 -\sqrt{\tilde{d}_1^r f^N} \xi_4,\cr
				\frac{dg^N}{d\tau} &= r g^N \left( 1-\frac{g^N}{K} \right) -\alpha f^N g^N+\sqrt{2\tilde{b}^r g^N (1-f^N-g^N)}\xi_1 \cr
				&\qquad\qquad\qquad-\sqrt{2\tilde{p}_2^r f^N g^N +\tilde{d}_2^r} \xi_3 -\sqrt{2\tilde{p}_1^r f^N g^N} \xi_2,
			\end{aligned}
		\end{equation}
		where we have substituted the value $A_{ij}$ and $D_{ij}$ defined in \eqref{eq:A8}. Time is scaled according to $\tau = t/N$, and parameters are defined as $\beta = 2\mu p^r_1$, $\delta =(1-\mu) d^r_1$, $\alpha = 2\mu(p^r_1+p^r_2+b^r)$,$r=2\mu {b^r}-(1-\mu){d^r_2}$, and $K=1-\frac{(1-\mu){d^r_2}}{2\mu{b^r}}$. 
		Note that in the limit $N\to \infty$ the Langevin equations in \eqref{eq:langevin} collapse to the spatial homogeneous mean-field equations \eqref{eq:6}.

		\section{Monte Carlo algorithms for agent-based dynamics}\label{sec:algH}
		We report some of the classic stochastic algorithms for agent-based models with predator-prey interaction without spatial interaction. We refer to \cite{gillespie1976general, gillespie2007stochastic, chen2016non} for classical approaches, and to \cite{marchetti2017simulation} for a comprehensive collection on stochastic algorithms for agent-based dynamics. 	
		\paragraph{Direct method} 
		The main idea of the direct method is to estimate which is the next firing event $j$ and at which time $\tau$ it will occur from the probability density function
			\begin{equation}\label{eq:pdf}
				p(\tau,j|\xx,t) = a_{j}(\xx) e^{-a_0(\xx)\tau}.
			\end{equation}
			Here $a_{j}(\xx) = N \pi(\xx+\vv_{j}|\xx)$ 
			represents the propensity of the event $j$ and $a_0$ is the sum of all the propensities.
		At each time $t$ suppose to split the probability density function defined in equation \eqref{eq:pdf} into the product of two probability functions, one for the firing time $\tau$ and the other for the event $j$ that occurs at time $t+\tau$, as follows 
		\begin{equation}
			p(\tau,j|\xx,t) = p_1(\tau|\xx,t) p_2(j|\tau,\xx,t),
		\end{equation}
		where 
		\begin{equation}\label{eq:pdf1}
			p_1(\tau|\xx,t) = a_0(\xx) e^{-a_0(\xx) \tau}, \quad p_2(j|\tau,\xx,t) = \frac{a_{j}(\xx)}{a_0(\xx)}.
		\end{equation}
		Consider two uniformly distributed random numbers $r_1,r_2$ $\sim U(0,1)$, the index of the next firing event $j$ can be computed as the smallest index such that
		\begin{equation}\label{eq:DMindex} 
			\sum_{k=1}^{j} a_k(\xx) \geq r_1 a_0(\xx).
		\end{equation}
		The time in which the next event happens can be computed from the first equation in \eqref{eq:pdf1} as 
		\begin{equation}\label{eq:DMtime}
			\tau = \frac{1}{a_0(\xx)} \ln(r_2^{-1}).
		\end{equation}
		Algorithm \ref{alg_DM} outlines the details of the direct method. 
		\begin{alg}[Direct method]~ \label{alg_DM}
			\begin{enumerate}
				\item[\texttt 1.] Define the stoichiometry matrix $V$, the initial state $\xx$, the initial time $t=0$ and the final time $T$.
				\item[\texttt 2.] \texttt{while} $t<T$ 
				\begin{enumerate}
					\item Compute the propensities $ a_j$, $j=1,\ldots, M$.
					\item Consider two uniformly distributed random numbers $r_1,r_2$ $\sim U(0,1)$.
					\item Select the index of the next firing event as in equation \ref{eq:DMindex}.
					\item Compute the time in which the next event fires as in equation \ref{eq:DMtime}.
					\item Set $\xx=\xx+\vv_j$ where $\vv_j$ is the $j$-th row of the stoichiometry matrix $V$ defined as in equation \ref{eq:stoichH}.
					\item Set $t \leftarrow t+\tau$.
				\end{enumerate}
				\texttt{repeat}
			\end{enumerate}
		\end{alg}
		\paragraph{Classic Monte Carlo algorithm}
	The classic Monte Carlo method is an approximated stochastic algorithm. It allows multiple events to occur at the same time step that is supposed to be inversely proportional to the sample size $N$. Its computational costs increase as the sample size increases and can be comparable with the ones of exact algorithms. The main idea of the classic Monte Carlo algorithm is the following.
		Suppose to divide a priori the time interval considering a constant time step $\tau/N$. At each time $t$ select two components and assume that with probability $\mu \in [0,1]$ a birth or a competition event can occur. If the two selected components are
		\begin{itemize}
			\item a prey and an empty space: with probability $b_r \tau$ a new prey born and occupies the empty space;
			\item a predator and a prey: with probability $p_1^r \tau$ a new predator born and with probability $p_2^r \tau$ the prey dies and a new empty space is added to the system.
		\end{itemize}
		Extract another component and assume that with probability $1-\mu$ a death event can occur. If the selected component is occupied by a predator then with probability $d_1^r\tau$ the predator dies and if it is occupied by a prey then with probability $d_2^r\tau$ the prey dies. 
		Algorithm \ref{alg_MC} defines the details of the classic Monte Carlo algorithm. 
		\begin{alg}[Classic Monte Carlo algorithm]~ \label{alg_MC}
			\begin{enumerate}
				\item[\texttt 1.] Define the sample, the initial time $t=0$, the final time $T$, the time step as $\tau/N$ and a parameter $\mu \in [0,1]$.   
				\item[\texttt 2.] \texttt{while} $t<T$ 
				\begin{enumerate}
					\item Select two components inside the sample.
					\item Assume that with probability $\mu$ birth and competition events happen \[
					BE\xrightarrow{b}BB,\quad AB\xrightarrow{p_1} AA, \quad AB \xrightarrow{p_2}AE,
					\]
					where $b=b^r\tau$, $p_1^r\tau$ and $p_2^r\tau$.
					\item Select another component in the sample.
					\item Assume that with probability $1-\mu$ death events happen \[
					A\xrightarrow{d_1} E, \quad B \xrightarrow{d_2}E,
					\]
					with  $d_1=d_1^r\tau$ and $d_2=d_2^r\tau$.
					\item Update the sample.
					\item Set $t \leftarrow t+\tau/N$.
				\end{enumerate}
				\texttt{repeat}
			\end{enumerate}
		\end{alg}
		\paragraph{$\tau$-leaping method} The $\tau$-leaping method is another approximated stochastic algorithm. Its main idea is to assume that multiple events can happen in the same time interval that is suppose to vary in time. Therefore this algorithm, as the Monte Carlo and the direct method, has to cope with high computational costs. In the following we will describe more in details how the $\tau$-leaping algorithm works. Assume that more than one event can fire simultaneously at any time step. The probability that $k_j$ events fires in the time interval $(t,t+\tau)$ follows a Poisson distribution of parameter $a_j(\xx) \tau$ where $a_j(\xx)$ is the propensity of event $j$ and $\tau$ the time step. The number of firing events $k_j$ is generated by sampling its corresponding Poisson distribution, $Poi(a_j(\xx) \tau)$.  The state can be updated according to the following rule
		\begin{equation}\label{eq:leaping}
			\xx (t+\tau)= \xx (t)+ \sum_{j=1}^M k_j \vv_j = \xx + \sum_{j=1}^M Poi(a_j(\xx) \tau)\vv_j.
		\end{equation}
		The time step $\tau$ is updated according to a leap selection:
		\begin{itemize}
			\item choose an initial time step $\tau $;
			\item let $0<  \epsilon < 1$ and until $\Delta a_j(\xx) :=     \vert a_j(\xx(t+\tau))-a_j(\xx(\tau)) \vert >\epsilon$ set $\tau = \tau/2$. 
		\end{itemize}
		Algorithm \ref{alg_tau} outlines the details of the $\tau$-leaping method. 
		\begin{alg}[$\tau$ - leaping method]~ \label{alg_tau}
			\begin{enumerate}
				\item[\texttt 1.] Define the stoichiometry matrix $V$, the initial state $\xx$, the initial time $t=0$ and the final time $T$, the time step $\tau$.
				\item[\texttt 2.] \texttt{while} $t<T$ 
				\begin{enumerate}
					\item Compute the propensities $ a_j$, $j=1,\ldots, M$.
					\item Consider $M$ uniformly distributed random numbers $k_1,\ldots , k_M$ $\sim Poi(a_j(\xx)\tau)$.
					\item Update the state as in equation \eqref{eq:leaping}.
					\item Choose $\tau$ with leap selection.
					\item Set $t \leftarrow t+\tau$.
				\end{enumerate}
				\texttt{repeat}
			\end{enumerate}
		\end{alg}

		\section*{Acknowledgment} GA and FF would like to thank the Italian Ministry of Instruction, University and Research (MIUR) to support this research with funds coming from PRIN Project 2017 (No. 2017KKJP4X entitled “Innovative numerical methods for evolutionary partial differential equations and applications”).
	\bibliographystyle{abbrv}
	\bibliography{biblio}

\end{document}